\documentclass[UTF8,10pt]{article}
\usepackage[left=1.91cm,right=1.91cm,top=2.54cm,bottom=2.54cm]{geometry}
\usepackage{amsfonts}
\usepackage{amsmath}
\usepackage{amssymb}
\usepackage{pgfplots}
\usepackage{tikz}
\usetikzlibrary{arrows}
\usepackage{titlesec}
\usepackage{bm}
\usepackage{appendix}
\usepackage{tikz-cd}
\usepackage{mathtools}
\usepackage{amsthm}
\usepackage{setspace}
\usepackage{enumitem}
\setenumerate[1]{itemsep=0pt,partopsep=0pt,parsep=\parskip,topsep=5pt}
\setitemize[1]{itemsep=0pt,partopsep=0pt,parsep=\parskip,topsep=5pt}
\setdescription{itemsep=0pt,partopsep=0pt,parsep=\parskip,topsep=5pt}
\usepackage[colorlinks,linkcolor=black,anchorcolor=black,citecolor=black]{hyperref}
\usepackage[numbers]{natbib}
\usepackage{cases}
\usepackage{fancyhdr}
\usepackage[scaled=0.92]{helvet}	

\usepackage{txfonts}

\theoremstyle{definition}
\newtheorem{thm}{Theorem}[section]
\newtheorem{nota}[thm]{Notation}

\newtheorem{defn}{Definition}[section]
\newtheorem{lem}[thm]{Lemma}
\newtheorem{prop}[thm]{Proposition}

\newtheorem*{rmk}{Remark}

\newcommand{\R}{\mathbb{R}}  
\newcommand{\Z}{\mathbb{Z}}
\newcommand{\X}{\mathbb{X}}

\newcommand{\TP}{\overline{\partial}{}}
\newcommand{\TL}{\overline{\Delta}{}}
\newcommand{\tpl}{\overline{\partial}^2\overline{\Delta}}
\newcommand{\curl}{\text{curl }}
\newcommand{\dive}{\text{div }}

\newcommand{\q}{\quad}
\newcommand{\p}{\partial}
\newcommand{\dd}{\mathfrak{D}}
\newcommand{\DD}{\mathcal{D}}

\newcommand{\den}{{\bar\rho}_{0}}

\newcommand{\nab}{\nabla}
\newcommand{\ak}{\tilde{a}}
\newcommand{\ek}{\tilde{\eta}}
\newcommand{\ar}{\mathring{a}}
\newcommand{\ark}{\mathring{\tilde{{a}}}}
\newcommand{\Jrk}{\mathring{\tilde{{J}}}}
\newcommand{\er}{\mathring{\eta}}
\newcommand{\hr}{\mathring{h}}
\newcommand{\erk}{\mathring{\tilde{{\eta}}}}
\newcommand{\vr}{\mathring{v}}
\newcommand{\Jr}{\mathring{J}}
\newcommand{\psir}{\mathring{\psi}}
\newcommand{\park}{\nabla_{\mathring{\tilde{a}}}}
\newcommand{\divr}{\text{div}_{\mathring{\tilde{a}}}}

\newcommand{\pak}{\nabla_{\tilde{a}}}
\newcommand{\diva}{\text{div}_{\tilde{a}}}
\newcommand{\curla}{\text{curl}_{\tilde{a}}}
\newcommand{\lkk}{\Lambda_{\kk}}

\newcommand{\lap}{\Delta}

\newcommand{\di}{\text{div}\,}

\newcommand{\cp}{\overline{\partial}{}}

\newcommand{\dy}{\,dy}
\newcommand{\dz}{\,dz}
\newcommand{\dt}{\,dt}
\newcommand{\dS}{\,dS}

\newcommand{\kk}{\kappa}

\newcommand{\EE}{\mathcal{E}}

\newcommand{\VV}{\mathbf{V}}
\newcommand{\HH}{\mathbf{H}}
\newcommand{\GG}{\mathbf{G}}
\newcommand{\FF}{\mathbf{F}}
\newcommand{\VVr}{\mathring{\mathbf{V}}}
\newcommand{\HHr}{\mathring{\mathbf{H}}}

\newcommand{\PP}{\mathcal{P}}

\newcommand{\io}{\int_{\Omega}}

\newcommand{\ig}{\int_{\Gamma}}
\numberwithin{equation}{section}
\newcommand{\hc}{\mathcal{H}}
\newcommand{\dH}{\dot{H}}
\usepackage{xcolor}

\begin{document}
\bibliographystyle{plain}
\title{\bf Local Well-posedness for the Motion of a Compressible Gravity Water Wave with Vorticity}
\date{}
\author{Chenyun Luo\thanks{Department of Mathematics, The Chinese University of Hong Kong, Shatin, NT, Hong Kong.
Email: \texttt{cluo@math.cuhk.edu.hk}}\,\, and Junyan Zhang \thanks{Department of Mathematics, Johns Hopkins University, Baltimore, MD, USA.
Email: \texttt{zhang.junyan@jhu.edu}} }
\maketitle
\begin{abstract}
In this paper we prove the local well-posedness (LWP) for the 3D compressible Euler equations describing the motion of a liquid in an unbounded initial domain with moving boundary. The liquid is under the influence of gravity but without surface tension, and it is not assumed to be irrotational. We apply the tangential smoothing method introduced in Coutand-Shkoller \cite{coutand2007LWP, coutand2010LWP} to construct the approximation system with energy estimates uniform in the smooth parameter. It should be emphasized that, when doing the nonlinear a priori estimates, we need neither the higher order wave equation of the pressure and delicate elliptic estimates, nor the higher regularity on the flow-map or initial vorticity. Instead, we adapt the Alinhac's good unknowns to the estimates of full spatial derivatives. 
\end{abstract}


\setcounter{tocdepth}{1}
\tableofcontents

\section{Introduction}

In this paper we study the motion of a compressible gravity water wave in $\mathbb{R}^3$ described by the compressible Euler equations:
\begin{equation}\label{wweq}
\begin{cases}
\rho D_tu = -\nab p - \rho ge_3,&~~~ \text{in}\,\DD\\
D_t\rho+\rho\dive u=0 &~~~ \text{in}\,\DD\\
p=p(\rho) &~~~ \text{in}\, \DD
\end{cases}
\end{equation} where $\DD=\bigcup_{0\leq t\leq T}\{t\}\times \DD_t$ with $\DD_t:=\{(x_1, x_2,x_3)\in\R^3: x_3\leq S(t,x_1,x_2)\}$ representing the unbounded domain occupied by the fluid at each fixed time $t$, whose boundary $\p\DD_t=\{(x_1,x_2, x_3):x_3=S(t,x_1, x_2)\}$ moves with the velocity of the fluid. $\nabla:=(\p_{x_1},\p_{x_2},\p_{x_3})$ is the standard spatial derivative and $\dive X:=\nabla\cdot X$ is the divergence for any vector field $X$ in $\DD_t$. $D_t:=\p_t+u\cdot\nabla$ denotes the material derivative. In \eqref{wweq}, $u,\rho,p$ represent the fluid velocity, density and pressure, respectively, and $g>0$ is the gravity constant. The third equation of \eqref{wweq} is known to be the equation of states which satisfies 
\begin{align}
p'(\rho)> 0, \q \text{for}\,\, \rho\geq \bar{\rho}_0,
\label{EoS}
\end{align}
where $\bar{\rho}_0:=\rho |_{\p\DD}$ is a positive constant (we set $\den=1$ for simplicity), which is in the case of an isentropic liquid. The equation of states is required in order to close the system of compressible Euler equations.

The initial and boundary conditions of the system \eqref{wweq} are
\begin{align}
\DD_0=\{x:(0,x)\in \DD\},\q\text{and}\,\,
u=u_0,\rho=\rho_0~~~ \text{on}\,\{0\}\times\DD_0,\label{IC}\\
D_t|_{\p\DD}\in T(\p\DD)\q\text{and}\,\,
p|_{\p\DD}=0, \label{BC}
\end{align}
where $T (\p\DD)$ stands for the tangent bundle of $\p\DD$. The first condition in \eqref{BC} means the boundary moves with the velocity of the fluid, and the second one shows that outside the fluid region $\DD_t$ is the vacuum and the surface tension is neglected. 

\paragraph*{Energy conservation law} The system of equations \eqref{wweq}-\eqref{BC} admits a conserved energy. Let
\begin{align}
E_0(t)= \frac{1}{2}\int_{\DD_t} \rho |u|^2\,dx+\int_{\DD_t} \rho Q(\rho)\,dx+\int_{\DD_t\cap\{x_3>0\}} gx_3\,dx - \int_{\DD_t^c\cap\{x_3<0\}} gx_3 \,dx+\int_{\DD_t} g(\rho-1)x_3\,dx
\end{align}
where $Q(\rho)= \int_1^\rho p(\lambda)\lambda^{-2}\,d\lambda$ and $\DD_t^c$ denotes the complement of $\DD_t$.  A direct computation (cf. \cite[Section 1.1]{luo2018ww}) yields that $\frac{d}{dt}E_{0}(t)=0$. 


\paragraph*{Enthalpy formulation and Rayleigh-Taylor sign condition}
We introduce the new variable $h=h(\rho):=\int_1^\rho p'(\lambda)\lambda^{-1}\,d\lambda$, which is known to be the enthalpy of the fluid. It can be seen that $h'(\rho)> 0$ and $h|_{\p\DD}=0$ thanks to \eqref{EoS}. Since $\rho$ can then be thought as a function of $h$, we define $e(h):=\log \rho(h)$. Under these new variables, \eqref{wweq} and \eqref{IC}-\eqref{BC} becomes
\begin{align}
\begin{cases}
D_t u = -\nab h -ge_3,&~~~ \text{in}\,\DD,\\
\di u = -D_t e(h),&~~~ \text{in}\,\DD,\\
\DD_0=\{x:(0,x)\in \DD\},\\
u=u_0,h=h_0 &~~~ \text{on}\,\{0\}\times\DD_0,\\
D_t|_{\p\DD}\in T(\p\DD)\q\text{and}\,\,
h=0 &~~~\text{on}\,\p\DD. 
\end{cases} \label{ww}
\end{align}
The system \eqref{ww} looks exactly like the incompressible Euler equations, where $h$ takes the position of $p$ but $\di u$ is no longer $0$ and determined as a function of $\rho$ (and hence $h$). In addition, in Ebin \cite{ebin1987equations}, the free-boundary problem \eqref{ww} is known to be ill-posed unless the physical sign condition (also known as the Rayleigh-Taylor sign condition)
\begin{align}
-\nab_\mathcal{N} h \geq c_0>0, \q\text{on}\,\,\p\DD_t.
\label{Taylor}
\end{align}
holds. Here, $\mathcal{N}$ is the outward unit normal of $\p\DD_t$ and $\nab_\mathcal{N}:=\mathcal{N}\cdot \nab$. The condition \eqref{Taylor} is a natural physical condition which says that the enthalpy and hence the pressure and density is larger in the interior than on the boundary. We remark here that \eqref{Taylor} can be derived by the strong maximum principle if the water wave is assumed to be irrotational \cite{wu1997LWPww, wu1999LWPww, luo2018ww}, and the existence of the positive constant $c_0$ is a consequence of the presence of the gravity. Otherwise, we have merely that $ -\nab_\mathcal{N} h >0$, which is insufficient to close the a priori energy estimate for \eqref{ww}. 

\paragraph*{Equation of state for an isentropic liquid}
We would like to impose the following natural conditions on $e(h)$: For each fixed $k\geq 1$, there exists a constant $C>1$ such that
\begin{align}
C^{-1}\leq |e^{(k)}(h)|\leq C. \label{e(h) cond}
\end{align}
In fact, \eqref{e(h) cond} holds true if the equation of states is given by 
\begin{align}
p(\rho)= \gamma^{-1} (\rho^\gamma-1),\q \gamma\geq 1.
\end{align}
 In particular, when $\gamma=1$, a direct computation yields that $e(h)=h$.

\paragraph*{Compatibility conditions on initial data}
Finally, in order for the initial boundary value problem \eqref{ww}-\eqref{Taylor} to be solvable the initial data has to satisfy certain compatibility conditions at the boundary.  In particular $h$ verifies a wave equation by taking divergence to the first equation of \eqref{ww}:
\begin{align} \label{h wave pre}
\begin{aligned}
D_t^2 e(h) -\lap h = (\nab_\mu u^\nu)(\nab_\nu u^\mu),\quad &\text{in}\,\,\,\DD,\\
h=0,\quad&\text{on}\,\,\,\p\DD,\\
h|_{t=0} = h_0,\quad D_th|_{t=0}=h_1,\quad&\text{in } \{t=0\}\times \DD_0,
\end{aligned}
\end{align}
where $h_1$ can be determined in terms of $u_0$ and $h_0$ via the second equation of \eqref{ww}, i.e., $-e'(h_0)h_1 = \di u_0$. In above and throughout,  the summation convention is used for repeated upper and lower indices,  we adopt the convention that \textit{the Greek indices range over $1,2,3$, and the Latin indices range over $1$ and $2$.} The compatibility conditions must be satisfied in order for \eqref{h wave pre} to have a sufficiently regular solution. 

Since $h|_{\p\DD}=0$ and $D_t\in T(\p\DD)$, the second equation of \eqref{ww} implies that $\di u|_{\p\DD}=0$. We must therefore have $h_0|_{\p\DD_0}=0$ and $\di u_0|_{\p\DD_0}=0$, which is the zero-th compatibility condition. In general, for each $k\geq 0$, the $k$-th order compatibility condition reads
\begin{align}
D_t^j h|_{\{0\}\times\p\DD_0}=0,\q j=0,1,\cdots, k. 
\label{c'cond}
\end{align}
In \cite[Sect. 7]{luo2018ww}, we have proved that for each fixed $k\geq 0$, there exists initial data verifying the compatibility condition up to order $k$ such that the initial energy norm is bounded.  

\subsection{History and Background}

The study of the motion of a fluid has a long history in mathematics, and the study of the free-boundary problems has blossomed over the past two decades or so. However, much of this activity has focus on incompressible fluid models, i.e., the velocity vector field satisfies $\di u=0$ and the density $\rho$ is fixed to be a constant. Also, the pressure $p$ is not determined by the equation of states. Rather, it is a Lagrange multiplier enforcing the divergence free constraint. It is worth mentioning here that when the fluid domain is unbounded and the velocity $u_0$ is irrotational (i.e., $\curl u_0=0$, a condition that preserved by the evolution), this problem is called the (incompressible) water wave problem, which has received a great deal of attention. The local well-posedness (LWP) for the free-boundary incompressible Euler equations in either bounded or unbounded domains have been studied in \cite{alazard2014cauchy,masmoudi05, BMSW17, christodoulou2000motion, coutand2007LWP, coutand2010LWP, DKT, ifrim1, lannes2005ww, lindblad2002, lindblad2005well, lindblad2009priori, mz2009, nalimov1974cauchy, shatah2008geometry, shatah2008priori, shatah2011local, wzzz2015, wang2015good, wu1997LWPww, wu1999LWPww, yosihara82, zhangzhang08Euler}. In addition, the long time well-posedness for the  water wave problem with small initial data is available in \cite{alazard2015global, GMSgwp12, ifrim4, ifrim2, ionescu2015, wangxcww2016, wu2009GWPww, wu2011GWPww,  zhengww2019}, and there are recent results concerning the life-span for the water wave problem with vorticity \cite{dan2018ww, ifrim3, sqt2018}.

On the other hand, much less is known for the free-boundary compressible Euler equations, especially for the ones modeling a liquid, as opposed to a gas whose density can be zero on the moving boundary. The LWP for the free-boundary compressible gas model was obtained in \cite{coutand2010priori, coutand2012LWP, Hao2015gas, jang2009gas, jang2014gas, luozeng2014}, whereas for suitable initial data (e.g., data satisfying the compatibility condition), the LWP for the free-boundary compressible liquid model with \textit{a bounded} fluid domain is available in \cite{coutand2013LWP, disconzi2017prioriI, luo2019limit, GLL2019LWP, lindblad2003well, lindblad2005cwell, lindblad2018priori}.

When the fluid domain is unbounded, the free-boundary compressible Euler equations modeling a liquid is known to be the compressible water wave problem and little is known for this case. The only existence result is due to Trakhinin \cite{trakhiningas2009}, who proved the LWP for the compressible gravity water wave with vorticity using the Nash-Moser iteration (and thus with a loss of regularity). Recently, Luo \cite{luo2018ww} established the a priori energy estimates for the compressible gravity water wave with vorticity and proved the incompressible limit by adapting the the approach used in Lindblad-Luo \cite{lindblad2018priori} to an unbounded domain.

The goal of this paper is to prove the LWP for the motion of a compressible gravity water wave without the use of Nash-Moser iteration. The main idea is to approximate the nonlinear compressible water wave problem in the Lagrangian coordinates using a sequence of ``tangentially smoothed'' problems, whose solutions converge to that of the original problem when the smoothing coefficient goes to $0$. This in the incompressible free-boundary Euler equations goes back to Coutand-Shkoller \cite{coutand2007LWP}. Also, for its application in the compressible free-boundary Euler equations modeling a liquid in a bounded domain, Coutand-Hole-Shkoller \cite{coutand2013LWP} obtained the LWP for the case with surface tension and Ginsberg-Lindblad-Luo \cite{GLL2019LWP} obtained the LWP for the self-gravitating liquid. However, here we use a different set of approximate problems by adapting what appears in \cite{GLL2019LWP} which yields a simpler construction of the sequence of approximate solutions. This will be discussed in Sect. \ref{section 1.3}.

\subsubsection{Difference between a liquid and a gas}
This manuscript concerns the compressible Euler equations modeling the motion of \textit{a liquid}, which is treated very differently from \textit{a gas}, as what is studied in \cite{jang2009gas, jang2014gas} and \cite{luozeng2014}. 
The fundamental difference is that the energy for the compressible gas model is weighted by the sound speed $c_s(\rho):=\sqrt{p'(\rho)}$ which vanishes at the physical vacuum boundary (since $\rho$ vanishes there), and thus the estimates on the moving boundary are greatly simplified. In the case of a compressible liquid, however, we have to exploit the structure of the equations carefully in order to control the top order terms on the boundary even at the a priori estimate level. Also, it appears that the wave equation verified by the enthalpy $h$ plays a crucial role in the construction of a solution. We refer to Section \ref{section 1.3} for the detailed explanations.

\subsection{The Lagrangian coordinates}\label{lagrangian}
We introduce the Lagrangian coordinates, under which the moving domain becomes fixed. Let $\Omega:=\R^2\times(-\infty,0)$ to be the lower half space of $\R^3$. Denoting coordinates on $\Omega$ by $y=(y_1,y_2,y_3)$, we define $\eta:[0,T]\times\Omega\to \DD$ to be the flow map of $u$, i.e., 
\begin{align}
\p_t\eta(t,y) = u(t, \eta(t,y)),\q \eta(0,y)=\eta_0(y), 
\end{align}
where $\eta_0: \Omega\to\DD_0$ is a diffeomorphism, satisfying $\|\eta_0\|_{\hc}:= \|\p \eta_0\|_{L^\infty(\Omega)}+\|\p^2 \eta_0\|_{H^2(\Omega)}<\infty$ (See Notation \ref{norm} and the remark after \eqref{energy} for more details on the choice of this norm). For the sake of simplicity, we assume $\eta_0=$ Id, i.e., the initial domain is $\DD_0=\Omega=\R^2\times(-\infty,0)$. In fact, our approach is also applicable to the case for general data $\eta_0$. It is not hard to see that in $(t,y)$ coordinates $D_t$ becomes $\p_t$ and the boundary $\Gamma:=\p\Omega$ becomes fixed (i.e., $\Gamma=\R^2$). We introduce the Lagrangian velocity by $v(t,y):=u(t, \eta(t,y))$, and denote the Lagrangian enthalpy $h(t,y):=h(t,\eta(t,y))$ by a slight abuse of notations. 

Let $\p=\p_y$ be the spatial derivative in the Lagrangian coordinates. We introduce the matrix $a=(\p \eta)^{-1}$, specifically $a^{\mu\alpha}=a^{\mu}_{\alpha}:=\frac{\p y^{\mu}}{\p \eta^{\alpha}}$. This is well-defined since $\eta(t,\cdot)$ is almost an identity map whenever $t$ is sufficiently small. In terms of $v, h$ and $a$, \eqref{ww}-\eqref{Taylor} becomes
\begin{align}
\begin{cases}
\p_t v^\alpha = -\nab_a^\alpha h -ge_3,&~~~ \text{in}\,[0,T]\times\Omega,\\
\di_a v= -\p_t e(h),&~~~ \text{in}\,[0,T]\times\Omega,\\
\eta=\text{Id}, v=v_0,h=h_0 &~~~ \text{on}\,\{0\}\times\Omega,\\
\p_t|_{[0,T]\times \Gamma}\in T([0,T]\times \Gamma)\\
h=0 &~~~\text{on}\,\Gamma.
\end{cases} \label{wwl}
\end{align}
Here, the differential operator $\nab_a := (\nab_a^1, \nab_a^2, \nab_a^3)$ with $\nab_a^\alpha = a^{\mu\alpha}\p_\mu$ denotes the Eulerian (covariant) derivative and  $\di_a v =\nab_a\cdot v=a^{\mu\alpha}\p_\mu v_\alpha$  denotes the Eulerian divergence of $v$. In addition, since $\eta(0,\cdot)=\text{Id}$, we have $a(0,\cdot)=I$, where $I$ is the identity matrix, and $u_0$ and $v_0$ agree. Furthermore, let $J:=\det (\p\eta)$. Then $J$ satisfies
\begin{align}
\p_t J = Ja^{\mu\alpha} \p_\mu v_\alpha. 
\label{eq of J}
\end{align}
Finally, we assume the physical sign condition holds initially
\begin{align}
-\p_3 h_0\geq c_0>0
\label{Taylor 0}
\end{align}
and it can be shown that \eqref{Taylor 0} propagates to a later time. 

\subsection{The main result}\label{main result}
The goal of this paper is to prove the LWP of the compressible gravity water wave system in the Lagrangian coordinates. Specifically, we want to construct a solution to \eqref{wwl} with localized initial data $(v_0, h_0)$, i.e.,  $|v_0(y)|\to 0$ and $|h_0(y)|\to 0$ as $|y|\to\infty$ that satisfies the compatibility condition \eqref{c'cond} up to 4-th order as well as \eqref{Taylor 0}. The localized data is required so that the initial $L^2$-based higher order energy functional is bounded and the existence of such data can be found in \cite[Section 7]{luo2018ww}. Also, we remark here that \eqref{Taylor 0} remains hold thanks to the presence of the gravity (cf. \cite[Section 7]{luo2018ww}). 

\nota ($\hc$-norm) \label{norm} Let $f$ be a smooth function. We define
$$
\|f\|_{\hc}:= \|\p f\|_{L^\infty(\Omega)}+\|\p^2 f\|_{H^2(\Omega)}.
$$

\defn We define the higher order energy functional 
\begin{align}
\EE(t) := \|\eta(t)\|_{\hc}^2+\sum_{k=0}^4\|\p_t^{4-k}v(t)\|_{H^k(\Omega)}^2+\left(\|h(t)\|_{\hc}^2+\sum_{k=0}^3\|\p_t^{4-k}h(t)\|_{H^k(\Omega)}^2\right)+|a^{3\alpha}\TP^4\eta_{\alpha}(t)|_{L^2(\Gamma)}^2,
\label{energy}
\end{align}
where $\cp=(\p_1, \p_2)$ is the tangential Lagrangian spatial derivative.

\begin{rmk} The terms $\|\eta\|_{\hc}^2$ and $\|h\|_{\hc}^2$ may be replaced by $\|\eta\|_{H^4(\Omega)}^2$ and $\|h\|_{H^4(\Omega)}^2$, respectively, in the case when $\Omega$ is bounded. However, we have to be more careful in the case of an unbounded $\Omega$ since neither $\eta$ nor $\p \eta$ are in $L^2(\Omega)$, which is due to that $\eta_0=\text{Id}$. In addition to this, we cannot control $\p h$ in $L^2(\Omega)$ due to the presence of the gravity. Because of these, the lower-order terms $\p \eta$ and $\p h$ will be controlled in $L^\infty(\Omega)$ instead. 

\end{rmk}
\begin{thm}(Main theorem)\label{MAIN}
Suppose that the initial data $v_0, h_0$ satisfies 
\begin{itemize}
\item [A.] $\|v_0\|_{H^4(\Omega)}, \|h_0\|_{\hc} \leq M_0$,
\item  [B.] the compatibility condition \eqref{c'cond} up to 4-th order, and
\item [C.] the physical sign condition \eqref{Taylor 0}.
\end{itemize} Then there exists a $T_0>0$ and a unique solution $(\eta, v, h)$ to \eqref{wwl} on the time interval $[0, T_0]$ which satisfies
\begin{align}\label{EEEE}
\sup_{t\in[0,T_0]} \mathcal{E}(t) \leq \mathcal{C}(M_0),
\end{align}
where $\mathcal{C}(M_0)$ is a constant depends on $M_0$. Also, let $(\hat{v}_0, \hat{h}_0)$ be another set of initial data satisfying conditions A, B, C, and
$$
\|v_0-\hat{v}_0\|_{H^4(\Omega)}, \quad \|h_0-\hat{h}_0\|_{\hc} \leq \epsilon_0.
$$
Let $(\hat{\eta}, \hat{v},\hat{h})$ be the solution to \eqref{wwl} with initial data $(\hat{v}_0, \hat{h}_0)$. If
\begin{align*}
[\EE](t):=\|\eta(t)-\hat{\eta}(t)\|_2^2+\sum_{k=0}^2\|\p_t^{2-k}(v(t)-\hat{v}(t))\|_k^2+\|\p_t^{2-k}(h(t)-\hat{h}(t))\|_k^2+|a(\eta)^{3\alpha}\TP^2(\eta_{\alpha}(t)-\hat{\eta}_\alpha(t))|_0^2,
\end{align*}
 then 
 \begin{equation}
 \sup_{t\in[0,T_0]} \mathcal{[E]}(t) \leq \mathcal{C}(\epsilon_0).
 \label{EEEE2}
 \end{equation}

\end{thm}
\begin{rmk}
In the case with general initial data $\eta_0$, we require that $\|\eta_0\|_{\hc}\leq M_0$. 
\end{rmk}

\subsection{Novelty of this result, comparison with existing results, and application to other fluid models }\label{comparsion}
The result presented in this paper addresses the natural question left open in \cite{coutand2013LWP}, namely, the case of an unbounded domain. However, the method in \cite{coutand2013LWP} requires propagating an extra derivative of the flow map which requires an extra one derivative for the initial vorticity.  This is caused by differentiating the Euler equations in the Lagrangian coordinates and all derivatives fall on the cofactor matrix. We are able to avoid this by adapting the Alinhac's good unknowns, which, in turn, satisfy equations with a better structure. \textit{This is due to that these good unknowns tie to the covariant derivatives of the velocity and pressure in the Eulerian coordinates.} We refer to subsection \ref{sect. AGU} for the detailed analysis.

In addition, some lower order terms (e.g., $\p h$ and $\p \eta$) are no longer in $L^2(\Omega)$ as opposed to the case with a bounded fluid domain. As a result, our energy functional \eqref{energy}  has to be chosen carefully so that the aforementioned quantities are merely in $L^\infty$. Nevertheless, \eqref{energy} reduces to the following energy in the case of a bounded fluid domain:
$$
\|\eta(t)\|_{H^4(\Omega)}^2+\sum_{k=0}^4\left(\|\p_t^{4-k}v(t)\|_{H^k(\Omega)}^2+\|\p_t^{4-k}h(t)\|_{H^k(\Omega)}^2\right)+|a^{3\alpha}\TP^4\eta_{\alpha}(t)|_{L^2(\Gamma)}^2.
$$
In fact, our proof also works for the case of a bounded domain, producing a LWP result but \textit{without propagating the extra regularity of the vorticity and the flow map}.  Furthermore, we do not need to consider the surface tension to regularize the free surface and then take the zero surface tension limit.  

Finally, the method developed in this manuscript can be adapted to study the LWP for the non-isentropic compressible fluids, relativistic fluids, as well as the (inviscid) complex fluids with moving surface boundary. For a non-isentropic fluid,  the equation of states depends on both $\rho$ and the entropy $\mathfrak{s}$, i.e., $p=p(\rho, \mathfrak{s})$, where $\mathfrak{s}$ verifies $D_t \mathfrak{s} = 0$ in $\DD$. The LWP for the free-boundary problem in non-isentropic fluids is proved in \cite{trakhiningas2009}  with a loss of derivatives. Unlike the isentropic case, the enthalpy formulation \eqref{ww} is no longer available when $\mathfrak{s}$ is present. It is, however, possible to avoid the regularity loss by employing our method in the non-isentropic case by studying the new variable $\log (\rho/\bar{\rho}_0)$ instead of $h$.   Moreover, we have learned that Ginsberg-Lindblad \cite{dan2020} have employed a similar method to study a relativistic fluid with free-surface boundary. 
On the other hand, the complex fluids, e.g., magnetohydrodynamics (MHD) and elastodynamics, can be regarded as \textit{Euler equations under the influence of various external forces brought by other physical quantities (e.g., the elasticity, the Lorentz force in an  electromagnetic field)}. 
The presence of such external forces destroys the Cauchy invariance and thus we are \textit{unable} to propagate the extra regularity on the flow map. Nevertheless, the method developed in this manuscript has been adapted to treat the aforementioned complex fluid models by the second author \cite{LindbladZhang, Zhang1}. 

\paragraph*{List of Notations: }
\begin{itemize}
\item $\Omega:=\R^2\times(-\infty,0)$ and $\Gamma:=\R^2\times\{0\}$.
\item $\|\cdot\|_{s},\|\cdot\|_{\dot{H}^s}$:  We denote $\|f\|_{s}: = \|f(t,\cdot)\|_{H^s(\Omega)}$ and $\|f\|_{\dot{H}^s}=\|f(t,\cdot)\|_{\dot{H}^s(\Omega)}$ for any function $f(t,y)\text{ on }[0,T]\times\Omega$.
\item $|\cdot|_{s},|\cdot|_{\dot{H}^s}$:  We denote $|f|_{s}: = |f(t,\cdot)|_{H^s(\Gamma)}$ and $|f|_{\dot{H}^s}=|f(t,\cdot)|_{\dot{H}^s(\Gamma)}$ for any function $f(t,y)\text{ on }[0,T]\times\Gamma$.
\item $P(\cdots)$:  A generic polynomial in its arguments.
\item $\PP_0$:  $\PP_0=P(\|v_0\|_4,\|h_0\|_{\hc})$.
\item $[T,f]g:=T(fg)-T(f) g$, and $[T,f,g]:=T(fg)-T(f)g-fT(g)$, where $T$ denotes a differential operator or the mollifier and $f,g$ are arbitrary functions.
\item $\TP,\TL$: $\TP=\p_1,\p_2$ denotes the tangential derivative and $\TL:=\p_1^2+\p_2^2$ denotes the tangential Laplacian.
\item (Eulerian spatial derivative, divergence and curl) Let $f$ be a smooth function. Then $\nabla^{\alpha}_a f:=a^{\mu\alpha}\p_{\mu} f$, $\alpha=1,2,3$. 
Let $\mathbf{f}$ be a smooth vector field. Then $\dive_a\mathbf{f}:=a^{\mu\alpha}\p_\mu \mathbf{f}_{\alpha}$ and $(\curl_a\mathbf{f})_\lambda:=\epsilon_{\lambda\mu\alpha}a^{\nu\mu}\p_\nu \mathbf{f}^{\alpha}$, where $\epsilon_{\lambda\mu\alpha}$ is the sign of $(\lambda\mu\alpha)\in S_3.$
\end{itemize}

\paragraph*{Acknowledgment:} The authors would like to thank the referee for his/her careful reading and comments on improving this manuscript. 

\section{Strategy of the proof and some auxiliary results}

\subsection{An overview of the proof of Theorem \ref{MAIN}}\label{section 1.3}
The compressible water waves with vorticity are treated \textit{very differently from their incompressible and irrotational counterparts}, as one can completely reduce the latter to a system of quasilinear dispersive equations on the moving interface.  The strategy that we employed to prove Theorem \ref{MAIN} contains three parts: 
\begin{enumerate}
\item The a priori energy estimates in certain functional spaces. 

\item A suitable approximate problem which is asymptotically consistent with the a priori estimate. 

\item Construction of solutions to the approximate problem.
\end{enumerate}

 These steps are highly nontrivial in the case of a compressible water wave thanks to the nontrivial divergence of the velocity field and the unbounded fluid domain. The rest of this section is devoted to the elaboration of these steps. Also, we assume $e(h)=h$ in the rest of this section for the sake of simple exposition. But general $e(h)$ will be studied in the later sections.
 
 \nota  The following notations will be used throughout the rest of this manuscript. Let $f(t,y), g(t,y)$ be  smooth functions on $[0,T]\times\Omega$ and $[0,T]\times \Gamma$, respectively. Then we define $\|f\|_{s}: = \|f(t,\cdot)\|_{H^s(\Omega)}$  and $|g|_{s}: = |g(t,\cdot)|_{H^s(\Gamma)}$.  

\subsubsection{Construction of the approximate problem: Tangential smoothing}  

Although the a priori estimate has been established by the first author in \cite{luo2018ww}, it is still quite difficult to obtain a local-in-time solution by a direct iteration scheme based on the a priori bounds. The reason is that a loss of tangential derivative necessarily appears in the linearized system. Specifically, if we start the iteration with the trivial solution $(\eta^{(0)},v^{(0)},h^{(0)})=(\text{Id},0,0)$ and inductively define $(\eta^{(n+1)},v^{(n+1)},h^{(n+1)})$ by the following linearized system
\begin{equation}\label{wwlll}
\begin{cases}
\p_t\eta^{(n+1)}=v^{(n+1)}~~~&\text{ in }\Omega, \\
\p_t v^{(n+1)}=-\nabla_{{a^{(n)}}}h^{(n+1)}-ge_3~~~&\text{ in }\Omega, \\
\text{div}_{a^{(n)}}v^{(n+1)}=-\p_t h^{(n+1)}~~~&\text{ in }\Omega, \\
h^{(n+1)}=0~~~&\text{ on }\Gamma, \\
(\eta^{(n+1)},v^{(n+1)}, h^{(n+1)})|_{t=0}=(\text{Id},v_0, h_0),
\end{cases} 
\end{equation} where $a^{(n)}:=[\p\eta^{(n)}]^{-1}$, then:
\begin{enumerate}
\item We have to control $\|\p^4 v^{(n+1)}\|_0$  when constructing the solution for the linearized system, which requires the control of $\|\p^4 (\nabla_{{a^{(n)}}}h^{(n+1)})\|_0$. The elliptic estimate derived in \cite{GLL2019LWP} (i.e., Lemma \ref{GLL} with $f=h^{(n+1)}$ and $\ek$ replaced by $\eta^{(n)}$) yields
$$
\|\nab_{a^{(n)}} h^{(n+1)}\|_{\dot{H}^4} \lesssim \|\lap_{a^{(n)}} h^{(n+1)}\|_{3}+\|\cp\p \eta^{(n)}\|_{3}\|h^{(n+1)}\|_{\hc}
$$
However, the term $\|\TP\p \eta^{(n)}\|_3$ on the RHS cannot be controlled.

\item Also, we need a uniform-in-$n$ energy estimate for \eqref{wwlll} in order to get a solution for the original nonlinear problem by passing $n\to \infty$.  During this process, we pick up a boundary term that reads
\[
\ig \p_3 h^{(n+1)}\underbrace{\TP^4\eta_\beta^{(n)}}_{n\text{-th solution}} a^{(n)3\beta}a^{(n)3\alpha}\underbrace{\p_t\TP^4\eta^{(n+1)}_\alpha}_{(n+1)\text{-th solution}}\dS-\ig \p_3 h^{(n+1)}\TP^4\eta_\beta^{(n)} a^{(n)3\beta}a^{(n)3\alpha}\TP^4\eta_\gamma^{(n)}a^{(n)\mu\gamma}\p_\mu v^{(n+1)}\dS.
\]
It can be seen that the first term no longer contributes to the positive energy term $|a^{3\alpha}\TP^4\eta_{\alpha}|_{L^2(\Gamma)}^2$ as opposed to what happens to the original problem due to the loss of symmetry. In addition to this, a cancellation structure that is required to control the second term becomes unavailable as well. 
\end{enumerate}

In fact, the issues listed above appear also in the study of incompressible Euler equations \cite{coutand2007LWP,coutand2010LWP}. To overcome this difficulty, Coutand-Shkoller \cite{coutand2007LWP} introduced the tangential smoothing method: Let $\zeta=\zeta(y_1,y_2)\in C_c^{\infty}(\R^2)$ be the standard cut-off function such that $\text{Spt }\zeta=\overline{B(0,1)}\subseteq\R^2,~~0\leq\zeta\leq 1$ and $\int_{\R^2}\zeta=1$. The corresponding dilation is $$\zeta_{\kk}(y_1,y_2)=\frac{1}{\kk^2}\zeta\left(\frac{y_1}{\kk},\frac{y_2}{\kk}\right),~~\kk>0,$$ 
and we define the smoothing operator as
\begin{equation}\label{lkk0'}
\lkk f(y_1,y_2,y_3):=\int_{\R^2}\zeta_{\kk}(y_1-z_1,y_2-z_2)f(z_1,z_2,z_3)\dz_1\dz_2.
\end{equation}
Let $\ak:=[\p\ek]^{-1}$ be the smoothed version of $a$ with $\ek:=\Lambda_\kappa^2 \eta$ and define the approximate system by replacing the coefficient $a$ with $\ak$. Under this setting, we introduce the ``tangentially-smoothed" approximate system of the compressible water wave system \eqref{wwl} as follows 
\begin{align}
\begin{cases}
\p_t \eta = v+\psi &~~~\text{in } \Omega, \\
\p_t v^\alpha=-\pak^\alpha h -ge_3 &~~~\text{in } \Omega, \\
\diva v=-\p_t h &~~~\text{in } \Omega, \\
h=0 &~~~\text{on } \Gamma, \\
(\eta,v, h)|_{\{t=0\}}=(\text{Id},v_0, h_0),
\end{cases} \label{approx}
\end{align}
where $\psi=\psi(\eta,v)$ is a correction term which solves the half-space Laplace equation
\begin{equation}\label{psi'}
\begin{cases}
\Delta \psi=0  &~~~\text{in }\Omega, \\
\psi=\TL^{-1}\mathbb{P}\left(\TL\eta_{\beta}\ak^{i\beta}\TP_i\lkk^2 v-\TL\lkk^2\eta_{\beta}\ak^{i\beta}\TP_i v\right) &~~~\text{on }\Gamma,
\end{cases}
\end{equation}where $\TL:=\p_1^2+\p_2^2$ is the tangential Laplacian operator and $\TL^{-1}f:=(|\xi| ^{-2}\hat{f})^{\vee}$ is the inverse of $\TL$ on $\R^2$. The index $\beta$ ranges from 1 to 3 and $i$ ranges from 1 to 2, as stated after \eqref{h wave pre}. The notation $\mathbb{P}f:=P_{\geq 1} f$ denotes the standard Littlewood-Paley projection in $\R^2$ which removes the low-frequency part, i.e., $P_{\geq 1} f:=((1-\chi(\xi))\hat{f}(\xi))^{\vee},$ where $0\leq \chi(\xi)\leq 1$ is a $C_c^{\infty}(\R^d)$ cut-off function which is supported in $\{|\xi|\leq 2\}$ and equals to 1 in $\{|\xi|\leq 1\}$. Also, we mention here that the correction term $\psi\to 0$ as $\kk\to 0$.  
\begin{rmk} 
The Littlewood-Paley projection $\mathbb{P}$ is necessary when we apply the elliptic estimates to control $\psi$:
\[
|\psi|_{3.5}=|\TL^{-1}\mathbb{P}f|_{3.5}\lesssim |f|_{1.5},
\]otherwise the low-frequency part of $\TL^{-1}f$ loses control. 
\end{rmk}

 In Ginsberg-Lindblad-Luo \cite{GLL2019LWP}, the compressible Euler equations are approximated by a ``fully smoothed system", in the sense that all variables are replaced by their smoothed version. Specifically, they smooth the velocity vector field in the tangential direction first and then obtain the smoothed flow map by integrating it in time (see \cite{GLL2019LWP} Section 4).  In this paper, however,  we smooth the flow map directly and through this we replace the nonlinear coefficients $a^{\mu\alpha}$ by their smoothed version $\ak^{\mu\alpha}$ in \eqref{approx}. The advantage of our mollification and the correction term is three-fold.
\begin{itemize}
\item  The existence of the solution to the approximate system \eqref{approx} can be obtained by  passing to the limit as $n\to\infty$ in a sequence of approximate solutions $\{(\eta^{(n)}, v^{(n)}, h^{(n)})\}$ which are constructed by solving a linearized version of \eqref{approx} (see \eqref{approx n}). 

\item We do not need to construct the initial data for each linearized approximate system as what was done in \cite{GLL2019LWP}, because $a|_{t=0}=\ak|_{t=0}=I$ (the identity matrix) implies that the compatibility conditions of the (linearized) approximate problem are the same as the original system.

\item We are allowed to include a correction term $\psi$ in the first equation of \eqref{approx}, which was first introduced by Gu-Wang \cite{gu2016construction}. This is crucial in order to eliminate the higher order terms on the boundary when performing the tangential energy estimate, which shall be explained in Section \ref{sect psi1}-\ref{sect psi2}. 
\end{itemize}

\subsubsection{Avoiding extra regularity on the flow map: Alinhac's good unknown method} \label{sect. AGU}

The crucial part of the a priori estimates for the approximate system \eqref{approx} is the estimate for the tangential derivatives. In particular, the top order tangential energy with full spatial derivatives $\TP^4$ enters to the highest order. Due to the special structure of the correction term $\psi$ on the boundary, it is more convenient to replace $\TP^4$ by $\tpl$. The corresponding energy reads
\begin{align}
\underbrace{\|\tpl v\|_{L^2(\Omega)}^2+\|\tpl h\|_{L^2(\Omega)}^2}_{=\mathcal{E}_{TI}}+\underbrace{|\ak^{3\alpha}\tpl\lkk\eta_{\alpha}|_{L^2(\Gamma)}^2}_{=\mathcal{E}_{TB}}.
\end{align}
In the control of $\mathcal{E}_{TI}$, it is necessary to deal with the commutator between $\tpl$ and $\pak$, namely $[\tpl,\ak^{\mu\alpha}]\p_{\mu}f$ for $f=h$ or $v_\alpha$. Such commutators contain the higher order term $(\tpl a^{\mu\alpha})(\p_\mu f)$ with $\tpl \ak=\tpl\p\eta\times\p\eta+\cdots$ whose $L^2(\Omega)$-norm cannot be directly controlled. In Ginsberg-Lindblad-Luo \cite{GLL2019LWP}, such commutators are controlled by adding $\kk^2$-weighted higher order terms to the energy, which corresponds to the fifth order full spatial energy of the wave equation verified by $h$. In particular, the extra regularity of the flow map $\eta$ is necessary in \cite{GLL2019LWP} to close the energy of 5-th order wave equation of $h$. However, we can use the Alinhac's good unknowns method to get rid of the higher regularity requirement for the flow map.

\paragraph*{Motivation of Alinhac's good unknowns}
 The main idea is to rewrite $\tpl(\pak h)$ and $\tpl(\pak\cdot v)$ as
\begin{align}
\label{goodhh} \tpl(\pak h)=\pak\HH+C(h),\text{ with }\|\HH-\tpl h\|_0+\|\p_t(\HH-\tpl h)\|_0+\|C(h)\|_0\leq P(\EE(t)),\\
\label{goodvv} \tpl(\pak\cdot v)=\pak\cdot\VV+C(v),\text{ with }\|\VV-\tpl v\|_0+\|\p_t(\VV-\tpl v)\|_0+\|C(v)\|_0\leq P(\EE(t)),
\end{align}
where $P$ is a generic polynomial. 
Here $\HH,\VV$ are called the ``Alinhac's good unknowns" of $h$ and $v$, respectively. In other words, the Alinhac's good unknowns allow us to take into account the covariance under the change of coordinates to avoid the extra regularity assumption on the flow map.
\begin{rmk}
For Euler equations (including the water wave system), it is possible to have higher regularity for $\eta$ than $v$ thanks to the propagation of the extra regularity assumption on the vorticity. However, when modeling complex fluids with free boundary, such as magnetohydrodynamics (MHD) equations, MHD current-vortex sheets, and elastic fluid equations, it is not possible to have $\eta$ more regular than $v$. 
\end{rmk}

\paragraph*{Derivation of Alinhac's good unknowns}  It remains to derive the precise expressions of the good unknowns $\HH,\VV$, which is recorded in Lemma \ref{lem AGU property} in full details. Here we give a brief explanation on the precise forms of the good unknowns from the perspective of change of variables. In fact, by chain rule, we can rewrite $\tpl(\pak f)$ in terms of covariant derivatives via $\TP_i=\dfrac{\p}{\p y^i}=\dfrac{\p \ek^{\alpha}}{\p y^i}\dfrac{\p}{\p \ek^{\alpha}}=\TP_i \eta\cdot\pak$:
\[
\pak(\tpl f)=\sum_{i=1}^2\pak\left((\TP\eta\cdot\pak)(\TP\eta\cdot\pak)(\TP_i\eta\cdot\pak)(\TP_i\eta\cdot\pak)f\right)=\tpl(\pak f)+\pak(\tpl\ek\cdot\pak f)+\text{l.o.t.}
\] It is not difficult to find that, other than $\pak\tpl f$, there is another highest order term $-\pak(\tpl\ek\cdot\pak f)$ corresponding to the term that all the derivatives fall on $\ek$. Therefore, the essential highest order term in $\tpl(\pak f)$ is indeed the covariant derivative of $\tpl f-\tpl\ek\cdot\pak f$, called the ``Alinhac's good unknown" of $f$ with respect to $\tpl$.

Therefore, one has
$
\VV:=\tpl v-\tpl\ek\cdot\pak v,
$
and
$
\HH:=\tpl h-\tpl\ek\cdot\pak h
$, satisfying \eqref{goodhh}-\eqref{goodvv} and 
$$\|C(v)\|_0\lesssim P(\|\p^2\eta\|_2,\|\p\eta\|_{L^{\infty}})(\|\p v\|_{L^{\infty}}+\|\p^2 v\|_2),\quad \|C(h)\|_0\lesssim P(\|\p^2\eta\|_2,\|\p\eta\|_{L^{\infty}})(\|\p h\|_{L^{\infty}}+\|\p^2 h\|_2).$$ This circumvents a loss of regularity caused by differentiating the equation $\p_t v^\alpha=-\pak^\alpha h -ge_3$ and all derivative fall on $\ak$. Such a remarkable observation is due to Alinhac \cite{alinhacgood89}. In the study of free-surface fluid, it was first implicitly used in the $Q$-tensor energy introduced by Christodoulou-Lindblad \cite{christodoulou2000motion} which was later generalized by \cite{luo2018ww}. It has also been applied explicitly in \cite{gu2016construction, MRgood2017, wang2015good}. Recently, Ginsberg-Lindblad \cite{dan2020} have adapted these good unknowns to study the LWP for the free-boundary relativistic Euler equations in a fixed hyperbolic space-time domain.  

\paragraph*{Interior estimates via the good unknowns}Now $\VV$ and $\HH$ satisfy
\begin{align}
\p_t \VV = -\pak \HH +\text{error},\q \pak\cdot \VV= \tpl (\diva v)+\text{error},\q\text{in}\,\,\Omega \label{AGU}\\
\HH = -\tpl\ek_\beta \ak^{3\beta}\p_3 h\q\text{on}\,\,\Gamma,
\end{align}
multiplying $\VV$ through \eqref{AGU}, integrating over $\Omega$ and integrating $\pak$ in $-\int_\Omega \pak \HH \cdot \VV$ by parts yield 
\begin{align}
\frac{1}{2}\frac{d}{dt}\int_\Omega |\VV|^2=\io \HH\tpl (\diva v)\dy+\ig\p_3h\ak^{3\beta}\ak^{3\alpha}\tpl\ek_{\beta}\VV_{\alpha} dS +\text{error}.
\label{AGU control}
\end{align}
For the first term on the RHS of \eqref{AGU control}, invoking the definition of $\HH$ and the third equation of \eqref{approx}, we have
\begin{align*}
\io \HH\tpl (\diva v)\dy = -\io (\tpl h)(\tpl \p_t h)\dy + \io (\tpl \ek\cdot \nab_{\ak} h)(\tpl \p_t h)\dy.
\end{align*}
Here, the first term  is equal to $-\frac{d}{dt}\frac{1}{2}\|\tpl h\|_{L^2(\Omega)}^2$, which contributes to the positive energy term that controls $\|\tpl h\|_{L^2(\Omega)}^2$.  The second term can be controlled after considering its time integral (See \eqref{K*}-\eqref{tgK} for the details). 
In addition, invoking the definition of $\VV$, we have
$$
\|\TP^2 \TL v\|_{0}^2\leq \|\VV\|_{L^2(\Omega)}^2 + \|\TP^2\TL \ek\cdot \nab_{\ak} v\|_{0}^2,
$$ 
and this implies that it suffices to bound $\|\VV\|_{0}^2$ in order to control $\|\TP^4 v\|_{0}^2$ as the last term $\|\TP^2\TL \ek\cdot \nab_{\ak} v\|_{0}^2$ can be controlled straightforwardly. For details we refer to the proof of Lemma \ref{lem AGU controls 4 tangential}.

\subsubsection{Crucial cancellation structure on the boundary}\label{sect psi1}
The second term on the RHS of \eqref{AGU control} is equal to 
\begin{align} 
\ig \p_3 h\ak^{3\alpha}\ak^{3\beta}\tpl \ek_{\beta}(\tpl\p_t\eta_{\alpha}-\tpl\psi-\tpl\ek\cdot\pak v_{\alpha})\dS.
\label{AGU bdy}
\end{align}
By plugging the definition of $\VV$ and invoking the first equation of \eqref{approx} and then ``moving" one $\Lambda_\kk$ from $\ek_\beta$ to $\eta_\alpha$, we have
\begin{align}\label{1.26}
&~~~~\ig \p_3 h\ak^{3\alpha}\ak^{3\beta}\tpl \ek_{\beta}\left(\tpl\p_t\eta_{\alpha}-\tpl\ek\cdot\pak v_{\alpha}\right) \nonumber \\
&=\frac{1}{2}\frac{d}{dt}\int_\Gamma \p_3 h |\ak^{3\alpha}\tpl\Lambda_\kk \eta_\alpha|^2\dS\nonumber\\
&~~~~~+\ig \p_3 h \ak^{3\beta}\tpl\lkk\eta_{\beta}\ak^{3\gamma}\TP_i\lkk^2v_{\gamma}\ak^{i\alpha}\tpl\lkk\eta_{\alpha}\dS
-\ig \p_3 h\ak^{3\alpha}\ak^{3\beta}\tpl \ek_{\beta}\tpl\ek_{\gamma}\ak^{i\gamma} \TP_iv_{\alpha}\dS\nonumber\\
&~~~~~+\ig \p_3 h \ak^{3\beta}\tpl\lkk\eta_{\beta}\ak^{3\gamma}\TP_i\lkk^2\psi_{\gamma}\ak^{i\alpha}\tpl\lkk\eta_{\alpha}\dS+\text{error}.
\end{align}
The higher order terms on the third line are exactly cancelled out for the original problem (i.e., $\kk=0$) but we are unable to control them when $\kk>0$. However, in light of the definition of $\psi$ \eqref{psi'},  both of the higher order terms can indeed be cancelled by $-\ig \p_3 h\ak^{3\alpha}\ak^{3\beta}\tpl \ek_{\beta}\tpl\psi\dS$ in \eqref{AGU bdy} up to lower order terms plus the low-frequency term $$\TP^2\left((\text{Id}-\mathbb{P})\left(\TL\eta_{\beta}\ak^{i\beta}\TP_i\lkk^2 v-\TL\lkk^2\eta_{\beta}\ak^{i\beta}\TP_i v\right)\right),$$ which can be controlled by using Bernstein's inequality \eqref{bern1} in Lemma \ref{bernstein}.  The details can be found in Section \ref{bdrycancel}.

\begin{rmk}
Alternatively, one may boost interior regularity of the flow map to $H^{4.5}(\Omega)$. This can be done via the Cauchy invariance but one has to assume the initial vorticity $\curl v_0\in H^{3.5}(\Omega)$ (cf. \cite[Section 4.2]{KTV2016}). The correction term here helps us get rid of extra regularity of the flow map.
\end{rmk}

\subsubsection{Discussion on the uniform-in-$\kk$ energy estimate}\label{sect psi2}

We have to make sure that our energy estimate is uniform-in-$\kk$ in order to pass the sequence of approximate solutions to a limit as $\kk\rightarrow 0$, which, in fact, solves the original problem. This depends crucially on the aforementioned cancellation scheme on the boundary, as the terms in the second line of \eqref{1.26} would otherwise contribute to $\|\lkk \eta\|_{\dot{H}^{4}(\Gamma)}$ and $\|\ek\|_{\dot{H}^{4}(\Gamma)}=\|\lkk^2\eta\|_{\dot{H}^{4}(\Gamma)}$, respectively, which are of $0.5$-derivatives more regular than $v$ after moving to the interior. Of course, one may control these terms by ``moving $\TP^{0.5}$ to the tangential mollifier", i.e., 
\begin{align} \label{non kk uniform}
\|\lkk \eta \|_{\dot{H}^4(\Gamma)}\lesssim \kk^{-1/2}\|\eta\|_{\dot{H}^{3.5}(\Gamma)}.
\end{align}
But this fails to be uniform-in-$\kk$ when $\kk\rightarrow 0$.

Nevertheless, we have to use \eqref{non kk uniform} to treat the term in the third line of \eqref{1.26} but we can get an extra $\sqrt{\kk}$ owing to the structure of $\psi$ and this cancels $\kk^{-1/2}$ out.  Specifically, 
\begin{align*}
&\ig \p_3 h \ak^{3\beta}\tpl\lkk\eta_{\beta}\ak^{3\gamma}\TP_i\lkk^2\psi_{\gamma}\ak^{i\alpha}\tpl\lkk\eta_{\alpha}\dS\\
\lesssim &|\p_3 h \ak^{3\gamma}\ak^{i\alpha}|_{L^{\infty}(\Gamma)}|\TP\lkk^2\psi_{\gamma}|_{L^{\infty}(\Gamma)}|\lkk\tpl\eta|_{L^2(\Gamma)}|\ak^{3\beta}\tpl\lkk\eta_{\beta}|_{L^2(\Gamma)}\\
\lesssim &|\p_3 h \ak^{3\gamma}\ak^{i\alpha}|_{L^{\infty}(\Gamma)}|\TP\lkk^2\psi_{\gamma}|_{L^{\infty}(\Gamma)}|\ak^{3\beta}\tpl\lkk\eta_{\beta}|_{L^2(\Gamma)}\left(\kk^{-1/2}|\eta|_{\dot{H}^{3.5}(\Gamma)}\right).
\end{align*}
We employ the Sobolev embeddings $W^{1,4}(\R^2)\hookrightarrow L^{\infty}(\R^2)$ and $H^{0.5}(\R^2)\hookrightarrow L^{4}(\R^2)$, and the tangential smoothing property \eqref{lkk3} to have $|\TP\psi|_{L^{\infty}}\lesssim \sqrt{\kk} P(\|\p^2\eta\|_2,\|\p\eta\|_{L^{\infty}},\|v\|_3)$. We refer to \eqref{LB3 start}-\eqref{LB3} for the details. 

\subsubsection{Discussion on the existence of the approximate system}

The approximate system \eqref{approx} can be solved by an iteration of the approximate solutions. Specifically, let $(\eta^{(0)},v^{(0)},h^{(0)})=(\eta^{(1)},v^{(1)},h^{(1)})=(\text{Id},0,0)$ (i.e., the trivial solution). For each $n\geq 1$, we inductively define  $(\eta^{(n+1)},v^{(n+1)}, h^{(n+1)})$ to be the solution of the linearized system of equations
\begin{equation}
\begin{cases}
\p_t\eta^{(n+1)}=v^{(n+1)}+\psi^{(n)}~~~&\text{ in }\Omega, \\
\p_t v^{(n+1)}=-\nabla_{{\ak^{(n)}}}h^{(n+1)}-ge_3~~~&\text{ in }\Omega, \\
\text{div}_{\ak^{(n)}}v^{(n+1)}=-\p_t h^{(n+1)}~~~&\text{ in }\Omega, \\
h^{(n+1)}=0~~~&\text{ on }\Gamma, \\
(\eta^{(n+1)},v^{(n+1)}, h^{(n+1)})|_{t=0}=(\text{Id},v_0, h_0),
\end{cases} \label{approx n}
\end{equation} 
Here, $a^{(n)}:=[\p\eta^{(n)}]^{-1}$, $\ak^{(n)}:=[\p \ek^{(n)}]^{-1}$ and the correction term $\psi^{(n)}$ is determined by \eqref{psi'} with $(\eta^{(n)}, v^{(n)},\ak^{(n)})$.  The existence of $(\eta^{(n+1)},v^{(n+1)}, h^{(n+1)})$ follows from showing that the map $\Xi:\X\to \X$ (defined below) has a fixed point, where the Banach space $\X$ define as
\begin{align} \label{def X intro}
\begin{aligned}
\X= &\bigg\{(\xi, w,\pi):(w, \xi)|_{t=0}=(v_0,\text{Id}),\\
\sup_{t\in [0,T]}&\left(\|w(t),\p_t\pi(t)\|_{Z^4}+\|\nab_{\ak^{(n)}} \pi(t)\|_{L^\infty}+\|\p\nab_{\ak^{(n)}} \pi(t), \p_t\nab_{\ak^{(n)}} \pi(t)\|_{Z^3}+\|\p_t\xi(t)\|_{Z^3}+\|\p^2\xi(t)\|_{Z^2}+\|\p\xi(t)\|_{L^{\infty}}\right)\leq M\bigg\}.
\end{aligned}
\end{align}
Here, $Z^k$ denotes the mixed space-time $L^2$-Sobolev norm of order $\leq k$. 
The map $\Xi$ is given by $$\Xi:(\xi, w,\pi)\mapsto(\eta^{(n+1)}, v^{(n+1)}, h^{(n+1)})$$ where
we define $\eta^{(n+1)}, v^{(n+1)}$ and $h^{(n+1)}$, respectively, by 
\begin{align}
\p_t \eta^{(n+1)}=&w+\psi^{(n)},\eta^{(n+1)}(0)=\text{Id},\\
\p_t v^{(n+1)}=&-\nab_{\ak^{(n)}}\pi-ge_3, v^{(n+1)}(0)=v_0,
\label{eq v}
\end{align}
\begin{equation}
\text{and }
\begin{cases} \label{eq h}
\p_t^2 h^{(n+1)}-\Delta_{\ak^{(n)}}h^{(n+1)}=-\p_t \ak_{(n)}^{\nu\alpha}\p_\nu v_{\alpha}^{(n+1)}~~~&\text{ in }\Omega,\\
h^{(n+1)}=0~~~&\text{ on }\Gamma,
\end{cases} \text{ with } (h^{(n+1)},\p_t h^{(n+1)})|_{t=0}=(h_0,h_1).
\end{equation}
\begin{rmk}
The quantity $\|\nab_{\ak^{(n)}} \pi\|_{L^\infty}+\|\p\nab_{\ak^{(n)}} \pi, \p_t\nab_{\ak^{(n)}} \pi\|_{Z^3}$ can be replaced by $\|\nab_{\ak^{(n)}} \pi\|_{Z^4}$ if $\Omega$ is bounded. However, we can merely control $\nab_{\ak^{(n)}} \pi$ in $L^\infty$ as it corresponds to $\nab_{\ak^{(n)}} h^{(n+1)}$. 
\end{rmk}

The estimates for $\eta^{(n+1)}$ and $v^{(n+1)}$ are straightforward since they verify transport equations. However, the estimate for $\|\p\nab_{\ak^{(n)}} h^{(n+1)}\|_{Z^3}$ requires that of  $\|\nab_{\ak^{(n)}} h^{(n+1)}\|_{\dH^4}$ which cannot be done directly by commuting $\cp^4$ through the wave equation, since there is no hope to control the corresponding source term consists $\p_t \ak_{(n)}^{\nu\alpha}(\cp^4 \p_\nu v_{\alpha}^{(n+1)})$ in $L^2$. 

The key observation here is that $\cp^3\p_t \p_\nu v_{\alpha}^{(n+1)}$ can in fact be controlled thanks to \eqref{eq v} and the finiteness of $\|\p^4\nab_{\ak^{(n)}} \pi\|_{0}$, and so there is no problem to control the wave energies by commuting $\DD^3\p_t$ (where $\DD=\cp$ or $\p_t$) through the wave equation \eqref{eq h}. Now, the remaining  $\|\nab_{\ak^{(n)}} h^{(n+1)}\|_{\dH^4}$ can be treated using the elliptic estimate
\begin{align}
\|\nab_{\ak^{(n)}} h^{(n+1)} \|_{\dH^4} \lesssim  \|\lap_{\ak^{(n)}} h^{(n+1)}\|_{3}+\|\cp \p \ek^{(n)}\|_{3}\|h^{(n+1)}\|_{\hc}, 
\end{align}
which indicates that the control of $\|\nab_{\ak^{(n)}} h^{(n+1)}\|_{\dH^4}$ requires that of $\|\lap_{\ak^{(n)}} h^{(n+1)}\|_{3}$ up to the highest order. But this term is under control since \eqref{eq h} suggests that 
$$
\|\lap_{\ak^{(n)}} h^{(n+1)}\|_{3} \leq \|\p_t^2 h^{(n+1)}\|_{3}+\|\p_t \ak_{(n)}^{\nu\alpha}\p_\nu v_{\alpha}^{(n+1)}\|_{3}, 
$$
where the second term is of lower order and the first term can be controlled by invoking the wave energy with $2$ time derivatives. 

\subsection{Auxiliary results}
\subsubsection{Sobolev inequalities}
\begin{lem}\textbf{(Kato-Ponce {\cite{kato1988commutator} inequalities) }}  Let $J=(I-\Delta)^{1/2},~s\geq 0$. Then the following estimates hold:

(1) $\forall s\geq 0$, we have 
\begin{equation}\label{product}
\begin{aligned}
\|J^s(fg)\|_{L^2}&\lesssim \|f\|_{W^{s,p_1}}\|g\|_{L^{p_2}}+\|f\|_{L^{q_1}}\|g\|_{W^{s,q_2}},\\
\|\p^s(fg)\|_{L^2}&\lesssim \|f\|_{\dot{W}^{s,p_1}}\|g\|_{L^{p_2}}+\|f\|_{L^{q_1}}\|g\|_{\dot{W}^{s,q_2}},
\end{aligned}
\end{equation}with $1/2=1/p_1+1/p_2=1/q_1+1/q_2$ and $2\leq p_1,q_2<\infty$;

(2) $\forall s\geq 1$, we have
\begin{equation}\label{kato3}
\|J^s(fg)-(J^sf)g-f(J^sg)\|_{L^p}\lesssim\|f\|_{W^{1,p_1}}\|g\|_{W^{s-1,q_2}}+\|f\|_{W^{s-1,q_1}}\|g\|_{W^{1,q_2}}
\end{equation} for all the $1<p<p_1,p_2,q_1,q_2<\infty$ with $1/p_1+1/p_2=1/q_1+1/q_2=1/p$.
\end{lem}

\begin{lem}\label{harmonictrace}
\textbf{(Trace lemma for harmonic functions {\cite[Prop. 5.1.7]{taylorPDE1}) }}
Suppose that $s\geq 0.5$ and $u$ solves the boundary-value problem
\[
\Delta u=0~\text{ in }\Omega\text{ with }u=g~\text{ on }\Gamma
\] where $g\in H^{s}(\Gamma)$. Then it holds that 
\[
|g|_{s}\lesssim\|u\|_{s+0.5}\lesssim|g|_{s}
\]
\end{lem}

\begin{lem}[\textbf{Normal trace lemma}]\label{normaltrace}
It holds that for a vector field $X$
\begin{equation}\label{ntr}
|\TP X\cdot N|_{-0.5}\lesssim\|\TP X\|_0+\|\dive X\|_0
\end{equation}
\end{lem}
\begin{proof}
Let $\varphi\in H^{0.5}(\p\Omega)$ be a scalar test function, whose bounded extension in $\Omega$ is denoted by $\phi\in H^1(\Omega)$. Then 
\[
\int_{\p\Omega}\TP X\cdot N\varphi=\int_{\Omega}\dive(\TP X\phi)=\int_{\Omega}\TP X\nabla\phi-\int_{\Omega}\dive X\TP\phi\lesssim(\|\TP X\|_0+\|\dive X\|_0)|\varphi|_{0.5}.
\]
\end{proof}

\begin{lem}\label{bernstein}
\textbf{(Bernstein-type inequalities)} Let $0\leq \chi(\xi)\leq 1$ be a $C_c^{\infty}(\R^d)$ cut-off function which is supported in $\{|\xi|\leq 2\}$ and equals to 1 in $\{|\xi|\leq 1\}$. Define the Littlewood-Paley projection $P_{\leq N}$ in $\R^d$ with respect to $\chi$ by 
\[
P_{\leq N} f:=\left(\chi(\xi/N)\hat{f}(\xi)\right)^{\vee},~~P_{\geq N} f:=\left((1-\chi(\xi/N))\hat{f}(\xi)\right)^{\vee},~~P_{N} f:=\left((\chi(\xi/N)-\chi(2\xi/N))\hat{f}(\xi)\right)^{\vee}.
\] Then the following inequalities hold
\begin{align}
\label{bern1} \|P_{\leq N}f\|_{\dot{H}_x^s(\R^d)}&\lesssim_{s,d} N^s\|f\|_{L^2(\R^d)},~~~\forall s\geq 0;\\
\label{bern2} \|P_{\geq N}f\|_{\dot{H}_x^s(\R^d)}&\lesssim_{s,d} \|f\|_{\dot{H}_x^s(\R^d)},~~~\forall s\in \R.
\end{align}Analogous results also hold for $H_x^s(\R^d)$.
\end{lem}
\begin{proof}
For the first inequality, we apply Plancherel's identity to get 
\[
\|P_{\leq N}f\|_{\dot{H}_x^s(\R^d)}=\||\p|^s P_{\leq N}f\|_{L^2(\R^d)}=\||\xi|^s\chi(\xi/N)\hat{f}(\xi)\|_{L^2(\R^d)}\lesssim N^s\cdot 1\cdot\|f\|_{L^2(\R^d)}.
\] Note that $s\geq 0$ is used in the last inequality.
For the second inequality, we just replace $\chi(\xi/N)$ above by $1-\chi(\xi/N)$ and notice that $0\leq 1-\chi(\xi/N)\leq 1$ to get
\[
\|P_{\geq N}f\|_{\dot{H}_x^s(\R^d)}\leq\||\xi|^s\hat{f}(\xi)\|_{L^2(\R^d)}\lesssim \|f\|_{\dot{H}_x^s(\R^d)}.
\]
Analogous results hold for $H^s(\R^d)$ by replacing $|\p|$ and $|\xi|$ with $\langle \p\rangle$ and $\langle \xi\rangle$ respectively. One can see Tao \cite[Appendix A]{tao2006nonlinear} for more Bernstein-type inequalities.
\end{proof}

\subsubsection{Properties of tangential smoothing operator}

As stated in the introduction, we are going to use the tangential smoothing to construct the approximate solutions. Here we list the definition and basic properties which are repeatedly used in this paper. Let $\zeta=\zeta(y_1,y_2)\in C_c^{\infty}(\R^2)$ be a standard cut-off function such that $\text{Spt }\zeta=\overline{B(0,1)}\subseteq\R^2,~~0\leq\zeta\leq 1$ and $\int_{\R^2}\zeta=1$. The corresponding dilation is $$\zeta_{\kk}(y_1,y_2)=\frac{1}{\kk^2}\zeta\left(\frac{y_1}{\kk},\frac{y_2}{\kk}\right),~~\kk>0.$$ Now we define
\begin{equation}\label{lkk0}
\lkk f(y_1,y_2,y_3):=\int_{\R^2}\zeta_{\kk}(y_1-z_1,y_2-z_2)f(z_1,z_2,z_3)\dz_1\dz_2.
\end{equation}

The following lemma records the basic properties of tangential smoothing.
\begin{lem}\label{tgsmooth}
 \textbf{(Regularity and Commutator estimates)} For $\kk>0$, we have

(1) The following regularity estimates:
\begin{align}
\label{lkk11} \|\lkk f\|_s&\lesssim \|f\|_s,~~\forall s\geq 0;\\
\label{lkk1} |\lkk f|_s&\lesssim |f|_s,~~\forall s\geq -0.5;\\ 
\label{lkk2} |\TP\lkk f|_0&\lesssim \kk^{-s}|f|_{1-s}, ~~\forall s\in [0,1];\\  
\label{lkk3} |f-\lkk f|_{L^{\infty}}&\lesssim \sqrt{\kk}|\TP f|_{0.5}.  
\end{align}

(2) Commutator estimates: Define the commutator $[\lkk,f]g:=\lkk(fg)-f\lkk(g)$. Then it satisfies
\begin{align}
\label{lkk4} |[\lkk,f]g|_0 &\lesssim|f|_{L^{\infty}}|g|_0,\\ 
\label{lkk5} |[\lkk,f]\TP g|_0 &\lesssim |f|_{W^{1,\infty}}|g|_0, \\ 
\label{lkk6} |[\lkk,f]\TP g|_{0.5}&\lesssim |f|_{W^{1,\infty}}|g|_{0.5}.
\end{align}
\end{lem}
\begin{proof}
(1): The estimates \eqref{lkk1} and \eqref{lkk2} follows directly from the definition \eqref{lkk0} and the basic properties of convolution. \eqref{lkk3} is derived by using Sobolev embedding and H\"older's inequality:
\begin{align*}
|f-\lkk f|=\left|\int_{\R^2\cap B(0,\kk)} \zeta_{\kk}(z) (f(y-z)-f(y))\dz\right|\lesssim \left|\zeta_{\kk}\right|_{L^{\frac43}}\left|\kk\TP f\right|_{L^4}\lesssim \sqrt{\kk}\left|\zeta\right|_{L^{\frac43}} \left|\TP f\right|_{0.5}.
\end{align*}

(2): The first three estimates can be found in \cite[Lemma 5.1]{coutand2010LWP}. To prove the fourth one, we note that
\[
\TP([\lkk,f]g)=\lkk(\TP f\TP g)+\lkk(f\TP^2 g)-\TP f\lkk \TP g-f \lkk \TP^2g=[\lkk,\TP f]\TP g+[\lkk, f]\TP^2 g.
\]
From \eqref{lkk4} and \eqref{lkk5} we know 
\begin{equation}\label{lkk7}
|\TP[\lkk,f]g|_0\lesssim|\TP f|_{L^{\infty}}|\TP g|_{0}+|f|_{W^{1,\infty}}|\TP g|_0\lesssim |f|_{W^{1,\infty}}|g|_1.
\end{equation} Therefore \eqref{lkk6} follows from the interpolation of \eqref{lkk5} and \eqref{lkk7}.
\end{proof}

\subsubsection{Elliptic estimates}
\begin{lem} \label{hodge} \textbf{(Hodge-type decomposition)} Let $X$ be a smooth vector field and $s\geq 1$, then it holds that
\begin{equation}
\|X\|_s\lesssim\|X\|_0+\|\curl X\|_{s-1}+\|\dive X\|_{s-1}+|X\cdot N|_{s-0.5}.
\end{equation}
\end{lem}
\begin{proof}
This follows from the well-known identity $-\Delta X=\curl\curl X-\nabla\dive X$ and integration by parts.
\end{proof}

\begin{lem}\label{GLL}\textbf{(Interior elliptic estimate)}
The following elliptic estimate holds for $f=0$ on $\Gamma$.
\begin{equation}\label{GLL ell}
\begin{aligned}
\|\pak f\|_{\dH^1}^2 \leq&  C(\|\ek\|_{\hc})\left(\|\lap_{\ak} f\|_{0}^2+\|\p f\|_0^2\right), \quad\text{or}\quad  \|\pak f\|_{\dH^1}^2 \leq  C(\|\ek\|_{\hc})\left(\|\lap_{\ak} f\|_{0}^2+\|\p f\|_{L^\infty}^2\right), \\
\|\pak f\|_{\dH^r}^2 \leq&  C(\|\ek\|_{\hc})\left(\|\lap_{\ak} f\|_{r-1}^2+\|\p f\|_{L^\infty}^2+\|\p^2 f\|_{r-2}^2\right),\quad r=2,3, \\
\|\pak f\|_{\dH^r}^2 \leq&  C(\|\p \ek\|_{L^\infty}, \|\p^2 \ek\|_{r-2})\Bigg(\|\lap_{\ak} f\|_{r-1}^2+\|\cp \p \ek\|_{r-1}^2\left(\|\p f\|_{L^\infty}^2+\|\p^2 f\|_{r-2}^2\right)\Bigg),\quad r\geq 4.
\end{aligned}
\end{equation}
\end{lem}
\begin{proof}
The proof is largely similar to what is in \cite[Appendix B]{GLL2019LWP} and so we shall only sketch the details. The main idea here is to apply the div-curl estimate on $\|\nab_{\ak} f\|_{\dH^r}^2$. The (Eulerian) divergence contributes to the Laplacian term, and the (Eulerian) curl of $\pak f$ vanishes. Then the term $\|\TP^r \nab_{\ak} f\|_0^2$ will be generated by Lemma \ref{normaltrace} during this process. To control this term, we write
\begin{align*}
 \int_{\Omega} (\TP^r \nab_{\ak} f)(\TP^r\nab_{\ak} f) 
= \int_{\Omega} (\TP^r \nab_{\ak} f)(\nab_{\ak}\TP^r f) + \int_{\Omega} (\TP^r \nab_{\ak} f) ( [\TP^r, \nab_{\ak} ] f),
\end{align*}
where 
\begin{align*}
\int_{\Omega} (\TP^r \nab_{\ak} f) ( [\TP^r, \nab_{\ak} ] f)\lesssim \epsilon \|\TP^r \nab_{\ak} f\|_0^2+ \|[\TP^r, \nab_{\ak}] f\|_0^2,
\end{align*}
and
$\|[\TP^r, \nab_{\ak} ] f\|_{0}^2$ is controlled by either $C(\|\ek\|_{\hc})\|\p f\|_1^2$ or $C(\|\ek\|_{\hc})\|\p f\|_{L^\infty}^2$ when $r=1$,  by $ C(\|\ek\|_{\hc})\left(\|\p f\|_{L^\infty}^2+\|\p^2 f\|_{r-2}^2\right)$ when $r=2,3$, and by $ C(\|\p\ek\|_{L^\infty}, \|\p^2 \ek\|_{r-2})\|\TP\p \ek\|_{r-1}^2\left(\|\p f\|_{L^\infty}^2+\|\p^2 f\|_{r-2}^2\right)$ when $r\geq 4$. Moreover, by integrating $\nab_{\ak}$ by parts, we have
\begin{align}
\int_{\Omega} (\TP^r \nab_{\ak} f)(\nab_{\ak}\TP^r f)  = -\int_{\Omega} (\nab_{\ak}\TP^r\nab_{\ak} f) (\TP^r f), \label{nab TP nab}
\end{align}
and there is no boundary term since $f=0$ on $\Gamma$ implies $\TP^r f=0$ on $\Gamma$. The main term contributed by the RHS of \eqref{nab TP nab} after commuting $\nab_{\ak}$ through $\TP^r$ is $-\int_\Omega (\TP^r \lap_{\ak} f) (\TP^r f)$, which can be controlled by integrating $\TP$ by parts and then using the $\epsilon$-Young's inequality, i.e., 
\begin{align*}
-\int_\Omega (\TP^r \lap_{\ak} f) (\TP^r f) = \int_{\Omega}  (\TP^{r-1} \lap_{\ak} f) (\TP^{r+1} f) \lesssim \|\lap_{\ak}f\|_{r-1}^2+ \epsilon \|\TP^{r+1} f\|_{0}^2,
\end{align*}
where $\epsilon \|\TP^{r+1} f\|_{0}^2 \leq \epsilon \|\p \TP^r f\|_0^2$, and $\epsilon \|\p \TP^r f\|_0^2$ is comparable to $\epsilon\|\TP^r \nab_{\ak} f\|_0^2$ modulo error terms that take the form $\|[\TP^r, \nab_{\ak} ] f\|_{0}^2$. On the other hand, the error term generated by the RHS of \eqref{nab TP nab} after commuting $\nab_{\ak}$ through $\TP^r$  takes the form
\begin{align*}
\int_{\Omega} [\TP^r, \ak^{\mu\alpha}] \p_\mu (\ak^{\nu}_{\,\,\alpha} \p_\nu f) (\TP^r f),
\end{align*}
and it contributes to (up to the highest order)
$$
I:=\int_{\Omega} (\TP^r \ak^{\mu\alpha}) \left(\p_\mu (\ak^{\nu}_{\,\,\alpha} \p_\nu f) \right)(\TP^r f) ,\quad \text{and}\quad II:=\int_{\Omega} (\TP \ak^{\mu\alpha}) \left(\TP^{r-1}\p_\mu (\ak^{\nu}_{\,\,\alpha} \p_\nu f) \right)(\TP^r f).
$$
Here, 
\begin{align*}
 II \lesssim \epsilon \|\nab_{\ak} f\|_{\dH^r}^2+ \|(\TP \ak)(\TP^r f)\|_{0}^2, 
\end{align*}
and $\|(\TP \ak)(\TP^r f)\|_{0}^2$ can be bounded directly by the RHS of \eqref{GLL ell}. 
Also, for $I$, we have
\begin{align*}
\int_{\Omega} (\TP \ak^{\mu\alpha}) \left(\p_\mu (\ak^{\nu}_{\,\,\alpha} \p_\nu f) \right)(\TP f) \lesssim \epsilon \|\nab_{\ak} f\|_{\dH^1}^2+ \|(\TP \ak)(\TP f)\|_{0}^2,\\
\int_{\Omega} (\TP^r \ak^{\mu\alpha}) \left(\p_\mu (\ak^{\nu}_{\,\,\alpha} \p_\nu f) \right)(\TP^r f) \lesssim \|\TP^r f\|_0^2 + \|\p (\nab_{\ak} f)\|_{L^\infty}^2\|\TP^r \ak\|_0^2,  \quad \text{when}\,\,r\geq 2, 
\end{align*}
where $\|(\TP \ak)(\TP f)\|_{0}^2, \|\TP^r f\|_0^2$ and $\|\p (\nab_{\ak} f)\|_{L^\infty}^2\|\TP^r \ak\|_0^2$ can all be bounded directly by the RHS of \eqref{GLL ell}. 
\end{proof}

\begin{rmk}
The inequalities in \eqref{GLL ell} can be simplified to 
\begin{align*}
\|\pak f\|_{\dH^r}^2 \leq  C(\|\ek\|_{\hc})\left(\|\lap_{\ak} f\|_{r-1}^2+\|f\|_r^2\right),\quad r=1, 2,3, \\
\|\pak f\|_{\dH^r}^2 \leq  C(\|\p \ek\|_{L^\infty}, \|\p^2 \ek\|_{r-2})\Bigg(\|\lap_{\ak} f\|_{r-1}^2+\|\cp \p \ek\|_{r-1}^2\|f\|_r^2\Bigg),\quad r\geq 4,
\end{align*}
when $\Omega$ is a bounded domain. Nevertheless, for an unbounded domain $\Omega$ we have to be more careful when $f=h$ since $\p h$ can only be controlled in $L^\infty(\Omega)$.
\end{rmk}

\section{The Approximate system and uniform a priori estimates}\label{kkapriori}

In this section we are going to introduce the approximation of the water wave problem and derive its uniform a priori estimates.

\subsection{The approximate system}

For $\kk>0$, we consider the following approximate system 
\begin{equation}\label{app1}
\begin{cases}
\p_t \eta=v+\psi &~~~\text{in } \Omega, \\
\p_t v=-\pak h -ge_3 &~~~\text{in } \Omega, \\
\diva v=-e'(h)\p_t h &~~~\text{in } \Omega, \\
h=0 &~~~\text{on } \Gamma, \\
(\eta,v, h)|_{\{t=0\}}=(\text{Id},v_0, h_0).
\end{cases}
\end{equation} Here $\ak:=(\p\ek)^{-1}$ where $\ek$ is the smoothed version of the flow map $\eta$ defined by $\ek:=\lkk^2\eta$. The term $\psi=\psi(\eta,v)$ is a correction term which solves the half-space Laplacian equation
\begin{equation}\label{psi}
\begin{cases}
\Delta \psi=0  &~~~\text{in }\Omega, \\
\psi=\TL^{-1}\mathbb{P}\left(\TL\eta_{\beta}\ak^{i\beta}\TP_i\lkk^2 v-\TL\lkk^2\eta_{\beta}\ak^{i\beta}\TP_i v\right) &~~~\text{on }\Gamma,
\end{cases}
\end{equation}where $\mathbb{P}f:=P_{\geq 1} f$ denotes the standard Littlewood-Paley projection in $\R^2$ defined is Lemma \ref{bernstein}, which removes the low-frequency part. $\TL:=\p_1^2+\p_2^2$ denotes the tangential Laplacian operator and $\TL^{-1}f:=(-|\xi|^{-2}\hat{f})^{\vee}$ is the inverse of $\TL$ on $\R^2$.

\begin{rmk}
~\\
\begin{enumerate}
\item The correction term $\psi\to 0$ as $\kk\to 0$. We introduce such a term to eliminate the higher order boundary terms which appears in the tangential estimates of $v$. These higher order boundary terms are zero when $\kk=0$ but cannot be controlled when $\kk>0$.
\item The Littlewood-Paley projection is necessary here because we will repeatedly use $$|\TL^{-1}\mathbb{P}f|_{s}\lesssim |\mathbb{P}f|_{{H}^{s-2}}\approx|\mathbb{P}f|_{\dot{H}^{s-2}}\lesssim|f|_{\dot{H}^{s-2}},$$ which can be proved via Bernstein inequality \eqref{bern2}. Without $\mathbb{P}$ the low-frequency part loses control when taking $\TL^{-1}.$
\end{enumerate}
\end{rmk}

Fix any $\kk>0$, we will prove in Section \ref{kkexist} that there exists a $T_{\kk}>0$ depending on the initial data and $\kk>0$ such that there is a unique solution $(v(\kk),h(\kk),\eta(\kk))$ to \eqref{app1} in $[0,T_{\kk}]$. For simplicity we omit the $\kk$ and only write $v,h,\eta$ in this manuscript. The remaining context in this section is to derive the uniform-in-$\kk$ a priori estimates for the solutions to \eqref{app1}. \textit{This guarantees that we are able to obtain the solution of the original problem in some fixed time interval by passing $\kk\to 0$.}   

For simplicity in notations, we still denote the solution to the $\kk$-approximation system by $(\eta,v,h)$ with $\kk$ omitted. Define the energy functional for the $\kk$-approximate problem \eqref{app1} to be
\begin{equation}\label{Ekk}
\EE_{\kk}(t):=\|\eta(t)\|_{\hc}^2+\sum_{k=0}^4\|\p_t^{4-k}v(t)\|_k^2+\left(\|h(t)\|_{\hc}^2+\sum_{k=0}^3\|\p_t^{4-k}h(t)\|_k^2\right)+|\ak^{3\alpha}\TP^4\lkk\eta_{\alpha}(t)|_0^2.
\end{equation} 
\begin{rmk}
We recall that \eqref{Ekk} can be simplified to 
$$
\|\eta(t)\|_4^2+ \sum_{k=0}^4 (\|\p_t^{4-k} v(t)\|_{k}^2+ \|\p_t^{4-k} h(t)\|_{k}^2)+|\ak^{3\alpha}\TP^4\lkk\eta_{\alpha}(t)|_0^2.
$$ in the case a bounded domain. We refer to the remark after \eqref{energy} and Section \ref{comparsion} for the details. 
\end{rmk}
The rest of this section is devoted to prove:
\begin{prop}\label{uniformkk}
Let $\EE_\kk$ be defined as above. Then there exists a time $T>0$ independent of $\kk$ such that 
\begin{equation}\label{Ekk0}
\sup_{0\leq t\leq T} \EE_{\kk}(t)\leq P(\|v_0\|_4,\|h_0\|_{\hc}).
\end{equation}
\end{prop}
Proposition \ref{uniformkk} is a direct consequence of the following proposition:
\begin{prop}
Let $\EE_{\kk}$ be defined as above. Then it holds that
\begin{align}
\EE_{\kk}(t) \leq P(\|v_0\|_4, \|h_0\|_{\hc})+\int_0^t P(\EE_\kk(\tau))\,d\tau, \q \forall t\in[0,T]
\label{energy int form}
\end{align}
provided the following a priori assumptions hold
\begin{align}
\label{taylor2} -\p_3 h(t)\geq \frac{c_0}{2}   &~~~\text{on }\Gamma, \\
\label{Jkk1} \|\tilde{J}(t)-1\|_3\leq \epsilon &~~~\text{in }\Omega, \\
\label{akk1} \|\text{Id}-\ak(t)\|_{3}\leq\epsilon &~~~\text{in }\Omega,
\end{align} 
where $\tilde{J}:=\det(\p\ek)$ and we use $\epsilon>0$ to denote the sufficiently small number which appears here and the $\epsilon$-Young inequality.
\end{prop}

\begin{rmk}
It suffices to show that \eqref{energy int form} holds true when $t=T$. 
Also, \eqref{energy int form} can in fact be reduced to 
\begin{align}
\EE_{\kk}(T) \leq \EE_\kk(0)+\int_0^T P(\EE_\kk(t))\,dt.
\end{align}
In \cite{luo2018ww} we are able to prove that there exists initial data satisfying the compatibility condition \eqref{c'cond} up to order $5$ such that $\EE_\kk(0) \leq P(\|v_0\|_4, \|h_0\|_{\hc})$ holds. For notation simplicity we define $\PP_0:=P(\|v_0\|_4, \|h_0\|_{\hc})$.
\end{rmk}

\subsection{Estimates for the flow map and correction term}

First we bound the flow map and the correction term together with their smoothed version by the quantities in $\EE_{\kk}$. The following estimates will be repeatedly use in this section.

\begin{lem}\label{etapsi}
Let $(v,h,\eta)$ be the solution to \eqref{app1}. Then we have
\begin{align}
\label{etaLinfty}  \|\p \ek\|_{L^\infty}&\lesssim \|\p\eta\|_{L^{\infty}}, \\
\label{eta4} \|\p^2\ek\|_2&\lesssim \|\p^2\eta\|_2, \\
\label{psi4} \|\psi\|_4&\lesssim P(\|\p\eta\|_{L^{\infty}},\|\p^2\eta\|_2,\|v\|_3), \\
\label{psit4}\|\p_t \psi\|_4&\lesssim P(\|\p\eta\|_{L^{\infty}},\|\p^2\eta\|_2,\|v\|_4,\|\p_t v\|_3),\\
\label{psitt3}\|\p_t^2 \psi\|_3&\lesssim P(\|\p\eta\|_{L^{\infty}},\|\p^2\eta\|_2,\|v\|_4, \|\p_tv\|_3, \|\p_t^2 v\|_2), \\
\label{psittt2}\|\p_t^3\psi\|_2&\lesssim P(\|\p\eta\|_{L^{\infty}},\|\p^2\eta\|_2,\|v\|_4, \|\p_tv\|_3, \|\p_t^2 v\|_2, \|\p_t^3v\|_{1}).
\end{align}
and
\begin{align}
\label{etat4} \|\p_t\ek\|_4\lesssim\|\p_t\eta\|_4&\lesssim P(\|\p\eta\|_{L^{\infty}},\|\p^2\eta\|_2,\|v\|_4) , \\
\label{etatt3}\|\p_t^2\ek\|_3\lesssim\|\p_t^2\eta\|_3&\lesssim P(\|\p\eta\|_{L^{\infty}},\|\p^2\eta\|_2,\|v\|_4,\|\p_t v\|_3) , \\
\label{etattt2}\|\p_t^3\ek\|_2\lesssim\|\p_t^3\eta\|_2&\lesssim P(\|\p\eta\|_{L^{\infty}},\|\p^2\eta\|_2,\|v\|_4,\|\p_t v\|_3,\|\p_t^2 v\|_2), \\
\label{etatttt1}\|\p_t^4\ek\|_1\lesssim\|\p_t^4\eta\|_1&\lesssim  P(\|\p\eta\|_{L^{\infty}},\|\p^2\eta\|_2,\|v\|_4, \|\p_tv\|_3, \|\p_t^2 v\|_2, \|\p_t^3v\|_{1}). 
\end{align}
\end{lem}
\begin{proof}
First, \eqref{etaLinfty} and \eqref{eta4} follow from \eqref{lkk11}, i.e., $\|\p \ek\|_{L^\infty}=\|\lkk^2\p \eta\|_{L^\infty}\lesssim\|\p \eta\|_{L^\infty}$, $\|\p^2\ek\|_2=\|\lkk^2\p^2\eta\|_2\lesssim\|\p^2\eta\|_2$. To bound $\p_t^k\ek$, it suffices to bound the same norm of $\p_t^k\eta$ and then apply \eqref{lkk11} again. From the first equation of \eqref{app1}, one has $\p_t^{k+1}\eta=\p_t^k v+\p_t^k \psi$, so the estimates \eqref{etat4}-\eqref{etatttt1} automatically holds once we prove \eqref{psi4}-\eqref{psittt2}.

Commuting time derivatives through \eqref{psi}, we get the equations for $\p_t^k\psi~(k=0,1,2,3,4)$:
\begin{equation}\label{psitk}
\begin{cases}
\Delta \p_t^k\psi=0  &~~~\text{in }\Omega, \\
\p_t^k\psi=\TL^{-1}\mathbb{P}\p_t^k\left(\TL\eta_{\beta}\ak^{i\beta}\TP_i\lkk^2 v-\TL\lkk^2\eta_{\beta}\ak^{i\beta}\TP_i v\right) &~~~\text{on }\Gamma.
\end{cases}
\end{equation}
By the standard elliptic estimates, Sobolev trace lemma and Bernstein inequality \eqref{bern2} in Lemma \ref{bernstein}, we can get 
\begin{equation}\label{psi040}
\begin{aligned}
\|\psi\|_{4}&\lesssim\left|\TL^{-1}\mathbb{P}\left(\TL\eta_{\beta}\ak^{i\beta}\TP_i\lkk^2 v-\TL\lkk^2\eta_{\beta}\ak^{i\beta}\TP_i v\right)\right|_{3.5} \\
&\lesssim \left|\TL\eta_{\beta}\ak^{i\beta}\TP_i\lkk^2 v-\TL\lkk^2\eta_{\beta}\ak^{i\beta}\TP_i v\right|_{1.5} \\
&\lesssim \|\TL\eta_{\beta}\ak^{i\beta}\TP_i\lkk^2 v-\TL\lkk^2\eta_{\beta}\ak^{i\beta}\TP_i v\|_{2} \\
&\lesssim \|\p^2\eta\|_2\|\ak\|_{L^{\infty}}\|v\|_3\leq P(\|\p^2\eta\|_2,\|\p\eta\|_{L^{\infty}},\|v\|_3).
\end{aligned}
\end{equation}
Also, when $k=1,2,3$, one has 
\begin{equation}\label{psit40}
\begin{aligned}
\|\p_t\psi\|_{4}&\lesssim\left|\TL^{-1}\p_t\mathbb{P}\left(\TL\eta_{\beta}\ak^{i\beta}\TP_i\lkk^2 v-\TL\lkk^2\eta_{\beta}\ak^{i\beta}\TP_i v\right)\right|_{3.5} \\
&\lesssim \left|\p_t\left(\TL\eta_{\beta}\ak^{i\beta}\TP_i\lkk^2 v-\TL\lkk^2\eta_{\beta}\ak^{i\beta}\TP_i v\right)\right|_{1.5} \\
&\lesssim \|\p_t(\TL\eta_{\beta}\ak^{i\beta}\TP_i\lkk^2 v-\TL\lkk^2\eta_{\beta}\ak^{i\beta}\TP_i v)\|_{2} \\
&\lesssim  P(\|\p^2\eta\|_2,\|\p\eta\|_{L^{\infty}}, \|v\|_4, \|\p_t v\|_3),
\end{aligned}
\end{equation} 

\begin{equation}\label{psitt30}
\begin{aligned}
\|\p_t^2\psi\|_{3}&\lesssim\left|\TL^{-1}\p_t^2\mathbb{P}\left(\TL\eta_{\beta}\ak^{i\beta}\TP_i\lkk^2 v-\TL\lkk^2\eta_{\beta}\ak^{i\beta}\TP_i v\right)\right|_{2.5} \\
&\lesssim \left|\p_t^2\left(\TL\eta_{\beta}\ak^{i\beta}\TP_i\lkk^2 v-\TL\lkk^2\eta_{\beta}\ak^{i\beta}\TP_i v\right)\right|_{0.5} \\
&\lesssim \|\p_t^2(\TL\eta_{\beta}\ak^{i\beta}\TP_i\lkk^2 v-\TL\lkk^2\eta_{\beta}\ak^{i\beta}\TP_i v)\|_{1} \\
&\lesssim P(\|\p^2\eta\|_2,\|v\|_4, \|\p\eta\|_{L^{\infty}},\|\p_t v\|_3,\|\p_t^2v\|_2),
\end{aligned}
\end{equation}
and
\begin{equation}\label{psittt20}
\begin{aligned}
\|\p_t^3\psi\|_{2}&\lesssim\left|\TL^{-1}\p_t^3\mathbb{P}\left(\TL\eta_{\beta}\ak^{i\beta}\TP_i\lkk^2 v-\TL\lkk^2\eta_{\beta}\ak^{i\beta}\TP_i v\right)\right|_{1.5} \\
&\lesssim \left|~\mathbb{P}\p_t^3\left(\TL\eta_{\beta}\ak^{i\beta}\TP_i\lkk^2 v-\TL\lkk^2\eta_{\beta}\ak^{i\beta}\TP_i v\right)\right|_{-0.5}\\
\end{aligned}
\end{equation}
where in the last step we apply the Bernstein's inequality \eqref{bern2}.

Combining with $\p_t^{k+1}\eta=\p_t^k v+\p_t^k \psi$, \eqref{etat4}, \eqref{etatt3} and \eqref{etattt2} directly follows from \eqref{psit40} and \eqref{psitt30}, respectively. When $k=3$, one has to be cautious because the leading order term in \eqref{psittt20} is of the form $(\p_t^3\TL\eta)\ak\TP v$ and $\TL\eta\ak(\p_t^3\TP v)$ which can only be bounded in $L^2(\Omega)$ by the quantites in $\EE_{\kk}$ and thus loses control on the boundary. To control these terms on the boundary, we have to use the fact that $\dot{H}^{0.5}(\R^2)=(\dot{H}^{0.5}(\R^2))^*$.

First we separate them from other lower order terms which has $L^2(\Gamma)$ control.
\begin{equation}\label{XY}
\begin{aligned}
&~~~~\mathbb{P}\p_t^3\left(\TL\eta_{\beta}\ak^{i\beta}\TP_i\lkk^2 v-\TL\lkk^2\eta_{\beta}\ak^{i\beta}\TP_i v\right) \\
&=\mathbb{P}\underbrace{\left(\p_t^3 \TL\eta_{\beta}\ak^{i\beta}\TP_i\lkk^2 v-\p_t^3\TL\lkk^2\eta_{\beta}\ak^{i\beta}\TP_i v+\TL\eta_{\beta}\ak^{i\beta}\p_t^3\TP_i\lkk^2 v-\TL\lkk^2\eta_{\beta}\ak^{i\beta}\p_t^3\TP_i v\right)}_{\text{leading order terms=:X}}+\mathbb{P}Y.
\end{aligned}
\end{equation}
The control of $Y$ is straightforward by using Sobolev trace lemma and \eqref{etat4}, \eqref{etatt3},
\begin{equation}\label{Y}
\begin{aligned}
|\mathbb{P}Y|_{-0.5}&\leq|\mathbb{P}Y|_{0.5}\lesssim\|Y\|_{1} \\
&\lesssim P(\|\p_t^2\eta\|_{2.5}, \|\p_t \eta\|_{3.5}, \|\p_t^2 \ak\|_{1.5},\|\p_t^2 v\|_{1.5}, \|\p_t v\|_{2.5}) \\
&\lesssim P(\|\p^2\eta\|_2, \|v\|_4,\|\p_t v\|_3,\|\p_t^2 v\|_2).
\end{aligned}
\end{equation}
As for the $|\mathbb{P} X|_{-0.5}$ term, we first use the Bernstein inequality \eqref{bern2} to get $|\mathbb{P} X|_{-0.5}\approx |\mathbb{P}X|_{\dot{H}^{-0.5}}\lesssim |X|_{\dot{H}^{-0.5}}$. Then the duality between $\dot{H}^{-0.5}$ and $\dot{H}^{0.5}$ yields that for any test function $\phi\in \dot{H}^{0.5}(\R^2)$ with $|\phi|_{\dot{H}^{0.5}}\leq 1$, one has
\begin{equation}\label{X1}
\begin{aligned}
\langle \TL\eta_{\beta}\ak^{i\beta}\p_t^3\TP_i\lkk^2 v, \phi\rangle &=\langle \p_t^3\TP_i\lkk^2 v, \TL\eta_{\beta}\ak^{i\beta}\phi\rangle \\
&=\langle \TP_i^{0.5}\p_t^3\lkk^2 v,\TP_i^{0.5}(\TL\eta_{\beta}\ak^{i\beta}\phi)\rangle \\
&\lesssim|\p_t^3\lkk^2 v|_{\dot{H}^{0.5}}|\TL\eta\ak\phi|_{\dot{H}^{0.5}} \\
&\lesssim\|\p_t^3 v\|_{1}(|\phi|_{\dot{H}^{0.5}}|\TL\eta\ak|_{L^{\infty}}+|\TL\eta\ak|_{\dot{W}^{0.5,4}}|\phi|_{L^4})\\
&\lesssim\|\p_t^3 v\|_1(\|\p^2\eta\|_2\|a\|_{L^{\infty}})|\phi|_{\dot{H}^{0.5}}.
\end{aligned}
\end{equation} Here we integrate 1/2-order tangential derivative on $\Gamma$ by part in the second step, and then apply trace lemma to control $|\p_t^3 \lkk^2 v|_{\dot{H}^{0.5}}$ and Kato-Ponce product estimate \eqref{product} to bound $|\TL\eta\ak\phi|_{\dot{H}^{0.5}}$. Taking supremum over all $\phi\in H^{0.5}(\R^2)$ with $|\phi|_{\dot{H}^{0.5}}\leq 1$, we have by the definition of $\dot{H}^{0.5}$-norm that 
\begin{equation}\label{X0}
|\TL\eta_{\beta}\ak^{i\beta}\p_t^3\TP_i\lkk^2 v|_{\dot{H}^{-0.5}}\lesssim P(\|\p^2\eta\|_2,\|\p\eta\|_{L^{\infty}},\|\p_t^3 v\|_1).
\end{equation}
Similarly as above, we have 
\begin{align}
\label{X2} |\TL\lkk^2\eta_{\beta}\ak^{i\beta}\p_t^3\TP_i v|_{\dot{H}^{-0.5}}&\lesssim P(\|\p^2\eta\|_2,\|\p_t^3 v\|_1), \\
\label{X3} |\p_t^3 \TL\eta_{\beta}\ak^{i\beta}\TP_i\lkk^2 v-\p_t^3\TL\lkk^2\eta_{\beta}\ak^{i\beta}\TP_i v|_{\dot{H}^{-0.5}}&\lesssim P(\|\p_t^3\eta\|_2, \|\p\eta\|_{L^{\infty}}, \|v\|_3).
\end{align}
Combining \eqref{psittt20}-\eqref{X3} and the bound \eqref{etattt2} for $\p_t^3\eta$, we get 
\[
\|\p_t^3\psi\|_2\lesssim P(\|\p^2\eta\|_2,\|v\|_4, \|\p_tv\|_3, \|\p_t^2 v\|_2, \|\p_t^3v\|_{1}),
\]which is exactly \eqref{psittt2}. Hence, \eqref{etatttt1} directly follows from \eqref{psittt2} and $\p_t^4\eta=\p_t^3(v+\psi)$.

\end{proof}

\subsection{Estimates for the enthalpy $h$}\label{hapriori}

In this section we are going to control the Sobolev norm $\|\p_t^{4-k} h\|_k$ for $k=0,1,2,3$ and $\|h\|_{\hc}$. 


\subsubsection{Estimates for the first order derivatives of $h$} \label{lower order h}

We need to control $\|\p h\|_{L^{\infty}}$ and $\|\p_t h\|_0$. First, we have $$
\frac{\p h}{\p y^{\alpha}}=\delta^{\mu\alpha}\p_\mu h=\ak^{\mu\alpha}\p_{\mu} h+(\delta^{\mu\alpha}-\ak^{\mu\alpha})\p_{\mu}h.
$$
 Invoking the a priori assumption \eqref{akk1} and the second equation in \eqref{app1}, we get
\begin{equation}
\|\p h\|_{L^{\infty}}\leq \|\pak h\|_{L^{\infty}}+\epsilon\|\p h\|_{L^{\infty}}\leq  (g+\|\p_t v\|_{L^{\infty}})+\epsilon\|\p h\|_{L^{\infty}}.
\end{equation} 
Therefore, for sufficiently small $\epsilon$, which can be achieved by choosing $T>0$ smaller if needed, we have
\begin{equation}\label{h infty}
\|\p h\|_{L^{\infty}}\lesssim g+\|\p_t v\|_{L^{\infty}}.
\end{equation}
Second, as for $\|\p_t h\|_0$, we use the third equation of \eqref{app1} and the physical assumption \eqref{e(h) cond} to get
\begin{equation}\label{ht L2}
\|\p_t h\|_{0}\lesssim \|\ak^{\mu\alpha}\p_{\mu}v_{\alpha}\|_0\lesssim\|\ak\|_{L^{\infty}}\|\p v\|_0\lesssim\|\eta\|_{\hc}^2\|\p v\|_{0}.
\end{equation}

\subsubsection{Estimates for the top order derivatives of $h$} \label{tangential h}

We take the Eulerian divergence (i.e., $\diva$) in the second equation of system \eqref{app1} and use the third equation of  \eqref{app1} to get a wave equation of $h$:
\begin{equation}\label{waveh}
\tilde{J}e'(h)\p_t^2 h-\p_{\nu}(E^{\nu\mu}\p_{\mu}h)=\underbrace{-\tilde{J}\p_t\ak^{\nu\alpha}\p_{\nu}v_{\alpha}}_{:=F}-\tilde{J}e''(h)(\p_t h)^2,
\end{equation}where $E^{\nu\mu}=\tilde{J}\ak^{\nu\alpha}\ak^{\mu}_{\alpha}$. Note that the matrix $E$ is symmetric and positive-definite thanks to  \eqref{Jkk1}.

Let $\dd=\p_t$ or $\TP$. Let $\dd^3=\TP^3,\TP^2\p_t,\TP\p_t^2$ or $\p_t^3$, i.e., all the 3rd-order tangential derivatives.  Applying $\dd^3$ to \eqref{waveh}, we get
\begin{equation}\label{waveh3}
\tilde{J} e'(h)\p_t^2\dd^3 h-\p_{\nu}(E^{\nu\mu}\dd^3\p_{\mu}h)=\dd^3 F\underbrace{-[\dd^3,\tilde{J} e'(h)]\p_t^2h+\dd^3(\tilde{J} e''(h)(\p_t h)^2)}_{F_3}+\p_{\nu}([\dd^3,E^{\nu\mu}]\p_{\mu}h).
\end{equation}
Multiplying \eqref{waveh3} by $\p_t\dd^3 h$, then integrating $\p_{\nu}$ by parts, we have
\begin{align}
&\label{h4energy0} ~~~~~~~\frac{1}{2}\frac{d}{dt}\io\tilde{J}  e'(h)|\dd^3\p_t h|^2+ E^{\mu\nu}\p_{\nu}\dd^3 h\p_\mu\dd^3 h\dy \\
&\label{h4e1}  =~~\frac{1}{2}\io \p_t E^{\nu\mu}\dd^3\p_\nu h\dd^3\p_\mu h\dy \\
&\label{h4e2} ~~~~+\io F_3 \p_t\dd^3 h\dy \\
&\label{h4e3} ~~~~+\io \dd^3 F\p_t\dd^3 h\dy \\
&\label{h4e4} ~~~~+\io \p_{\nu}([\dd^3,E^{\nu\mu}]\p_{\mu}h)\p_t\dd^3 h\dy.
\end{align}
\eqref{h4e1} can be directly bounded by the energy:
\begin{equation}\label{h4e11}
\eqref{h4e1}\lesssim\|\p_t E\|_{L^{\infty}}\|\dd^3\p h\|_0^2\lesssim P(\|\p\eta\|_{L^{\infty}},\|\p v\|_2,\|\p\psi\|_2,\|\dd^3\p h\|_0)
\end{equation}
To estimate \eqref{h4e2}, it suffices to bound $\|F_3\|_0$. The precise form of $F_3$ is
\[
F_3=\sum_{m=2}^5 \sum e^{(m)}(h)(\p_t^{i_1}\dd^{j_1} h)\cdots(\p_t^{i_m}\dd^{j_m} h),
\]where the second sum is taken over the set $\{i_1+\cdots+i_m=2,~j_1+\cdots+j_m=3,~1\leq i_m+j_m\leq 4\}$. Invoking the condition imposed on $e(h)$ (i.e., \eqref{e(h) cond}), one has
\begin{equation}\label{h4e21}
\sum_{\dd^3}\|F_3\|_0\lesssim P(\|\p_t^4 h\|_0,\|\p_t^3 h\|_1,\|\p_t^2h\|_2,\|\p_t h\|_3,\|\dd h\|_{L^{\infty}}).
\end{equation}
As for \eqref{h4e3}, one has $\dd^3 F=\dd^3(\tilde{J}\ak^{\nu\alpha}\ak^{\mu}_{\alpha})$. 
\begin{itemize}
\item When $\dd^3=\TP^3$, then $\|\dd^3(\tilde{J}\ak^{\nu\alpha}\ak^{\mu}_{\alpha})\|_0\lesssim P(\|\eta\|_{\hc})$.

\item When $\dd^3$ contains at least one time derivative, then 
\begin{align*}
\|\dd^3 F\|_0&=\|\dd^2\p_t(\tilde{J}\ak^{\nu\alpha}\ak^{\mu}_{\alpha})\|_0\lesssim\|\p_t J\|_2\|a\|_{L^{\infty}}^2+\|J\|_2\|\p_ta\|_2\|a\|_{L^{\infty}} \lesssim P(\|\p v\|_2,\|\eta\|_{\hc}).
\end{align*}
\end{itemize}
Therefore, 
\begin{equation}\label{h4e31}
\eqref{h4e3}\lesssim P(\|\eta\|_{\hc}, \|v\|_3)\|\p_t\dd^3 h\|_0
\end{equation}

Finally, one has to be cautious when controlling \eqref{h4e4}. The leading order term in $\p_{\nu}([\dd^3,E^{\nu\mu}]\p_{\mu}h)$ is $\p\dd^3 E$.  If $\dd^3=\TP^3$, then this term loses control in $L^2$. To avoid this problem, one can integrate $\p_{\nu}$ by parts, and then integrate $\p_t$ by parts in the time integral of \eqref{h4e4} to replace $\p_\nu$ falling on $E$ by $\p_t$. This is because $J$ and $\p_t J$ (also for $a$ and $\p_t a$) have the same spatial regularity. If $\dd^3$ contains at least one time derivative, then the $L^2$-norm of $\p\dd^3 E$ can be controlled directly thanks to the same reason above.

\begin{itemize}
\item $\dd^3$ contains at least one time derivative, i.e., $\dd^3=\dd^2\p_t$. Then
\begin{align*}
&~~~~\io \p_\nu ([\dd^2\p_t, E^{\nu\mu}]\p_\mu h)\p_t\dd^3 h\dy \\
&\lesssim \|[\dd^2\p_t, E^{\nu\mu}]\p_\mu h\|_1\|\p_t\dd^3 h\|_0 \\
&\lesssim (\|\dd^2\p_t E\|_1\|\p h\|_{L^{\infty}}+\|\dd^2 E\|_{L^{\infty}}\|\p_t h\|_2+\|\dd\p_t E\|_{L^{\infty}}\|\dd h\|_2+\|\dd E\|_{L^{\infty}}\|\dd\p_t h\|_1+\|\p_t E\|_{L^{\infty}}\|\dd^2 h\|_1)\|\p_t\dd^3 h\|_0.
\end{align*} So 
\begin{equation}\label{h4e41}
\sum_{\dd^3\backslash\{\TP^3\}}\eqref{h4e4}\lesssim P(\|\p^2\eta\|_2,\|\p\eta\|_{L^{\infty}},\|v\|_4, \|\p_t v\|_3,\|\p_t^2 v\|_2, \|\p h\|_{L^{\infty}},\|\p_t h\|_2,\|\p_t^2 h\|_1)\|\p_t\dd^3 h\|_0.
\end{equation}

\item When $\dd^3=\TP^3$, we consider the time integral of \eqref{h4e4}. We first integrate $\p_\nu$ by parts, then integrate $\p_t$ by parts to get the following equality
\begin{align*}
&~~~~\int_0^T\io \p_\nu ([\TP^3, E^{\nu\mu}]\p_\mu h)\p_t\TP^3 h\dy \dt\\
&=-\int_0^T\io ([\TP^3, E^{\nu\mu}]\p_\mu h)\p_\nu\p_t\TP^3 h\dy \dt+\int_0^T\ig([\TP^3, E^{\nu\mu}]\p_\mu h)N_\nu\underbrace{\p_t\TP^3 h}_{=0}~dS\dt \\
&=\int_0^T\io \p_t([\TP^3, E^{\nu\mu}]\p_\mu h)\p_\nu\TP^3 h\dy \dt-\io([\TP^3, E^{\nu\mu}]\p_\mu h)\p_\nu\TP^3 h\dy\bigg|^{t=T}_{t=0}.
\end{align*}
\end{itemize}
The leading order term in the first integral is $\TP^3\p_t E$ which has $L^2$- control, so one can bound this directly by using H\"older's inequality
\begin{equation}\label{h4e42}
\int_0^T\io \p_t([\TP^3, E^{\nu\mu}]\p_\mu h)\p_\nu\TP^3 h\dy \dt\lesssim \int_0^T P(\|\eta\|_{\hc},\|\p v\|_3,\|\p h\|_{L^{\infty}},\|\p_t h\|_3)\|\TP^3h\|_1\dt.
\end{equation}
As for the second integral, we can use H\"older's inequality first, then use $\epsilon$-Young's inequality and Jensen's inequality
\begin{equation}
\begin{aligned}
&~~~~-\io([\TP^3, E^{\nu\mu}]\p_\mu h)\p_\nu\TP^3 h\dy\bigg|_{t=T} \\
&\lesssim\|\TP^3 E(T)\|_0\|\p h(T)\|_{L^{\infty}}\|\TP^3 h(T)\|_1+ \|\TP^2 E(T)\|_1\|\p\TP h(T)\|_{1}\|\TP^3 h(T)\|_1+\|\TP E(T)\|_{L^{\infty}}\|\p\TP^2 h(T)\|_{0}\|\TP^3 h(T)\|_1  \\
&\lesssim\epsilon\|\TP^3 h(T)\|_1^2+\frac{1}{8\epsilon}(\|\TP E(T)\|_{2}^4+\|\TP E(T)\|_{L^{\infty}}^4+\|\TP^2 h(T)\|_1^4) \\
&\lesssim\epsilon\|\TP^3 h(T)\|_1^2+\frac{1}{8\epsilon}\left(\|\TP E(0)\|_2^4+\|\TP E(0)\|_{L^{\infty}}^4+\|h(0)\|_3^4+\int_0^T\|\p_t\TP E(t)\|_{2}^4+\|\p_t\TP^2h(t)\|_1^4\dt\right)\\
&\lesssim\epsilon\|\TP^3 h(T)\|_1^2+\PP_0+\int_0^T P(\|\eta\|_{\hc},\|v\|_4,\|\p_t h\|_3)\dt.
\end{aligned}
\end{equation}
The above estimates along with $$\io([\TP^3, E^{\nu\mu}]\p_\mu h)\p_\nu\TP^3 h\dy\bigg|_{t=0}\lesssim\PP_0$$ give the bound for the time integral of \eqref{h4e4}:
\begin{equation}\label{h4e43}
\int_0^T\io \p_\nu ([\TP^3, E^{\nu\mu}]\p_\mu h)\p_t\TP^3 h\dy \dt\lesssim\epsilon\|\TP^3 h(T)\|_1^2+\PP_0+\int_0^T P(\|\eta\|_{\hc},\|v\|_4,\|\p_t h\|_3)\dt.
\end{equation}

Now, summing up \eqref{h4e11}-\eqref{h4e42}, \eqref{h4e43} and then plugging it into \eqref{h4energy0}, we get the tangential derivative estimates of $h$:
\begin{equation}\label{h4energy00}
\sum_{\dd^3}\io\tilde{J}e'(h)|\dd^3\p_t h|^2+ |\p\dd^3 h|^2\dy\bigg|^{t=T}_{t=0}\lesssim \epsilon\|\TP^3h(T)\|_1^2+\PP_0+\int_0^T P(\EE_{\kk}(t))\dt.
\end{equation}Note that we have used $E$ is symmetric and positive-definite. Choosing $\epsilon>0$ sufficiently small, the term $\epsilon\|\TP^3h(T)\|_1^2$ can absorbed by the LHS of \eqref{h4energy00}.

\subsubsection{Estimates for the full Sobolev norm}\label{lessnormal}

Up to now, we have controlled all the tangential space-time derivative of $\p h$. Therefore it suffices to control $\geq 2$ normal derivatives of $h$. Actually this follows directly from the wave equation \eqref{waveh}
\[
e'(h)\p_t^2 h-\p_{\nu}(E^{\nu\mu}\p_{\mu}h)=\underbrace{-\tilde{J}\p_t\ak^{\nu\alpha}\p_{\nu}v_{\alpha}}_{:=F_0}-e''(h)(\p_t h)^2
\] that 
\[
\p_{33}h=-\frac{1}{E_{33}}\left(F_0-e''(h)(\p_t h)^2-\sum_{\nu+\mu\leq 5}\p_{\nu}(E^{\nu\mu}\p_{\mu}h)-e'(h)\p_t^2 h\right),
\] because the above identity shows that the second order normal derivative $\p_{33}h$ can be bounded by the terms containing $h$ with the same or lower order derivates and less normal derivatives. Hence, one can apply the same method to inductively control terms containing $h$ with more normal derivatives. For example, $\p_{3333}h$ can be controlled in the same way by taking $\p^2$ in \eqref{waveh} and then express $\p_{3333}h$ in terms of the terms with same or lower order and $\leq 3$ normal derivatives.

Therefore, combining with \eqref{Jkk1}, \eqref{h infty}, \eqref{ht L2} and \eqref{h4energy00}, one has the control for the Sobolev norm of enthalpy $h$ and its time derivatives after taking $\epsilon>0$ in \eqref{h4energy00} sufficiently small to be absorbed by $\|h\|_{\hc}^2$:
\begin{equation}\label{h4energy}
\left(\|h(t)\|_{\hc}^2+\sum_{k=1}^3\|\p_t^k h(t)\|_{4-k}^2+\|\sqrt{e'(h)}\p_t^4 h(t)\|_0^2\right)\bigg|_{t=0}^{t=T}\lesssim \PP_0+\int_0^T P(\EE_{\kk}(t))\dt.
\end{equation}

\subsection{The div-curl estimates for $v$}\label{divcurl}

In this section we are going to do the div-curl estimates for $v$ and its time derivatives in order to reduce the estimates of $\EE_{\kk}$ to the tangential estimates. Recall the Hodge-type decomposition in \eqref{hodge}:
\[
\forall s\geq 1:~~\|X\|_s\lesssim\|X\|_0+\|\curl X\|_{s-1}+\|\dive X\|_{s-1}+|X\cdot N|_{s-0.5}.
\]

Let $X=v,\p_t v, \p_t^2 v, \p_t^3 v$ and $s=4,3,2,1$, respectively. We get
\begin{equation}\label{hodgev}
\begin{aligned}
\|v\|_4&\lesssim\|v\|_0+\|\dive v\|_3+\|\curl v\|_3+|\TP^3 (v\cdot N)|_{0.5} \\
\|\p_t v\|_3&\lesssim\|\p_t v\|_0+\|\dive \p_t v\|_2+\|\curl \p_t v\|_2+|\TP^2 (\p_tv\cdot N)|_{0.5} \\
\|\p_t^2 v\|_2&\lesssim\|\p_t^2 v\|_0+\|\dive \p_t^2 v\|_1+\|\curl\p_t^2 v\|_1+|\TP (\p_t^2v\cdot N)|_{0.5} \\
\|\p_t^3 v\|_1&\lesssim\|\p_t^3 v\|_0+\|\dive \p_t^2 v\|_0+\|\curl\p_t^2 v\|_0+|\p_t^3 v\cdot N|_{0.5}.
\end{aligned}
\end{equation}
First, the $L^2$-norm of $v$ is controlled by:
\begin{equation}\label{v0l2}
\|v(T)\|_0 \leq \|v_0\|_0+\int_0^T \|\p_t v(t)\|_0\,dt,
\end{equation} while for $\|v_t\|_0,\|v_{tt}\|_0$ and $\|v_{ttt}\|_0$, we commute $\p_t$ through $\p_t v=-\pak h+ge_3$ and obtain
\begin{equation}\label{v0norm}
\begin{aligned}
\|\p_t v(T)\|_0&\lesssim\|\p_t v(0)\|_0+\int_0^T \|\p_t^2 v(t)\|_0\dt\lesssim\|\p_t v(0)\|_0+\int_0^T P(\|\eta\|_{\hc}, \|v\|_3,\|\p h\|_{L^{\infty}},\|\p_t h\|_1)\dt\\
\|\p_t^2 v(T)\|_0&\lesssim\|\p_t^2 v(0)\|_0 +\int_0^T P(\|\eta\|_{\hc},\|\p_t v\|_1,\|\p h\|_{L^{\infty}},\|\p_t h\|_1, \|\p_t^2 h\|_1)\dt \\
\|\p_t^3 v(T)\|_0&\lesssim\|\p_t^3 v(0)\|_0 +\int_0^T P(\|\eta\|_{\hc},\|v\|_3,\|\p_t v\|_2,\|\p_t^2 v\|_1, \|h\|_{\hc},\|\p_t h\|_2, \|\p_t^2 h\|_1, \|\p_t^3 h\|_0)\dt.
\end{aligned}
\end{equation}

Now we are going to control the curl term. Recall that $-\pak h$ is the Eulerian gradient of $h$ whose Eulerian curl is 0. This motivates us to take Eulerian curl in the equation $\p_tv=-\pak h+ge_3$ to get
\begin{align}
\p_t(\curla v)_{\lambda}=\epsilon_{\lambda\mu\alpha}\p_t\ak^{\nu\mu}\p_\nu v^{\alpha},
\label{curl}
\end{align} 
where $(\curla X)_{\lambda}:=\epsilon_{\lambda\mu\alpha}\ak^{\nu\mu}\p_\nu X^{\alpha}$ is the Eulerian curl of $X$ and $\epsilon_{\lambda\mu\alpha}$ is the sign of the 3-permutation $(\lambda\mu\alpha)\in S_3.$
Taking $\p^3$ in the last equation and then taking inner product with $\p^3\curla v$, we get
\begin{equation}
\frac{1}{2}\frac{d}{dt}\io|\p^3\curla v|^2=\io(\p^3\curla v^{\lambda})\p^3(\epsilon_{\lambda\mu\alpha}\p_t\ak^{\nu\mu}\p_\nu v^{\alpha})\dy\lesssim P(\|v\|_4,\|\p^2\eta\|_2),
\end{equation} and thus
\begin{equation}\label{curlv31}
\|\curla v(T)\|_3\lesssim P(\|v_0\|_4)+\int_0^T P(\|v\|_4,\|\p^2\eta\|_2)\dt.
\end{equation}
The Lagrangian curl only differs from the Eulerian curl by a sufficiently small term which shall be absorbed in the LHS
\begin{equation}\label{curlv3}
\begin{aligned}
\|\curl v(T)\|_3&=\|\text{curl}_{I-\ak} v(T)\|_3+\|\curla v\|_3 \\
&\lesssim\|\text{Id}-\ak\|_3\|v\|_4+P(\|v_0\|_4)+\int_0^T P(\|v\|_4,\|\eta\|_{\hc})\dt\\
&\lesssim\epsilon\|v\|_4+P(\|v_0\|_4)+\int_0^T P(\|v\|_4,\|\eta\|_{\hc})\dt.
\end{aligned}
\end{equation}
Commuting $\p_t^{k}$ ($k=1,2,3$) though \eqref{curl}, we get the evolution equation for $\curla \p_t^k v$
\[
\p_t(\curla \p_t^{k}v)_{\lambda}=\epsilon_{\lambda\mu\alpha}\p_t^k(\p_t\ak^{\nu\mu}\p_\nu v^{\alpha})-\p_t([\p_t^k,\curla ]v)_{\lambda}.
\]
Commuting $\p^{3-k}$ through the above equation, and then taking $L^2$ inner product with $\p^{3-k}(\curla \p_t^{k}v)$, we get
\begin{equation}\label{curlavtk}
\begin{aligned}
\frac{1}{2}\frac{d}{dt}\io|\p^{3-k}\curla \p_t^k v|^2\dy&=\io \p^{3-k}\left(\epsilon_{\lambda\mu\alpha}\p_t^k(\p_t\ak^{\nu\mu}\p_\nu v^{\alpha})-\p_t([\p_t^k,\curla ]v)_{\lambda}\right)\p^{3-k}(\curla \p_t^{k}v)^{\lambda}\dy \\
&\lesssim \|\curla \p_t^kv\|_0\cdot\underbrace{\left\|\p^{3-k}\left(\epsilon_{\lambda\mu\alpha}\p_t^k(\p_t\ak^{\nu\mu}\p_\nu v^{\alpha})-\p_t([\p_t^k,\curla ]v)_{\lambda}\right)\right\|_0}_{=:D_k}.
\end{aligned}
\end{equation}
One can use 
\begin{equation*}
 \ak = [\p \ek]^{-1} = \tilde{J}^{-1}
 \begin{pmatrix}
 \p_2 \ek\times \p_3 \ek\\
 \p_3\ek\times \p_1\ek\\
 \p_1\ek\times\p_2 \ek
 \end{pmatrix},
\end{equation*} 
$\p_t\eta=v+\psi$, and the estimates for $\psi$ in Lemma \ref{psi} to control $D_k$ directly. Note that the leading order terms in $D_k$ are $\p^{3-k}\p_t^{k+1}\ak$ and $\p^{4-k}\p_t^{k}v$. Therefore, 
\[
\sum_{k=1}^3D_k\lesssim \sum_{k=1}^3 P(\|\p_t^k v\|_{4-k},\|\p_t^k\psi\|_{4-k},\|\eta\|_{\hc})\lesssim P(\EE_{\kk}),
\]
which implies
\begin{equation}\label{curlvtk}
\|\curl\p_t^{k} v\|_{3-k}\lesssim\epsilon\|\p_t^k v\|_{4-k}+\PP_0+\int_0^TP(\EE_\kk(t))\dt.
\end{equation}

For boundary terms in \eqref{hodgev}, we invoke the normal trace lemma (cf. Lemma \ref{normaltrace}) to get
\begin{equation}
\label{vbdry}|\TP^3(v\cdot N)|_{0.5}\lesssim \|\TP^4 v\|_0+\|\TP^3 \dive v\|_0.
\end{equation}
Similarly we have
\begin{align}
\label{vtbdry} |\TP^2 (\p_tv\cdot N)|_{0.5}&\lesssim  \|\TP^3\p_t v\|_0+\|\TP^2 \dive \p_t v\|_0 \\
\label{vttbdry} |\TP (\p_t^2v\cdot N)|_{0.5}&\lesssim  \|\TP^2\p_t^2 v\|_0+\|\TP \dive \p_t^2 v\|_0 \\
\label{vtttbdry} |\p_t^3 v\cdot N|_{0.5}&\lesssim  \|\TP\p_t^3 v\|_0+\|\dive \p_t^3 v\|_0.
\end{align} Therefore the boundary estimates are all reduced to divergence and tangential estimates.

Now we come to estimate the divergence. Recall that the Eulerian divergence $\diva X=\dive X+(\ak^{\mu\alpha}-\delta^{\mu\alpha})\p_\mu X_{\alpha}$, which together with \eqref{Jkk1} implies 
\begin{equation}\label{iasmall}
\begin{aligned}
\forall s>2.5&:~\|\dive X\|_{s-1}\lesssim\|\diva X\|_{s-1}+\|I-\ak\|_{s-1}\|X\|_s\lesssim\|\diva X\|_{s-1}+\epsilon\|X\|_s\\
\forall 1\leq s\leq 2.5&:~\|\dive X\|_{s-1}\lesssim\|\diva X\|_{s-1}+\|I-\ak\|_{L^{\infty}}\|X\|_s\lesssim\|\diva X\|_{s-1}+\epsilon\|X\|_s.
\end{aligned}
\end{equation} The $\epsilon$-terms can be absorbed by $\|X\|_s$ on LHS by choosing $\epsilon>0$ sufficiently small. So it suffices to estimate the Eulerian divergence which satisfies $\diva v=-\p_t e(h)$. Taking time derivatives in this equation, we get
\[
\diva \p_t^k v=-\p_t^{k+1}e(h)-[\p_t^k,\ak^{\mu\alpha}]\p_\mu v_{\alpha},~~k=0,1,2,3.
\]

The leading order terms in $\diva \p_t^k v$ are $e'(h)\p_t^k h\p_t h$, $\p_t^k\ak^{\mu\alpha}\p_\mu v_{\alpha}$ and $\p_t\ak^{\mu\alpha}\p_{\mu}\p_t^{k-1}v_{\alpha}$ when $k\geq 1$. Therefore, we have 
\begin{equation}\label{divav1}
\begin{aligned}
\|\diva v\|_3&\lesssim\|e'(h)\p_t h\|_3\\
\|\diva \p_t v\|_2&\lesssim \|e'(h)\p_t^2 h\|_2\|\p_t h\|_2+\|\p v\|_2^2\lesssim P(\|e'(h)\p_t^2 h\|_2,\|\p_t h\|_2,\|v\|_3) \\
&\lesssim P(\|e'(h)\p_t^2 h\|_2,\|\p_t h\|_2)+P(\|v_0\|_3)+\int_0^TP(\|\p_t v(t)\|_3)\dt
\end{aligned}
\end{equation}
and similarly,
\begin{equation}\label{divav2}
\begin{aligned}
 \|\diva \p_t^2 v\|_1&\lesssim P(\|e'(h)\p_t^3 h\|_1,\|e'(h)\p_t^2 h\|_2,\|\p_th\|_2,\|v\|_3,\|\p_t v\|_2)\\
&\lesssim P(\|e'(h)\p_t^3 h\|_1,\|e'(h)\p_t^2 h\|_2,\|\p_th\|_2)+\PP_0+\int_0^T P(\|\p_t v\|_3,\|\p_t^2 v\|_2)\dt\\
\|\diva \p_t^3 v\|_0&\lesssim P(\|e'(h)\p_t^4 h\|_0,\|e'(h)\p_t^3 h\|_1,\|e'(h)\p_t^2 h\|_2,\|\p_th\|_2,\|v\|_3,\|\p_t v\|_2,\|\p_t^2 v\|_1) \\
&\lesssim P(\|e'(h)\p_t^4 h\|_0,\|e'(h)\p_t^3 h\|_1,\|e'(h)\p_t^2 h\|_2,\|\p_th\|_2)+\PP_0+\int_0^T P(\|\p_t v\|_3,\|\p_t^2 v\|_2, \|\p_t^3 v\|_1)\dt.
\end{aligned}
\end{equation}

Combining \eqref{iasmall} and \eqref{divav1}-\eqref{divav2}, we know the divergence estimates are all be reduced to the estimates of $h$ which has been done in Section \ref{hapriori}. By choosing $\epsilon>0$ in \eqref{iasmall} to be sufficiently small, and using the estimates of $h$ in \eqref{h4energy}, we finally finish the divergence estimates

\begin{equation}\label{divvtk}
\sum_{k=0}^3\|\dive \p_t^k v\|_{3-k}\lesssim \PP_0+\int_0^T P(\EE_{\kk}(t))\dt.
\end{equation}

\subsection{Estimates for time derivatives of $v$}\label{tgtime}

As a result of div-curl estimates, it suffices to estimate the $L^2$-norms of $\TP^4v,\TP^3\p_t v,\cdots, \p_t^4 v$. In this part we are going to do the tangential estimates for the time derivatives of $v$, in order to finish the control $\|\p_t^{k}v\|_{4-k}$ with $k\geq 1$. The fact that $\p^2\eta$ and $\p_t\p^2\eta$ are of the same spatial regularity in Sobolev norms is essential for us to close the estimates. 

Let $\dd^4=\p_t^4,\p_t^3\TP,\p_t^2\TP^2,\p_t\TP^3$. First we compute
\begin{equation}\label{tgt0}
\begin{aligned}
\frac{d}{dt}\frac{1}{2}\io |\dd^4 v|^2\dy&=\io\dd^4 v_{\alpha} \dd^4\p_tv^{\alpha} \dy=-\io\dd^4 v_{\alpha}\dd^4(\ak^{\mu\alpha}\p_{\mu} h)\dy\\
&=-\io(\dd^4 v_{\alpha})\ak^{\mu\alpha}(\p_{\mu}\dd^4 h)\dy -\underbrace{\io\dd^4 v_{\alpha} ([\dd^4,\ak^{\mu\alpha}]\p_\mu h)\dy}_{L_1}.
\end{aligned}
\end{equation}
In the first integral above, we integrate $\p_{\mu}$ by parts and invoking the equation $\diva v=-e'(h)\p_t h$ to obtain:
\begin{equation}\label{tgt1}
\begin{aligned}
&~~~~-\io(\dd^4 v_{\alpha})\ak^{\mu\alpha}(\p_{\mu}\dd^4 h)\dy \\
&=-\ig \dd^4 v_{\alpha} \ak^{\mu\alpha}N_{\mu}\underbrace{\dd^4 h}_{=0}dS\underbrace{-\io ([\dd^4,\ak^{\mu\alpha}]\p_\mu v_{\alpha})\dd^4 h\dy}_{L_2}+\underbrace{\io \dd^4 v_{\alpha}\p_\mu \ak^{\mu\alpha}\dd^4 h\dy}_{L_3}+\io \dd^4 \diva v \dd^4 h\dy \\
&=-\io \dd^4(e'(h)\p_t h) \dd^4 h \dy+L_2+L_3\\
&=-\frac{d}{dt}\frac{1}{2}\io e'(h)|\dd^4 h|^2\dy+\underbrace{\io e''(h)\p_t h|\dd^4 h|^2-[\dd^4,e'(h)]\p_t h\dd^4 h\dy}_{L_4}+L_2+L_3.
\end{aligned}
\end{equation}
It is not difficult to see $L_3$ and $L_4$ can be controlled directly:
\begin{align}
\label{tgl3} L_3&\lesssim\|\dd^4 v\|_0\|\p a\|_2\|\dd^4 h\|_{0}\lesssim P(\EE_{\kk}(t)), \\
\label{tgl4} \sum_{\dd^4}L_4&\lesssim\|\sqrt{e'(h)}\dd^4 h\|_0 P\left(\sum_{k\geq 1}\|\sqrt{e'(h)}
\p_t^{k}\p^{4-k} h\|_0\right)\lesssim P(\EE_{\kk}(t)).
\end{align}
To estimate $L_1$ and $L_2$, it suffices to control the commutator $[\dd^4,\ak]f$ in $L^2$-norm.
\begin{align*}
\|[\dd^4,\ak]f\|_0&=\|(\dd^4 a)f+4(\dd^3 a)(\dd f)+6(\dd^2 a)(\dd^2 f)+4(\dd a)(\dd^3 f)\|_{0} \\
&\lesssim \|\dd^4 a\|_0\|f\|_{L^{\infty}}+\|\dd^3 a\|_1\|\dd f\|_1+\|\dd^2 a\|_1\|\dd^2 f\|_1+\|\dd a\|_2\|\dd^3 f\|_0.
\end{align*}
Let $f=\p v$ and $\p h$ respectively (corresponding to $L_1$ and $L_2$), and recall $a=[\p\eta]^{-1}$. By Lemma \ref{etapsi}, we have
\begin{equation}\label{tgl1l2}
\begin{aligned}
L_1&\lesssim \|\dd^4 v\|_0(\|\dd^4 a\|_0\|\p v\|_2+\|\dd^3 a\|_1\|\dd \p v\|_1+\|\dd^2 a\|_1\|\dd \p v\|_1+\|\dd a\|_2\|\dd^3 \p v\|_0) \\
&\lesssim P(\sum_{k=1}^4 \|\p_t^k v\|_{4-k},\|\p^2\eta\|_2, \|v\|_4)\lesssim P(\EE_{\kk}(t));\\
L_2&\lesssim\|\dd^4 h\|_0(\|\dd^4 a\|_0\|\p h\|_{L^{\infty}}+\|\dd^3 a\|_1\|\dd \p h\|_1+\|\dd^2 a\|_1\|\dd \p h\|_1+\|\dd a\|_2\|\dd^3 \p h\|_0) \\
&\lesssim P(\sum_{k=0}^3 \|\p_t^k v\|_{4-k},\|\p h\|_{L^\infty},\|\p_t h\|_1,\|\p_t^2 h\|_1, \|\p_t^3 h\|_0, \|\p^2\eta\|_2)\lesssim P(\EE_{\kk}(t)).
\end{aligned}
\end{equation}
Summing up \eqref{tgt0}-\eqref{tgl1l2}, we are able to get the energy bound
\begin{equation}\label{tgt}
\frac{d}{dt}\frac{1}{2}\left(\sum_{k=1}^4\|\p_t^k\TP^{4-k} v\|_0^2+\|\sqrt{e'(h)}\p_t^k \TP^{4-k} h\|_0^2\right)\lesssim P(\EE_{\kk}(t)).
\end{equation}

\subsection{Estimates for spatial derivatives of $v$: Alinhac's good unknown method}\label{tgspace}

Now it remains to control $\|\TP^4 v\|_0^2$ to close the a priori estimates of the approximation system \eqref{app1}. It should be emphasized here that our method in Section \ref{tgtime} cannot be used in the full spatial derivatives, because the $L^2$-norm of the commutator $[\TP^4,\ak](\p v)$ and  $[\TP^4,\ak](\p h)$ cannot be controlled due to the lack of time derivatives. To overcome such difficulty, we introduce Alinhac's good unknowns for both $v$ and $h$, which actually uncover that the essential leading order terms in $\TP^4\pak v$ and $\TP^4\pak h$ is exactly the covariant derivative $\pak$ of their Alinhac's good unknowns. As a result, one can commute $\TP^4$ and $\pak$ in the energy estimate without producing any higher order commutator apart from $\TP^4(\diva v)\TP^4 h$. However, the third equation of \eqref{app1} yields $\TP^4(\diva v)\TP^4 h=-\TP^4(e'(h)\p_t h)\TP^4 h$, which gives the energy term $-\frac{1}{2}\frac{d}{dt}\io\|\TP^4 h\|_0^2$ and thus no extra higher order term appears. This being said, the use of Alinhac's good unknowns avoids the control of the 5-th order wave equation of $h$ together with delicate elliptic estimates, e.g., Lindblad-Luo \cite{lindblad2018priori}, Luo \cite{luo2018ww}, Ginsberg-Lindblad-Luo \cite{GLL2019LWP}.

 The Alinhac's good unknown was first introduced by Alinhac \cite{alinhacgood89}, and has been frequently used in the study of free-boundary problems of incompressible fluids because the incompressibility condition (Eulerian divergence-free) eliminates the only extra term $\TP^r(\diva v)=0$, e.g., Masmoudi-Rousset \cite{MRgood2017}, Gu-Wang \cite{gu2016construction}, Wang-Xin \cite{wang2015good}, etc. 
On the other hand, in free-boundary problems of compressible fluids, the Alinhac's good unknowns were crucial in \cite{trakhiningas2009} together with the Nash-Moser iteration.  Moreover, there are several studies for free-boundary problems in ideal compressible MHD equations in which the passage to the Alinhac's good unknowns is used to study the linearized problem in the framework of the Eulerian approach. For example, in this connection, we refer to \cite{Chen-Wang08, trakhinin2005, trakhinin2009} for compressible current-vortex sheets, and \cite{ST2014, TW2021, TW2021A} for the plasma-vacuum interface problem in compressible MHD.

\subsubsection{Introducing Alinhac's good unknowns}

For simplicity we replace $\TP^4$ by $\tpl$ which is more convenient for us to deal with the correction term $\psi$ on the boundary. For a smooth function $g=g(t,x)$, we define its ``Alinhac's good unknown" (for the 4-th order derivative) to be 
\begin{align}\label{alinhac}
\GG:=\tpl g-\tpl\ek\cdot\pak g=\tpl g-\tpl \ek_{\beta}\ak^{\mu\beta}\p_\mu g,
\end{align}
which enjoys the following important properties.
\begin{lem} \label{lem AGU property}
We have
\begin{equation}\label{alinhaca}
\tpl(\pak^{\alpha}g)=\pak^{\alpha}\GG+C^{\alpha}(g)
\end{equation}
with 
\begin{equation}\label{alinhacc}
\|C^{\alpha}(g)\|_0\lesssim P(\|\eta\|_{\hc})\|g\|_{\hc}.
\end{equation}
\end{lem}
\begin{proof}
Invoking the identity 
\begin{align}\label{derivative a}
\p \ak^{\mu\alpha} = -\ak^{\mu\gamma}\p\p_\beta \ek_\gamma \ak^{\beta\alpha}, 
\end{align}
which is obtained from differentiating $\ak^{\mu\alpha} \p_\mu \eta_\beta =\delta^{\alpha}_\beta$, then
\begin{align*}
\tpl(\pak^{\alpha}g)&=\pak^{\alpha}(\tpl g)+(\tpl\ak^{\mu\alpha})\p_\mu g+[\tpl,\ak^{\mu\alpha},\p_{\mu} g] \\
&=\pak^{\alpha}(\tpl g)-\TP\TL(\ak^{\mu\gamma}\TP\p_{\beta}\ek_{\gamma}\ak^{\beta\alpha})\p_\mu g+[\tpl,\ak^{\mu\alpha},\p_{\mu} g] \\
&=\pak^{\alpha}(\tpl g)-\ak^{\beta\alpha}\p_{\beta}\tpl\ek_{\gamma}\ak^{\mu\gamma}\p_\mu g-([\TP\TL,\ak^{\mu\gamma}\ak^{\beta\alpha}]\TP\p_{\beta}\ek_{\gamma})\p_\mu g+[\tpl,\ak^{\mu\alpha},\p_{\mu} g] \\
&=\underbrace{\pak^{\alpha}(\tpl g-\tpl \eta_{\gamma}\ak^{\mu\gamma}\p_\mu g)}_{=\pak^{\alpha}\GG}+\underbrace{\tpl \eta_{\gamma}\pak^{\alpha}(\pak^{\gamma} g)-([\TP\TL,\ak^{\mu\gamma}\ak^{\beta\alpha}]\TP\p_{\beta}\ek_{\gamma})\p_\mu g+[\tpl,\ak^{\mu\alpha},\p_{\mu} g] }_{=:C^{\alpha}(g)},
\end{align*} where $[\tpl,f,g]:=\tpl(fg)-\tpl(f)g-f\tpl(g)$.
A direct computation yields that
\begin{align*}
\|\tpl \eta_{\gamma}\pak^{\alpha}(\pak^{\gamma} g)\|_0&\lesssim P(\|\p^2\eta\|_2)(\|\p g\|_{L^{\infty}}+\|\p^2 g\|_2);\\
\|([\TP\TL,\ak^{\mu\gamma}\ak^{\beta\alpha}]\TP\p_{\beta}\ek_{\gamma})\p_\mu g\|_0&\lesssim\|[\TP\TL,\ak^{\mu\gamma}\ak^{\beta\alpha}]\TP\p_{\beta}\ek_{\gamma}\|_0\|g\|_{W^{1,\infty}}\lesssim P(\|\eta\|_{\hc})(\|\p g\|_{L^{\infty}}+\|\p^2 g\|_1) \\
\|[\tpl,\ak^{\mu\alpha},\p_{\mu} g]\|_0&\lesssim P(\|\eta\|_{\hc})(\|\p g\|_{L^{\infty}}+\|\p^2 g\|_2).
\end{align*}
\end{proof}
Moreover, 
$\|\GG\|_0^2$ controls $\|\TP^4 g\|_0^2$ modulo a controllable error term. Specifically, 
\begin{lem}\label{lem AGU controls 4 tangential}
We have
\begin{align}\label{alinhac4}
\|\TP^4 g(T)\|_0^2 \lesssim\|\GG(T)\|_0^2+\int_0^T P(\|\eta\|_{\hc}, \|v\|_4, \|\p g\|_{L^\infty}, \|\p_t g\|_{L^\infty}).
\end{align}
\end{lem} 
\begin{proof}
The definition of $\GG$ implies
\begin{equation*}
\|\TP^4 g(T)\|_0\approx\|\tpl g(T)\|_0\lesssim\|\GG(T)\|_0+\left\|\Big(\tpl\ek_\beta\ak^{\mu\beta}\p_\mu g\Big)\big|_{t=T}\right\|_0, 
\end{equation*}
where
$$
\left\|\tpl\ek_\beta\ak^{\mu\beta}\p_\mu g\big|_{t=T}\right\|_0\leq \left\| \tpl\ek_\beta\ak^{\mu\beta}\p_\mu g\big|_{t=0}\right\|_0+\int_0^T \|\p_t (\tpl\ek_\beta\ak^{\mu\beta}\p_\mu g)\|_0.
$$
Here, $\tpl\ek_\beta\ak^{\mu\beta}\p_\mu g |_{t=0} =0$ because $\tpl \eta_0=0$. For the integrand of the second term, invoking \eqref{derivative a} with $\p=\p_t$, we have
\begin{align*}
\p_t  (\tpl\ek_\beta\ak^{\mu\beta}\p_\mu g) = \tpl\tilde{v}_\beta\ak^{\mu\beta}\p_\mu g+\tpl\ek_\beta\ak^{\mu\beta}\p_\mu\p_tg - \tpl\ek_\beta \ak^{\mu\gamma} \p_\tau \tilde{v}_\gamma \ak^{\tau\beta}\p_\mu g
\end{align*}
whose $L^2$-norm can be controlled by $P(\|\eta\|_{\hc}, \|v\|_4, \|\p g\|_{L^\infty}, \|\p_t g\|_{L^\infty}).$
\end{proof}
\begin{rmk}
For general initial data $\eta_0$, the term $P(\|\eta_0\|_{\hc})\|\p g(0)\|_{L^{\infty}}$ should also appear on the RHS of \eqref{alinhac4}. Specifically, 
\[
\left\|\tpl\ek_\beta\ak^{\mu\beta}\p_\mu g |_{t=0}\right\|_0\lesssim \|\tpl\eta_0\|_0\|\p\eta_0\|_{L^{\infty}}^2\|\p g(0)\|_{L^{\infty}} \leq P(\|\eta_0\|_{\hc})\|\p g(0)\|_{L^{\infty}}.
\]
\end{rmk}

\subsubsection{Tangential estimates of $v$: Interior part}

Now we introduce the Alinhac's good unknowns for $v$ and $h$ 
\begin{align}
\label{vgood} \VV&:=\tpl v-\tpl\ek\cdot\pak v\\
\label{hgood} \HH&:=\tpl h-\tpl\ek\cdot\pak h.
\end{align}
Applying $\tpl$ to the second equation in system \eqref{app1} and then using \eqref{vgood}, \eqref{hgood} to get
\begin{equation}\label{goodapp1}
\p_t\VV=-\pak \HH+\underbrace{\p_t(\tpl\ek\cdot\pak v)-C(h)}_{=:\FF},
\end{equation} subject to the boundary condition
\begin{equation}\label{bdrygood}
\HH=-\tpl\ek_{\beta}\ak^{3\beta}\p_3 h~~~\text{on }\Gamma,
\end{equation}
with the continuity equation
\begin{equation}\label{divgood}
\pak\cdot\VV=\tpl (\diva v)-C^{\alpha}(v_{\alpha})~~~\text{in }\Omega.
\end{equation}

Thanks to Lemma \ref{lem AGU controls 4 tangential}, it suffices to bound $\|\VV\|_0^2+\|\HH\|_0^2$ to close the estimates for $\|\TP^4 v\|_0^2+\|\TP^4 h\|_0^2$. Taking $L^2$ inner product between \eqref{goodapp1} and $\VV$, one gets

\begin{equation}\label{tgs1}
\frac{1}{2}\frac{d}{dt}\io|\VV|^2\dy=-\io\pak\HH\cdot\VV\dy+\io\FF\cdot\VV\dy.
\end{equation}
The second term on the RHS of \eqref{tgs1} can be directly controlled
\begin{equation}\label{tgs2}
\begin{aligned}
\io\FF\cdot\VV\dy&\leq (\|\p_t(\tpl\ek\cdot\pak v)\|_0+\|C(h)\|_0)\|\VV\|_0 \\
&\lesssim (P(\|\tpl\eta\|_0,\|\tpl v\|_0,\|\p\eta\|_{L^{\infty}},\|v\|_3,\|\p_t v\|_3)+P(\|\eta\|_{\hc})\| h\|_{\hc})\|\VV\|_0 \\
&\lesssim P(\|\eta\|_{\hc},\|v\|_4,\|\p_t v\|_3, \|h\|_{\hc}).
\end{aligned}
\end{equation}

For the first term in RHS of \eqref{tgs1}, we integrate by part and use \eqref{bdrygood}, \eqref{divgood} to get
\begin{equation}\label{IKL0}
\begin{aligned}
-\io \pak \HH\cdot \VV&=-\io \ak^{\mu\alpha}\p_\mu \HH\cdot \VV_{\alpha}\dy \\
&=-\ig \HH(\ak^{\mu\alpha}N_{\mu}\VV_{\alpha})dS+\io \HH(\pak\cdot\VV)\dy+\io(\p_\mu \ak^{\mu\alpha})\HH\VV_{\alpha}\dy\\
&=\ig\p_3h\tpl\ek_{\beta}\ak^{3\beta}\ak^{3\alpha}\VV_{\alpha} dS+\io \HH\tpl (\diva v)\dy \underbrace{-\io \HH C^{\alpha}(v_{\alpha})+\io(\p_\mu \ak^{\mu\alpha})\HH\VV_{\alpha}\dy}_{L_0}\\
&=:I+K+L_0.
\end{aligned}
\end{equation}
First, $L_0$ can be directly controlled by $P(\EE_{\kk})$ by using \eqref{alinhacc}
\begin{equation}\label{L0}
L_0\lesssim\|\HH\|_0P(\|\eta\|_{\hc})\|v\|_4+\|\p\ak\|_2\|\HH\|_0\|\VV\|_0\lesssim P(\|\eta\|_{\hc},\|v\|_4,\|h\|_{\hc}).
\end{equation}
Then we use $\diva v=-e'(h)\p_t h$ to bound $K$
\begin{equation}\label{K}
\begin{aligned}
K&:=\io \HH\tpl (\diva v)\dy=-\io (\tpl h-\tpl\ek\cdot\pak h)\tpl(e'(h)\p_t h)\dy\\
&=-\frac{d}{dt}\frac{1}{2}\io e'(h)|\tpl h|^2\dy+\frac{1}{2}\io e''(h)\p_t h|\tpl h|^2\dy+\io \HH ([\tpl,e'(h)]\p_t h)\dy\\
&~~~~\underbrace{+\io e'(h)(\tpl\p_t h)\tpl\ek\cdot\pak h\dy}_{:=K^*}\\
&\lesssim-\frac{d}{dt}\frac{1}{2}\io e'(h)|\tpl h|^2\dy+K^*+P(\|\p^2\eta\|_2,\|v\|_4,\|h\|_{\hc},\|\p_t h\|_3).
\end{aligned}
\end{equation}
The term $K^*$ cannot be bounded directly because it contains a higher order term $\tpl\p_t h$, but we can consider its time integral and integrate $\p_t$ by parts, then using $\epsilon$-Young inequality to absorb the $\epsilon$-term.

\begin{equation}\label{K*}
\int_0^T K^*(t)\dt=-\int_0^T\io e'(h)(\tpl\p_t h)(\tpl\ek\cdot\pak h)\dy\dt=-\int_0^T\io e'(h)(\p_t\tpl h)(\tpl\ek\cdot\pak h)\dy\dt
\end{equation}
Integrating $\p_t$ by parts, we get
\begin{align*}
-\int_0^T\io e'(h)(\p_t\tpl h)(\tpl\ek\cdot\pak h)\dy\dt= -\io e'(h)(\tpl h)(\tpl\ek\cdot\pak h)\dy\bigg|^{t=T}_{t=0}\\
+\int_0^T\io (e'(h)\tpl h)\p_t(\tpl\ek\cdot\pak h)\dy.
\end{align*}
The second term on the RHS is controlled directly by $\int_0^T P(\|h\|_{\hc},\|\p^2\eta\|_2,\|v\|_4)\dt$. For the first term on the RHS, we have
\begin{align*}
\left|-\io e'(h)(\tpl h)(\tpl\ek\cdot\pak h)\dy\bigg|^{t=T}_{t=0}\right|
\lesssim  P(\|h_0\|_{\hc},\|v_0\|_4)+ \|e'(h)\tpl h(T)\|_0\|\p^2\eta(T)\|_2\|\pak h(T)\|_{L^{\infty}}.
\end{align*}

Using $\epsilon$-Young's inequality, we have
\begin{align*}
&~~~~\|e'(h)\tpl h(T)\|_0\|\p^2\eta(T)\|_2\|\pak h(T)\|_{L^{\infty}}\\
&\leq \epsilon \|e'(h)\tpl h(T)\|_0^2+\frac{1}{8\epsilon}(\|\p^2\eta(T)\|_2^4+\|\pak h(T)\|_{L^{\infty}}^4) \\
&\lesssim\epsilon \|e'(h)\tpl h(T)\|_0^2+\left(\PP_0+\int_0^TP(\|v\|_4,\|\p\eta\|_{L^{\infty}},\|h\|_{\hc},\|\p_t h\|_3)\dt\right).
\end{align*}
Therefore, 
\begin{equation}\label{tgK}
\begin{aligned}
\int_0^T K^{*}(t)\dt\lesssim \epsilon \|e'(h)\tpl h(T)\|_0^2 +\PP_0+\int_0^TP(\|v\|_4,\|\eta\|_{\hc},\|h\|_{\hc},\|\p_t h\|_3)\dt.
\end{aligned}
\end{equation}
Here, $\epsilon \|e'(h)\tpl h(T)\|_0^2$ can be moved to the LHS when $\epsilon$ is sufficiently small. This concludes the control of $K$. 
\subsubsection{Tangential estimates of $v$: Boundary part}\label{bdrycancel}

Now it remains to control the boundary term $I$, where the Taylor sign boundary term in $\EE_{\kk}$ is produced and the correction term $\psi$ exactly eliminates the extra out-of-control terms produced by the tangential smoothing (these terms are 0 if $\kk=0$).

\begin{equation}\label{I0}
\begin{aligned}
I&=\ig \p_3 h\ak^{3\alpha}\ak^{3\beta}\tpl \ek_{\beta}\VV_{\alpha}~dS\\
&=\ig \p_3 h\ak^{3\alpha}\ak^{3\beta}\tpl \ek_{\beta}(\tpl v_{\alpha}-\tpl\ek\cdot\pak v_{\alpha})~dS\\
&=\ig \p_3 h\ak^{3\alpha}\ak^{3\beta}\tpl \ek_{\beta}(\tpl\p_t\eta_{\alpha}-\tpl\psi-\tpl\ek\cdot\pak v_{\alpha})~dS.
\end{aligned}
\end{equation}

We construct the Taylor-sign term in the energy functional $\EE_{\kk}$ from the first term.

\begin{equation}\label{I1}
\begin{aligned}
&~~~~\ig \p_3 h\ak^{3\alpha}\ak^{3\beta}\tpl \ek_{\beta}\tpl\p_t\eta_{\alpha}~dS \\
&=\ig\p_3 h\ak^{3\alpha}\ak^{3\beta}\tpl \lkk\eta_{\beta}\tpl\p_t\lkk\eta_{\alpha}\dS \\
&~~~~+\ig (\tpl\lkk\eta_{\beta})([\lkk,\p_3 h\ak^{3\alpha}\ak^{3\beta}]\tpl\p_t\eta_{\alpha})\dS \\
&=\frac{d}{dt}\frac{1}{2}\ig \p_3 h |\ak^{3\alpha}\tpl\lkk\eta_{\alpha}|^2\dS-\frac{1}{2}\ig\p_t\p_3 h |\ak^{3\alpha}\tpl\lkk\eta_{\alpha}|^2\dS \\
&~~~~\underbrace{-\ig \p_3 h \ak^{3\beta}\tpl\lkk\eta_{\beta}\p_t\ak^{3\alpha}\tpl\lkk\eta_{\alpha}\dS}_{B_1}+\underbrace{\ig (\tpl\lkk\eta_{\beta})([\lkk,\p_3 h\ak^{3\alpha}\ak^{3\beta}]\tpl\p_t\eta_{\alpha})\dS}_{LB_1}.
\end{aligned}
\end{equation}
In $LB_1$, we integrate $\TP^{0.5}$ by parts (by interpreting it in the Fourier sense) and then use Sobolev trace lemma, \eqref{lkk6} and Lemma \ref{etapsi} to get
\begin{equation}\label{LB1}
\begin{aligned}
LB_1&=\ig (\TP^{1.5}\TL\lkk\eta_{\beta})\TP^{0.5}([\lkk,\p_3 h\ak^{3\alpha}\ak^{3\beta}]\TP(\TP\TL\p_t\eta_{\alpha}))\dS \\
&\lesssim\|\p^2\eta\|_2|\p_3 h\ak^{3\alpha}\ak^{3\beta}|_{W^{1,\infty}}|\TP\TL\p_t\eta_{\alpha}|_{0.5} \\
&\lesssim\|\p^2\eta\|_2\|\p^2 h\|_2\|\p\ak\|_{L^{\infty}} \|v+\psi\|_4\lesssim P(\|\eta\|_{\hc},\|v\|_3,\|\p^2 h\|_2).
\end{aligned}
\end{equation}

Next, we plug $\p_t \ak^{3\alpha}=-\ak^{3\gamma}\p_\mu\p_t\ek_{\gamma}\ak^{\mu\alpha}$ into $B_1$ and then separate the normal derivative of $\ek_{\gamma}$ from tangential derivatives.
\begin{equation}\label{B1}
\begin{aligned}
B_1&=\underbrace{\ig \p_3 h \ak^{3\beta}\tpl\lkk\eta_{\beta}\ak^{3\gamma}\p_3\p_t\ek_{\gamma}\ak^{3\alpha}\tpl\lkk\eta_{\alpha}\dS}_{LB_2}+\ig \p_3 h \ak^{3\beta}\tpl\lkk\eta_{\beta}\ak^{3\gamma}\TP_i\p_t\ek_{\gamma}\ak^{i\alpha}\tpl\lkk\eta_{\alpha}\dS \\
&=LB_2+\underbrace{\ig \p_3 h \ak^{3\beta}\tpl\lkk\eta_{\beta}\ak^{3\gamma}\TP_i\lkk^2\psi_{\gamma}\ak^{i\alpha}\tpl\lkk\eta_{\alpha}\dS}_{LB_3} \\
&~~~~~~~~~~~~~+\underbrace{\ig \p_3 h \ak^{3\beta}\tpl\lkk\eta_{\beta}\ak^{3\gamma}\TP_i\lkk^2v_{\gamma}\ak^{i\alpha}\tpl\lkk\eta_{\alpha}\dS}_{B_1^*}.
\end{aligned}
\end{equation}
$LB_2$ can be directly bounded
\begin{equation}\label{LB2}
\begin{aligned}
LB_2&\lesssim |\ak^{3\beta}\tpl\lkk\eta_{\beta}|_0^2|\p_3h\ak^{3\gamma}\p_3\p_t\ek_{\gamma}|_{L^{\infty}}\\
&\lesssim |\ak^{3\beta}\tpl\lkk\eta_{\beta}|_0^2 P(\|h\|_{\hc},\|v\|_3,\|\p\eta\|_{L^{\infty}})\lesssim P(\EE_{\kk}).
\end{aligned}
\end{equation}
In $LB_3$, the term $\tpl\lkk\eta_{\alpha}$ cannot be directly bounded, but we can use \eqref{lkk2} in Lemma \ref{tgsmooth} to control this term by $(1/\sqrt{\kk})|\eta|_{3.5}$.

\begin{align}
LB_3&\lesssim|\p_3 h\ak^{3\gamma}\ak^{i\alpha}|_{L^{\infty}}|\ak^{3\beta}\tpl\lkk\eta_{\beta}|_0|\TP\lkk^2\psi|_{L^{\infty}}\frac{1}{\sqrt{\kk}}|\TP\TL\eta|_{0.5} \nonumber\\
&\lesssim\frac{1}{\sqrt{\kk}} P(\|\p^2\eta\|_2, \|v\|_3,\|h\|_{\hc})|\ak^{3\beta}\tpl\lkk\eta_{\beta}|_0|\TP\psi|_{L^{\infty}}. 
\label{LB3 start}
\end{align}
The factor $1/\sqrt{\kk}$ can be eliminated by plugging the expression of $\psi$ in \eqref{psi}. We apply Sobolev embedding $W^{1,4}(\R^2)\hookrightarrow L^{\infty}(\R^2)$ first, and note that $\TP\psi=P_{\geq 1}(\TP\TL^{-1}(\cdots))$ does not contain the low-frequency part, which (actually follows from the Littlewood-Paley characterization of $W^{1,4}$ and $\dot{W}^{1,4}$) implies $|\TP\psi|_{W^{1,4}}\approx|\TP\psi|_{\dot{W}^{1,4}}\approx|\TL\psi|_{L^4}$. Hence, we have
\[
|\TP\psi|_{L^{\infty}}\lesssim |\TL\psi|_{L^4}= \left|\mathbb{P}\underbrace{\left(\TL\eta_{\beta}\ak^{i\beta}\TP_i\lkk^2 v-\TL\lkk^2\eta_{\beta}\ak^{i\beta}\TP_i v\right)}_{f}\right|_{L^4}. 
\]
According to the Littlewood-Paley characterization of $L^4(\R^2)$ and the almost orthogonality property, we know
\begin{align*}
|\mathbb{P}f|_{L^4}\approx\left|\left(\sum_{N\in\Z}|\tilde{P}_NP_{\geq 1}f |^2\right)^{1/2}\right|_{L^4}\approx\left|\left(\sum_{N\geq 0}|\tilde{P}_N f |^2\right)^{1/2}\right|_{L^4}\lesssim\left|\left(\sum_{N\in\Z}|\tilde{P}_N f |^2\right)^{1/2}\right|_{L^4}\approx|f|_{L^4},
\end{align*}where $\tilde{P}$ is the Littlewood-Paley projection with respect to $\tilde{\chi}(\cdot):=\chi(2\cdot)$. 

\begin{rmk}
For more details of Littlewood-Paley characterization of Sobolev spaces, we refer readers to Chapter 1.3 in Grafakos \cite{GTM250} or Appendix A in Tao \cite{tao2006nonlinear}.
\end{rmk}

Hence, we have
\begin{align*}
|\TP\psi|_{L^{\infty}}&\lesssim\left|\TL\eta_{\beta}\ak^{i\beta}\TP_i\lkk^2 v-\TL\lkk^2\eta_{\beta}\ak^{i\beta}\TP_i v\right|_{L^4}\\
&\lesssim \left|\TL(\eta_{\beta}-\lkk^2\eta_{\beta})\ak^{i\beta}\TP_i\lkk^2 v-\TL\lkk^2\eta_{\beta}\ak^{i\beta}\TP_i(v-\lkk^2 v)\right|_{L^4}\\
&\lesssim |\TL(\eta_{\beta}-\lkk^2\eta_{\beta})|_{L^{\infty}}|\ak^{i\beta}|_{L^{\infty}}|\TP_i\lkk^2 v|_{0.5}+|\TL\ek_{\beta}|_{0.5}|\ak^{i\beta}|_{L^{\infty}}|\TP(v-\lkk v)|_{L^{\infty}} ,
\end{align*} where in the last step we use $H^{0.5}\hookrightarrow L^4$ in $\R^2.$ Now, recall \eqref{lkk3} in Lemma \ref{tgsmooth} that we are able to control $|\TL\eta_{\beta}-\lkk^2\TL\eta|_{L^{\infty}}$ by $\sqrt{\kk}|\TL\eta|_{1.5}\leq\sqrt{\kk}\|\p^2\eta\|_2.$ Similarly, $|\TP(v-\lkk v)|_{L^{\infty}}\lesssim\sqrt{\kk}|\TP v|_{1.5}\lesssim \sqrt{\kk}\|v\|_3$. Therefore, one has 
\[
|\TP\psi|_{L^{\infty}}\lesssim \sqrt{\kk} P(\|\eta\|_{\hc},\|v\|_3),
\]and 

\begin{equation}\label{LB3}
LB_3\lesssim P(\|\eta\|_{\hc},\|v\|_3,\|h\|_{\hc})|\ak^{3\beta}\tpl\lkk\eta_{\beta}|_0\lesssim P(\EE_{\kk}).
\end{equation}

As for $B_1^*$, it cannot be directly bounded, but together with another  term they will be exactly eliminated by the correction term in \eqref{I0}.

Now we start to control the third term in \eqref{I0}. Again we separate the normal derivative of $v$ from tangential derivatives
\begin{equation}\label{I3}
\begin{aligned}
&~~~~-\ig \p_3 h\ak^{3\alpha}\ak^{3\beta}\tpl \ek_{\beta}\tpl\ek\cdot\pak v_{\alpha}~dS \\
&=-\ig \p_3 h\ak^{3\alpha}\ak^{3\beta}\tpl \ek_{\beta}\tpl\ek_{\gamma}\ak^{3\gamma} \p_3v_{\alpha}~dS \underbrace{-\ig \p_3 h\ak^{3\alpha}\ak^{3\beta}\tpl \ek_{\beta}\tpl\ek_{\gamma}\ak^{i\gamma} \TP_iv_{\alpha}~dS}_{B_2^*} \\
&=\ig(-\p_3 h\ak^{3\alpha}\p_3v_{\alpha})(\ak^{3\beta}\tpl \ek_{\beta})(\ak^{3\gamma}\tpl\ek_{\gamma})\dS+B_2^* \\
&\lesssim|\ak^{3\beta}\tpl \ek_{\beta}|_0^2 P(\|v\|_3,\|\p\eta\|_{L^{\infty}},\|h\|_{\hc})+B_2^*\lesssim P(\EE_{\kk})+B_2^*,
\end{aligned}
\end{equation}where in the last step we control $|\ak^{3\beta}\tpl \ek_{\beta}|_0$ as follows
\begin{equation}\label{I31}
\begin{aligned}
|\ak^{3\beta}\tpl \ek_{\beta}|_0&\leq|\lkk(\ak^{3\beta}\tpl\lkk \eta_{\beta})|_0+|[\lkk,\ak^{3\beta}]\TP(\TP\TL\lkk\eta_{\beta})|_0 \\
&\lesssim|\lkk(\ak^{3\beta}\tpl\lkk \eta_{\beta})|_0+|\ak|_{W^{1,\infty}}|\TP^3\eta|_0\lesssim P(\|\eta\|_{\hc})\lesssim P(\EE_\kk).
\end{aligned}
\end{equation}

So far, what remains to be bounded is the second term in RHS of \eqref{I0}
\begin{equation}\label{I2}
I_2:=-\ig \p_3 h\ak^{3\alpha}\ak^{3\beta}\tpl \ek_{\beta}\tpl \psi\,dS,
\end{equation} and
\begin{equation}\label{B1s}
B_1^*=\ig \p_3 h \ak^{3\beta}\tpl\lkk\eta_{\beta}\ak^{3\gamma}\TP_i\lkk^2v_{\gamma}\ak^{i\alpha}\tpl\lkk\eta_{\alpha}\dS
\end{equation} and
\begin{equation}\label{B2s}
B_2^*=-\ig \p_3 h\ak^{3\alpha}\ak^{3\beta}\tpl \ek_{\beta}\tpl\ek_{\gamma}\ak^{i\gamma} \TP_iv_{\alpha}~dS
\end{equation}

Plugging the expression of $\psi$ in \eqref{psi} into \eqref{I2}, one has 
\begin{align}
\label{I210} I_2&=-\ig \p_3 h\ak^{3\alpha}\ak^{3\beta}\tpl \ek_{\beta}\TP^2(\TL\eta_{\gamma}\ak^{ir}\TP_i\lkk^2 v_{\alpha})\dS\\
\label{I220} &~~~~+\ig \p_3 h\ak^{3\alpha}\ak^{3\beta}\tpl \ek_{\beta}\TP^2\ek_{\gamma}\ak^{i\gamma} \TP_iv_{\alpha}\dS\\
\label{I230} &~~~~+\ig \p_3 h\ak^{3\alpha}\ak^{3\beta}\tpl \ek_{\beta}([\TP^2,\ak^{i\gamma}\TP_i v_{\alpha}]\TL\ek_{\gamma})\dS\\
\label{I240} &~~~~+\ig \p_3 h\ak^{3\alpha}\ak^{3\beta}\tpl \ek_{\beta}\TP^2P_{<1}\left(\TL\eta_{\beta}\ak^{i\beta}\TP_i\lkk^2 v-\TL\lkk^2\eta_{\beta}\ak^{i\beta}\TP_i v\right)\,dS.
\end{align} 
It is clear that \eqref{I220} exactly cancels with $B_2^*$ in \eqref{B2s}, and \eqref{I230} can be directly bounded
\begin{equation}\label{I23}
\begin{aligned}
\eqref{I230}&\lesssim |\p_3 h\ak^{3\alpha}|_{L^{\infty}}|\ak^{3\beta}\tpl \lkk\eta_{\beta}|_0 |[\TP^2,\ak^{i\gamma}\TP_i v_{\alpha}]\TL\ek_{\gamma}|_0\\
&\lesssim P(\|h\|_{\hc},\|\eta\|_{\hc},\|v\|_4,|\ak^{3\beta}\tpl\lkk \eta_{\beta}|_0)\lesssim P(\EE_{\kk}).
\end{aligned}
\end{equation}
For \eqref{I240}, one can apply Bernstein's inequality \eqref{bern1} in Lemma \ref{bernstein} and \eqref{I31} to get
\begin{equation}\label{I23'}
\begin{aligned}
\eqref{I240}&\lesssim |\p_3 h\ak^{3\alpha}|_{L^{\infty}}|\ak^{3\beta}\tpl \lkk\eta_{\beta}|_0 \left|P_{<1}\left(\TL\eta_{\beta}\ak^{i\beta}\TP_i\lkk^2 v-\TL\lkk^2\eta_{\beta}\ak^{i\beta}\TP_i v\right)\right|_{\dot{H}^2} \\
&\lesssim|\p_3 h\ak^{3\alpha}|_{L^{\infty}}|\ak^{3\beta}\tpl \lkk\eta_{\beta}|_0\cdot\left|\TL\eta_{\beta}\ak^{i\beta}\TP_i\lkk^2 v-\TL\lkk^2\eta_{\beta}\ak^{i\beta}\TP_i v\right|_0\\
&\lesssim P(\EE_{\kk}).
\end{aligned}
\end{equation}
For \eqref{I210}, we try to move one $\lkk$ on $\eta_{\beta}$ to $\eta_{\alpha}$ to produce the cancellation with $B_1^*$ in \eqref{B1s}:
\begin{align}
\label{I211} \eqref{I210}&=-\ig \p_3 h \ak^{3\beta}\tpl\lkk\eta_{\beta}(\ak^{3\alpha}\TP_i\lkk^2v_{\alpha})(\ak^{i\gamma}\tpl\lkk\eta_{\gamma})\dS \\
\label{I212} &~~~~-\ig \p_3 h \ak^{3\beta}\tpl\lkk\eta_{\beta}([\lkk,\ak^{3\alpha}\ak^{3\beta}\ak^{ir}\TP_i\lkk^2 v_{\alpha}]\tpl\eta_{\gamma})\dS \\
\label{I213} &~~~~-\ig \p_3 h\ak^{3\alpha}\ak^{3\beta}\tpl \ek_{\beta} ([\TP^2,\ak^{i\gamma}\TP_i\lkk^2v_{\alpha}]\TL\eta_{\gamma})\dS
\end{align}
Now we see that \eqref{I211} exactly cancels with $B_1^*$ in \eqref{B1s}. The terms in \eqref{I212} can be controlled by using the mollifier property \eqref{lkk6} after integrating $\TP^{0.5}$ by part (similar to the estimates of $LB_1$), and \eqref{I213} can be directly controlled by using Sobolev trace lemma. We omit the detailed computation here.
\begin{equation}\label{I2123}
\eqref{I212}+\eqref{I213}\lesssim P(\EE_{\kk}).
\end{equation}

Finally, summing up \eqref{I1}-\eqref{B2s}, \eqref{I2123} and plugging it into \eqref{I0}, we get the estimate for the boundary term $I$ after using the Taylor sign condition $\p_3 h\leq -c_0/2<0$:
\begin{equation}\label{tgbdry}
\int_0^T I(t)\dt\lesssim\frac{1}{2}\ig\p_3 h|\ak^{3\alpha}\tpl\lkk\eta_{\alpha}|^2 \dS +\int_0^T P(\EE_{\kk}(t))\dt \lesssim-\frac{c_0}{4}|\ak^{3\alpha}\tpl\lkk\eta_{\alpha}|_0^2+\int_0^T P(\EE_{\kk}(t))\dt
\end{equation}

Now, summing up \eqref{tgs1}, \eqref{tgs2}, \eqref{IKL0}, \eqref{L0}, \eqref{tgK} and \eqref{tgbdry}, we get the estimates for the Alinhac's good unknowns
\begin{equation}\label{tgsgood}
\|\VV(T)\|_0^2+\|e'(h)\tpl h(T)\|_0^2+|\ak^{3\alpha}\tpl\lkk\eta_{\alpha}|_0^2\lesssim\PP_0+\int_0^T P(\EE_{\kk}(t))\dt.
\end{equation} Finally, from the property of Alinhac's good unknowns \eqref{alinhac4}, we can get the estimates of $\TP^4v$ that
\begin{equation}\label{tgs}
\|\TP^4 v(T)\|_0^2+\|e'(h)\TP^4 h(T)\|_0^2+|\ak^{3\alpha}\tpl\lkk\eta_{\alpha}|_0^2\lesssim\PP_0+\int_0^T P(\EE_{\kk}(t))\dt.
\end{equation}

\subsection{Closing the $\kk$-independent a priori estimates}

We conclude this section by deriving the uniform-in-$\kk$ a priori bound for the energy functional $\EE_\kk$ of approximation system \eqref{app1}. Let $\mathcal{T}(t):= -\frac{1}{\p_3 h(t)}$. Then 
\begin{align}
\frac{d}{dt}\|\mathcal{T}(t)\|_{L^\infty}=\|\mathcal{T}(t)\|_{L^\infty}^2 \|\p_3 \p_t h(t)\|_{L^\infty} \leq \|\mathcal{T}(t)\|_{L^\infty}^2 \EE_\kk. 
\end{align}
This implies that the physical sign condition can be propagated if $\EE_\kk$ remains finite.  Next, by plugging \eqref{hodgev}, \eqref{v0l2}, \eqref{v0norm}, \eqref{curlv3}, \eqref{curlvtk}, \eqref{vbdry}-\eqref{vtttbdry}, \eqref{divvtk}, \eqref{tgt} and \eqref{tgs} into \eqref{Ekk}, with $\epsilon>0$ chosen sufficiently small, together with the estimates for $\|\p\eta\|_{L^\infty}$ and $\|\p^2 \eta\|_{2}$, i.e., 
\begin{align}
\|\p \eta\|_{L^\infty} &\leq \|\p \eta_0\|_{L^\infty}+\int_0^T \|v(t)+\psi(t)\|_{L^\infty}\,dt\lesssim \|\p \eta_0\|_{L^\infty}+\int_0^T \|v(t)\|_{2}+\|\psi(t)\|_2\,dt,\\
\|\p^2 \eta\|_{2}&\leq \|\p^2\eta_0\|_{2}+\int_0^T \|v(t)\|_2+\| \psi(t)\|_2\,dt,
\end{align}
we get
\begin{equation}\label{Ekk1}
\EE_\kk(T)\lesssim \PP_0+\int_0^T P(\EE_\kk (t))\dt.
\end{equation} 
Now, \eqref{Ekk0} follows from \eqref{Ekk1} and the Gronwall-type inequality in Tao \cite{tao2006nonlinear}, which finishes the proof of Proposition \ref{uniformkk}.

\section{Construction of the solution to the approximation system}\label{kkexist}

The goal of this section is to construct the solution to the $\kk$-approximation (nonlinear) system \eqref{app1} by an iteration of the approximate solutions $\{(v^{(n)},h^{(n)},\eta^{(n)})\}_{n=0}^{\infty}$. We start with $(v^{(0)},h^{(0)},\eta^{(0)})=(v^{(1)},h^{(1)},\eta^{(1)})=(0,0,\text{Id})$. Inductively, given $(v^{(n)},h^{(n)},\eta^{(n)})$ for some $n\geq 1$, we construct the $(n+1)$-th approximate solutions $(v^{(n+1)},h^{(n+1)},\eta^{(n+1)})$ from the linearization of \eqref{app1} near $a^{(n)}:=[\p\eta^{(n)}]^{-1}$:
\begin{equation}\label{linearn}
\begin{cases}
\p_t\eta^{(n+1)}=v^{(n+1)}+\psi^{(n)}~~~&\text{ in }\Omega, \\
\p_t v^{(n+1)}=-\nabla_{\ak^{(n)}}h^{(n+1)}-ge_3~~~&\text{ in }\Omega, \\
\text{div}_{\ak^{(n)}}v^{(n+1)}=-e'(h^{(n)})\p_t h^{(n+1)}~~~&\text{ in }\Omega, \\
h^{(n+1)}=0~~~&\text{ on }\Gamma, \\
(\eta^{(n+1)},v^{(n+1)}, h^{(n+1)})|_{t=0}=(\text{Id},v_0, h_0).
\end{cases}
\end{equation} Here $\ak^{(n)}:=[\p\ek^{(n)}]^{-1}$ and the correction term $\psi^{(n)}$ is determined by \eqref{psi} with $\eta=\eta^{(n)}, v=v^{(n)}, \ak=\ak^{(n)}$ in that equation. Specifically, we need following facts for the linearized approximation system \eqref{linearn} to construct a solution to the $\kk$-approximation (nonlinear) system \eqref{app1}:
\begin{itemize}
\item System \eqref{linearn} has a (unique) solution (in a suitable function space).

\item The solution of \eqref{linearn} constructed in the last step has an energy estimate uniformly in $n$.

\item The approximate solutions $\{(v^{(n)},h^{(n)},\eta^{(n)})\}_{n=0}^{\infty}$ converge strongly (in some Sobolev spaces).
\end{itemize}

\subsection{A priori estimates for the linearized approximation system}

Before we construct the solution of \eqref{linearn}, we would like to derive the uniform-in-$n$ a priori estimates for this system. Define the energy functional for \eqref{linearn} to be
\begin{equation}\label{En}
\EE^{(n+1)}(t):=\|\eta^{(n+1)}(t)\|_{\hc}^2+\sum_{k=0}^4\|\p_t^{4-k}v^{(n+1)}(t)\|_k^2+\left(\|h^{(n+1)}(t)\|_{\hc}^2+\sum_{k=0}^3 \|\p_t^{4-k}h^{(n+1)}(t)\|_k^2\right)+W^{(n+1)}(t),
\end{equation}where $W^{(n+1)}$ is the energy functional for the 5-th order wave equation of $h^{(n+1)}$
\begin{equation}\label{Wn}
W^{(n+1)}(t):=\sum_{k=0}^4\|\p_t^{5-k}h^{(n+1)}(t)\|_k^2+\sum_{k=0}^3\|\p_t^{4-k}\nabla_{\ak^{(n)}}h^{(n+1)}(t)\|_k^2+\|\nabla_{\ak^{(n)}}h^{(n+1)}(t)\|_{L^\infty}^2+\|\p \nabla_{\ak^{(n)}}h^{(n+1)}(t)\|_3^2.
\end{equation}
\begin{rmk}
The last two terms in \eqref{Wn} can be simplified to $\|\nabla_{\ak^{(n)}}h^{(n+1)}(t)\|_4^2$ if $\Omega$ is bounded. In this case, the wave energy becomes $W^{(n+1)} = \sum_{k=0}^4(\|\p_t^{5-k}h^{(n+1)}\|_k^2+\|\p_t^{4-k}\nabla_{\ak^{(n)}}h^{(n+1)}\|_k^2)$.
\end{rmk}

Our conclusion is 
\begin{prop}\label{Enenergy}
For the solution $(v^{(n+1)},h^{(n+1)},\eta^{(n+1)})$ of \eqref{linearn}, there exists $T_{\kk}>0$ sufficiently small, depending only on $\kk>0$ such that 
\begin{equation}\label{Enuniform}
\sup_{0\leq t\leq T_\kk} \EE^{(n+1)}(t)\lesssim \PP_0.
\end{equation}
\end{prop}

\begin{rmk}
As we will see in the following computation, the control of 4-th order derivatives of $v$ and $h$ does not need the energy of 5-th order wave equation of $h$; the only important difference from the a priori estimates for \eqref{app1} is the boundary term \eqref{Ir} for which we apply the property of tangential smoothing to give a direct control with an extra factor $1/\kk$, instead of producing subtle cancellation as in Section \ref{tgspace}. \textit{However, we included $W^{(n+1)}$ in $\EE^{(n+1)}$ since we need this constraint when constructing the function space when proving the existence of the solution to the linearized system.}
\end{rmk}

We prove Proposition \ref{Enenergy} by induction on $n$. First, when $n=-1,0$, then the conclusion automatically holds because of $(v^{(0)},h^{(0)},\eta^{(0)})=(v^{(1)},h^{(1)},\eta^{(1)})=(0,0,\text{Id})$. Suppose uniform bound holds for all positive integers$\leq n-1$. Then from the induction hypothesis, one has 
\begin{equation}\label{Eninduction}
\forall k\leq n,~~\sup_{0\leq t\leq T_\kk} \EE^{(k)}(t)\lesssim \PP_0.
\end{equation}

We would like to first simplify our notation before we derive the energy estimate for $(v^{(n+1)},h^{(n+1)},\eta^{(n+1)})$. We denote $(v^{(n)},h^{(n)},\eta^{(n)})$ by $(\vr,\hr,\er)$ and $\ar:=[\p\er]^{-1}$, $\Jr:=\det[\p\er]$; and $(v^{(n+1)},h^{(n+1)},\eta^{(n+1)})$ by $(v,h,\eta)$. The smoothed version of $\ar,\er,\Jr$ are denoted by $\ark,\erk,\Jrk$ respectively. Besides, we define $\sigma:=e'(\hr)$. Now, the linearized system \eqref{linearn} becomes
\begin{equation}\label{linearr}
\begin{cases}
\p_t\eta=v+\psir~~~&\text{ in }\Omega, \\
\p_t v=-\park h-ge_3~~~&\text{ in }\Omega, \\
\divr v=-\sigma \p_t h~~~&\text{ in }\Omega, \\
h=0~~~&\text{ on }\Gamma, \\
(\eta,v, h)|_{t=0}=(\text{Id},v_0, h_0).
\end{cases}
\end{equation}We note that the initial data of \eqref{linearr} is the same as the original system \eqref{wwl} because $\ar|_{t=0}=a|_{t=0}=I$.

\subsubsection{Uniform-in-$n$ bounds for the coefficients}\label{linearapriori}

The energy functional for $\vr,\hr,\er$ reads 
\begin{equation}\label{Er}
\mathring{\EE}:=\|\er\|_{\hc}^2+\sum_{k=0}^4\|\p_t^{4-k}\vr\|_k^2+\left(\|\hr\|_{\hc}^2+\sum_{k=0}^3\|\p_t^{4-k}\hr\|_k^2\right)+\mathring{W},
\end{equation}
where $\mathring{W}$ is the energy functional for the 5-th order wave equation of $\hr$, i.e., 
\begin{equation}\label{Wr}
\mathring{W}:=\sum_{k=0}^4\|\p_t^{5-k}\hr\|_k^2+\sum_{k=0}^3\|\p_t^{4-k}\nabla_{\ark^{(n-1)}}\hr\|_k^2+\|\nabla_{\ark^{(n-1)}}\hr\|_{L^\infty}^2+\|\p\nabla_{\ark^{(n-1)}}\hr\|_{3}^2.
\end{equation}
We have
\begin{align} \label{induction'}
\sup_{0\leq t\leq T_\kk}\mathring{\EE}(t) \lesssim \PP_0
\end{align}
in light of the induction hypothesis \eqref{Eninduction}. \\In addition, we have the following bounds for $\ar,\er,\Jr$ provided they hold for $(v^{(k)},h^{(k)},\eta^{(k)})$ for $k\leq n-1$. The control of these quantities are important when we do the uniform-in-$n$ a priori estimates and construct the solution for system \eqref{linearr}.
\begin{lem}\label{apriorir}
Let $T\in(0,T_\kk)$. There exists some $0<\epsilon<<1$ and $N>0$ such that
\begin{align}
\label{psir} \psir&\in L_t^{\infty}([0,T];H^4(\Omega)),~~\p_t^l\psir\in  L_t^{\infty}([0,T];H^{5-l}(\Omega)),~~\forall 1\leq l\leq 4; \\
\label{Ida} \|\Jr-1\|_3&+\|\Jrk-1\|_3+\|\text{Id}-\ark\|_3+\|\text{Id}-\ar\|_3\leq\epsilon ;\\
\label{ehc} \p \er&\in L^{\infty}([0,T]; L^\infty(\Omega)),\quad \p^2\er \in L^{\infty}([0,T];H^2(\Omega));\\
\label{etart} \p_t\er&\in L^{\infty}([0,T];H^4(\Omega));\\
\label{etartt} \p_t^{l+1}\er&\in L^{\infty}([0,T];H^{5-l}(\Omega)),~~\forall 1\leq l\leq 4; \\
\label{Jr} \Jr&\in L^{\infty}([0,T];L^{\infty}(\Omega)),~~\p\Jr\in L^{\infty}([0,T];H^2(\Omega);\\
\label{Jtr} \p_t\Jr&\in L^{\infty}([0,T];H^3(\Omega)),~~\p_t^{l+1}\Jr\in L^{\infty}([0,T];H^{4-l}(\Omega)),~~\forall 1\leq l\leq 4; \\
\label{weightr} 1/N\leq\sigma\leq N,&~ \p_t^{l}\sigma\in L^{\infty}([0,T];H^{5-l}(\Omega)),~~\forall 1\leq l\leq 5.
\end{align}
Particularly, we have
\begin{align} \label{circile vari estimate}
\sup_{t\in [0,T]}\Bigg(\|\er\|_{\hc}&+\|\p_t \er(t)\|_{4}+\sum_{1\leq l\leq 4}\|\p_t^{l+1}\er(t)\|_{5-l}+ \|\Jr(t)\|_{L^\infty}\nonumber\\
&+\|\p \Jr(t)\|_{2} + \|\p_t \Jr(t)\|_{3}+\sum_{1\leq l\leq 4}\|\p_t^{l+1}\Jr\|_{4-l}+\sum_{1\leq l\leq 5} \|\p_t^l\sigma(t)\|_{5-l} \Bigg)\leq \PP_0.
\end{align} 
\end{lem}

\begin{proof}
First, the bound for $\psir$ and $\p_t^{l}\psir $ for $1\leq l\leq 3$ directly follows from \eqref{psi4}-\eqref{psittt2} in Lemma \ref{etapsi}. Then the identity $$\text{Id}-\ar=-\int_0^t \p_t \ar=\int_0^t \ar:(\p \p_t\er):\ar=\int_0^t \ar:(\p (\vr+\psi^{(n-1)})):\ar$$ yields \eqref{Ida} by choosing $\epsilon$ suitably small (depending on $T_\kk$). Similar results hold for $\Jr$.

As for $\er$, $\p_t\er=\vr+\psi^{(n-1)}$ gives \eqref{ehc} and \eqref{etart}. Taking $\p_t^l$ in this equation and combining the induction hypothesis on $\mathring{E}$ and $\psi^{(n-1)}$ we can get the bound for $\p_t^{l+1}\er$ in \eqref{etartt}. For $\Jr$, recall $\Jr:=\det[\p\er]$ which equals a multi-linear funtion of its elements $\p\er$. So the bound for $\er$ and $\p_t^l \er$ yields the bounds for $\p_t^l \Jr$.

To conclude the proof, it suffices to control $\|\p_t^4 \psir\|_1$. From \eqref{psitk}, we know 
\begin{equation}\label{psir4}
\begin{cases}
\Delta \p_t^4\psir=0  &~~~\text{in }\Omega, \\
\p_t^4\psir=\TL^{-1}\mathbb{P}\p_t^4\left(\TL\er_{\beta}\ark^{i\beta}\TP_i\lkk^2 \vr-\TL\lkk^2\er_{\beta}\ark^{i\beta}\TP_i \vr\right) &~~~\text{on }\Gamma.
\end{cases}
\end{equation}
Using Lemma \ref{harmonictrace} for harmonic functions, we know
\begin{align*}
\|\p_t^4\psir\|_1&\lesssim|\p_t^4\psir|_{0.5}=\left|\TL^{-1}\mathbb{P}\p_t^4\left(\TL\er_{\beta}\ark^{i\beta}\TP_i\lkk^2 \vr-\TL\lkk^2\er_{\beta}\ark^{i\beta}\TP_i \vr\right)\right|_{0.5} \\
&\lesssim|\mathbb{P}\p_t^4\left(\TL\er_{\beta}\ark^{i\beta}\TP_i\lkk^2 \vr-\TL\lkk^2\er_{\beta}\ark^{i\beta}\TP_i \vr\right)|_{\dot{H}^{-1.5}}\\
&\lesssim|\mathbb{P}\p_t^4\left(\TL\er_{\beta}\ark^{i\beta}\TP_i\lkk^2 \vr-\TL\lkk^2\er_{\beta}\ark^{i\beta}\TP_i \vr\right)|_{\dot{H}^{-0.5}},
\end{align*}where we used the Bernstein inequality \eqref{bern2} and the definition of $\mathbb{P}$ (restrict $|\xi|\gtrsim 1$ to get the last inequality).

The most difficult terms appear when $\p_t^4$ falls on $\TL\er$ or $\TP \vr$. Here we only show how to control $\TL\lkk^2\er_{\beta}\ark^{i\beta}\TP_i\p_t^4 \vr$ and the rest highest order terms can be controlled in the same way. For any test function $\phi\in \dot{H}^{0.5}(\R^2)$ with $|\phi|_{\dot{H}^{0.5}}\leq 1$, we consider
\begin{align*} 
|\langle \TL\lkk^2\er_{\beta}\ark^{i\beta}\TP_i\p_t^4 \vr, \phi\rangle| &=|\langle \TP_i\p_t^4 \vr,\TL\lkk^2\er_{\beta}\ark^{i\beta}\phi\rangle|\\
&=|\langle \TP^{0.5}\p_t^4 \vr, \TP^{0.5}(\TL\lkk^2\er_{\beta}\ark^{i\beta}\phi)\rangle| \\
&\lesssim|\p_t^4 v|_{\dot{H}^{0.5}} |\TL\lkk^2\er_{\beta}\ark^{i\beta}\phi|_{\dot{H}^{0.5}} \\
&\lesssim \|\p_t^4 v\|_1(|\phi|_{\dot{H}^{0.5}}|\TL\lkk^2\er\ark|_{L^{\infty}}+|\phi|_{L^4}|\TL\er|_{\dot{W}^{0.5,4}}|\ar|_{L^{\infty}})\\
&\lesssim  (\|\p_t^4 v\|_1\|\p^2\er\|_2\|\p\er\|_{L^{\infty}})|\phi|_{\dot{H}^{0.5}},
\end{align*} where we used $\dot{H}^{-0.5}$-$\dot{H}^{0.5}$ duality and Kato-Ponce inequality \eqref{product}. Taking supremum over all  $\phi\in \dot{H}^{0.5}(\R^2)$ with $|\phi|_{\dot{H}^{0.5}}\leq 1$, we obtain 
\[
\|\TL\lkk^2\er_{\beta}\ark^{i\beta}\TP_i\p_t^4 \vr\|_{\dot{H}^{-0.5}}\lesssim  \|\p_t^4 \vr\|_1\|\p^2\er\|_2\|\p\er\|_{L^{\infty}},
\] and thus gives the bound for $\|\p_t^4\psir\|_1$.
From the second equation of \eqref{linearn} we know that $\p_t^4 \vr=-\p_t^3\nabla_{\ak^{(n-1)}} \hr$, of which the $H^1$-norm of the RHS is exactly in the energy functional $\EE^{(n-1)}$ as in \eqref{En}-\eqref{Wn}. So $\|\p_t^4\psir\|_{1}$ is bounded by the induction hypothesis.
 
It remains to control $\|\p_t^5\er\|_1$ which also gives the bounds for $\p_t^5\Jr$. Taking $\p_t^4$ in the first equation of \eqref{linearr} we have $\p_t^5\er=\p_t^4 \vr+\p_t^4 \psi^{(n-1)}$ and $\p_t^4 \psi^{(n-1)}$ can be bounded in $H^1$ in the same way as above. Lastly, \eqref{circile vari estimate} follows from the estimates above and \eqref{induction'}. 
\end{proof}

\subsubsection{Uniform-in-$n$ a priori estimates for the linearized system}\label{linear4}

With the inductive hypothesis \eqref{Eninduction} and Lemma \ref{apriorir}, we are now able to control the energy functional for $(v,h,\eta)$ which solves the system \eqref{linearr}. Let
\begin{equation}\label{El}
\EE^{(n+1)}:=\|\eta\|_{\hc}^2+\sum_{k=0}^4\|\p_t^{4-k}v\|_k^2+\left(\|h\|_{\hc}^2+\sum_{k=0}^3\|\p_t^{4-k}h\|_k^2\right)+W^{(n+1)},
\end{equation}
where $W^{(n+1)}$ is the energy functional for the 5-th order wave equation of $h$
\begin{equation}\label{Wl}
W^{(n+1)}:=\sum_{k=0}^4\|\p_t^{5-k}h\|_k^2+\sum_{k=0}^3\|\p_t^{4-k}\park h\|_k^2+\|\park h\|_{L^\infty}^2+\|\p\park h\|_{3}^2.
\end{equation}

The estimate for $\EE^{(n+1)}-W^{(n+1)}$ is quite similar (actually a bit easier) to what we have done in Sect. \ref{kkapriori}, so we will not go over all the details, but still point out the different steps, especially the boundary term control, because we no longer need $\kk$-independent estimates.

\paragraph*{Step 1: Estimates for $h$}

The lower order term $\|\p h\|_{L^\infty}$ can be treated identically as what appears in Section \ref{lower order h}.  The arguments in Sections \ref{tangential h}-\ref{lessnormal} suggest that the key to control the top order Sobolev norm of $h$ is to study the wave equation 
\begin{equation}\label{hrwave0}
\Jrk\sigma \p_t^2 h-\p_{\mu}(\mathring{E}^{\nu\mu}\p_{\mu}h)=-\Jrk\p_t\ark^{\nu\alpha}\p_\nu v_{\alpha}-\Jrk\p_t\sigma\p_t h, \quad \text{where}\,\, \mathring{E}^{\nu\mu}:=\Jrk\ark^{\nu\alpha}\ark^{\mu}_{\alpha},
\end{equation}
which is obtained by taking $\Jrk\divr$ in the second equation of \eqref{linearr}.  Let $\dd=\TP$ or $\p_t$. We take $\dd^3$ in \eqref{hrwave0} to get
\begin{equation}\label{hrwave3}
\Jrk\sigma \p_t^2\dd^3 h-\p_\nu(\mathring{E}^{\nu\mu}\dd^3\p_{\mu}h)=-\dd^3(\Jrk\p_t\ark^{\nu\alpha}\p_\nu v_{\alpha})-[\dd^3,\Jrk\sigma]\p_t^2 h+\dd^3(\Jrk \p_t\sigma\p_t h)+\p_{\nu}([\dd^3,\mathring{E}^{\nu\mu}]\p_{\mu}h).
\end{equation} Compared with \eqref{waveh3}, we only replace $e'(h)$, $\ak$ and $\tilde{J}$ by $\sigma$, $\ark$ and $\Jrk$, respectively. Using the same method and the a priori bound in Lemma \ref{apriorir}, one can get
\begin{equation}
\begin{aligned}
&~~~~\sum_{\dd^3}\io\tilde{J}e'(h)|\dd^3\p_t h|^2+ |\p\dd^3 h|^2\dy\bigg|_{t=T} \\
&\lesssim\PP_0+ \epsilon\|\TP^3h(T)\|_1^2+\sum_{\dd^3}\int_0^T P(\|\p^2\er\|_2,\|\p\er\|_{L^{\infty}},\|v\|_4, \|\p_t v\|_3,\|\p_t^2 v\|_3,\|\p h\|_{L^{\infty}},\|\p_t h\|_3,\|\p_t^2 h\|_1)\|\dd^3 \p h\|_0\dt,
\end{aligned}
\end{equation}where $\epsilon>0$ can be chosen sufficiently small such that $\epsilon\|\TP^3h(T)\|_1^2$ can be absorbed by LHS. 
Finally, the full Sobolev norm $\|\p_t^{4-k} h\|_k,~k=0,1,2,3$ and $\|h\|_{\hc}$ can be bounded by adapting the arguments in Section \ref{lessnormal}.

\bigskip

\paragraph*{Step 2: The div-curl estimates for $v$}

From \eqref{Ida}, \eqref{etart} and \eqref{etartt} in Lemma \ref{apriorir}, we know all the steps can be copied as in Section \ref{divcurl} after replace $\ak$ by $\ark$, $\eta$ by $\er$ and $e'(h)$ by $\sigma$. We omit the detailed computations and only list the results here. 

\begin{itemize}

\item $L^2$-estimates:
\begin{equation}\label{vtrl2}
\begin{aligned}
\|v(T)\|_0 &\leq \|v_0\|_0+\int_0^T \|\p_t v(t)\|_0\,dt,\\ 
\|\p_t v(T)\|_0&\lesssim\|\p_t v(0)\|_0+\int_0^T \|\p_t^2 v(t)\|_0\dt\lesssim\|\p_t v(0)\|_0+\int_0^T P(\|\er\|_{\hc}, \|v\|_3,\|\p h\|_{L^\infty},\|\p_t h\|_1)\dt,\\
\|\p_t^2 v(T)\|_0&\lesssim\|\p_t^2 v(0)\|_0 +\int_0^T P(\|\er\|_{\hc},\|\p_t v\|_1,\|\p h\|_{L^\infty},\|\p_t h\|_1, \|\p_t^2 h\|_1)\dt, \\
\|\p_t^3 v(T)\|_0&\lesssim\|\p_t^3 v(0)\|_0 +\int_0^T P(\|\er\|_{\hc},\|v\|_3,\|\p_t v\|_2,\|\p_t^2 v\|_1, \|h\|_{\hc},\|\p_t h\|_2, \|\p_t^2 h\|_1, \|\p_t^3 h\|_0)\dt.
\end{aligned}
\end{equation}

\item Boundary estimates:
\begin{align}
\label{vrtbdry} |\TP^2 (\p_tv\cdot N)|_{0.5}&\lesssim  \|\TP^3\p_t v\|_0+\|\TP^2 \dive \p_t v\|_0 \\
\label{vrttbdry} |\TP (\p_t^2v\cdot N)|_{0.5}&\lesssim  \|\TP^2\p_t^2 v\|_0+\|\TP \dive \p_t^2 v\|_0 \\
\label{vrtttbdry} |\p_t^3 v\cdot N|_{0.5}&\lesssim  \|\TP\p_t^3 v\|_0+\|\dive \p_t^3 v\|_0.
\end{align}

\item The div-curl estimates:
\begin{align}
\label{divr} \sum_{k=0}^3\|\dive \p_t^k v\|_{3-k}&\lesssim \PP_0+\int_0^T P(\EE^{(n+1)}(t))\dt.\\
\label{curlr} \sum_{k=0}^3\|\curl\p_t^{k} v\|_{3-k}&\lesssim\epsilon\sum_{k=0}^3\|\p_t^k v\|_{4-k}+\PP_0+\PP_0\int_0^TP(\EE^{(n+1)}(t)-W^{(n+1)}(t))\dt.
\end{align}
\end{itemize}

Combining with Hodge's decomposition inequality in Lemma \ref{hodge}, to estimate the full Sobolev norm of $\p_t^k v$, it suffices to control $\|\TP^k\p_t^{4-k}\|_0$.

\bigskip

\paragraph*{Step 3: Tangential estimates for time derivatives of $v$}

This part also follows in the same way as Section \ref{tgtime}. Let $\dd^4=\p_t^4,\TP\p_t^3,\TP^2\p_t^2,\TP^3\p_t$. One can directly compute $\frac{d}{dt}\frac{1}{2}\io |\dd^4 h|\dy$ and follow the same method in \eqref{tgt0}-\eqref{tgl1l2} to get the analogous conclusion as \eqref{tgt}:
\begin{equation}\label{tgtr}
\sum_{k=1}^4\frac{d}{dt}\left(\|\p_t^k\TP^{4-k}v\|_0^2+\|\sqrt{\sigma}\p_t^k \TP^{4-k} h\|_0^2\right)\lesssim\PP_0\cdot P(\EE^{(n+1)}(t)-W^{(n+1)}(t)).
\end{equation}

\bigskip

\paragraph*{Step 4: Tangential estimates for $v$}

In this step we still mimic the proof as in Section \ref{tgspace}. 
For a given function $g$, we use $\mathring{\GG}$ to denote its Alinhac's good unknown.  Then there holds 
\begin{equation}\label{alinhacar}
\tpl(\park g)=\park \mathring{\GG}+\mathring{C}(g),
\end{equation} 
where the error term $\mathring{C}(g)$ is defined in the same way as in \eqref{alinhaca} but replacing $\ak$ by $\ark$. In view of \eqref{alinhacc}, we have
\begin{equation}\label{alihanccr}
\|\mathring{C}(g)\|_0\lesssim P(\|\er\|_{\hc})\|g\|_{\hc}.
\end{equation}
Similar to \eqref{alinhac4}, one also has
\begin{equation}\label{alinhac4r}
\|\TP^4 g(T)\|_0\lesssim\|\mathring{\GG}(T)\|_0+\int_0^T P(\|\er\|_{\hc}, \|\mathring{v}\|_{4}, \|\p g\|_{L^\infty}, \|\p_t g\|_{L^\infty}).
\end{equation}

We introduce the Alinhac's good unknown $\VVr$ and $\HHr$ for $v$ and $h$:
\begin{align}
\label{vrgood} \VVr&:=\tpl v-\tpl\erk\cdot\park v, \\
\label{hrgood} \HHr&:=\tpl h-\tpl\erk\cdot\park h.
\end{align}
Applying $\tpl$ to the second equation in the linearization system \eqref{linearr}, one gets
\begin{equation}
\label{goodlinearr} \p_t\VVr=-\park \HHr+\underbrace{\p_t(\tpl \erk\cdot\park v)-\mathring{C}(h)}_{=:\mathring{\FF}},
\end{equation}subject to the boundary condition
\begin{equation}\label{bdryrgood}
\HHr=-\tpl \erk_{\beta}\ark^{3\beta}\p_3 h~~~\text{ on }\Gamma,
\end{equation}and the corresponding compressibility condition
\begin{equation}\label{divrgood}
\park\cdot\VVr=\tpl(\divr v)-\mathring{C}^{\alpha}(v_{\alpha}),~~~\text{in }\Omega.
\end{equation}

Now we take $L^2$ inner product between \eqref{goodlinearr} and $\VVr$ to get analogous result to \eqref{tgs1}.
\begin{equation}\label{tgs1r}
\frac{1}{2}\frac{d}{dt}\io |\VVr|^2\dy=-\io \park \HHr\cdot \VVr\dy+\io\mathring{\FF}\cdot\VVr\dy,
\end{equation}where $\|\mathring{\FF}\|_0$ can be directly controlled as in \eqref{tgs2}. As for the first term, we integrate by parts to get 
\begin{equation}\label{tgs2r}
-\io \park \HHr\cdot \VVr\dy=-\ig \ark^{3\alpha}\VVr_{\alpha} \HHr\dS+\io \HHr(\park\cdot\VVr)\dy+\io\p_\mu\ark^{\mu\alpha}\HHr\VVr_{\alpha}\dy,
\end{equation}where the second and the third term can be controlled in the same way as in \eqref{IKL0}-\eqref{tgK}.

For the boundary term in \eqref{tgs2r}, we no longer need to plug the precise form of $\psir$ into it and find the subtle cancellation as in Section \ref{tgspace} because the energy estimate is not required to be $\kk$-independent. Instead, we integrate $\TP^{0.5}$ by parts, apply Kato-Ponce inequality \eqref{product} and Sobolev embedding ${H}^{0.5}(\R^2)\hookrightarrow L^4(\R^2)$ to get
\begin{equation}\label{Ir}
\begin{aligned}
-\ig \ark^{3\alpha}\VVr_{\alpha} \HHr\dS&=\ig \p_3 h\tpl(\lkk^2\er_{\beta})\ark^{3\beta}\ark^{3\alpha}\VVr_{\alpha}\dS \\
&\lesssim\big(|\p_3 h\ark^{3\beta}\ark^{3\alpha}|_{L^{\infty}}|\tpl(\lkk^2\er_{\beta})|_{0.5}|+|\p_3 h\ark^{3\beta}\ark^{3\alpha}|_{W^{0.5,4}}|\tpl(\lkk^2\er_{\beta})|_{L^4}\big)|\VVr|_{\dot{H}^{-0.5}}\\
&\lesssim\|h\|_{\hc} (|\ark|_{L^{\infty}}+|\TP \ark|_{0.5}^2)\frac{1}{\kk}\|\TP^3 \eta\|_1(|\TP^3 v|_{\dot{H}^{0.5}}+|\TP^3\erk|_{\dot{H}^{0.5}}|\ark\p v|_{L^{\infty}}+|\TP^3\erk|_{L^4}|\ark\p v|_{\dot{W}^{0.5,4}}) \\
&\lesssim \frac{1}{\kk} P(\|\er\|_{\hc}, \|h\|_{\hc},\|v\|_4) .
\end{aligned}
\end{equation}

Combining \eqref{alinhac4r} with the estimates above, we have
\begin{equation}\label{tgsr}
\|\TP^4v(T)\|_0\lesssim_{\kk}\PP_0+\int_0^T P(\EE^{(n+1)}(t)-W^{(n+1)}(t))\dt.
\end{equation}
Summing up the estimates for $h$, div-curl estimates and tangential estimates, we get 
\begin{equation}\label{En4final}
\sum_{k=0}^4 \|\p_t^{4-k}v\|_k^2+\Big(\|h\|_{\hc}^2+\sum_{k=0}^3\|\p_t^{4-k} h\|_0^2\Big)\lesssim_{\kk} \PP_0+\int_0^T P(\EE^{(n+1)}(t)-W^{(n+1)}(t))\dt.
\end{equation}

\subsubsection{Estimates for $W^{(n+1)}$: 5-th order wave equation of $h$}\label{wave5}
We would like to control $$W^{(n+1)}=\sum_{k=0}^4\|\sqrt{\sigma}\p_t^{5-k}h\|_k^2+\sum_{k=0}^3\|\p_t^{4-k}\park h\|_k^2+\|\park h\|_{L^\infty}^2+\|\p \park h\|_{3}^2.$$ 
It suffices to control only the top order terms. We take $\divr$ in the second equation of \eqref{linearr} to get the wave equation for $h$:
\begin{equation}\label{hwave0r}
\sigma\p_t^2 h-\Delta_{\ark} h=-\p_t\ark^{\nu\alpha}\p_{\nu}v_{\alpha}-\p_t\sigma\p_t h.
\end{equation}
Before we derive the higher order wave equation, we would like to reduce the estimates of $W^{(n+1)}$ to that of $\|\sqrt{\sigma}\p_t^5 h\|_0^2+\|\p_t^4\park h\|_0^2$ via \eqref{hwave0r} and the elliptic estimate Lemma \ref{GLL}.

We start with $\|\p^4\park h\|_0$ and $\|\p^4\p_t h\|_0$. By the elliptic estimate Lemma \ref{GLL}, we have
\begin{equation}
\|\p^4\park h\|_0\lesssim C(\|\erk\|_{\hc})(\sum_{r\leq 3}\|\p^r \Delta_{\ark} h\|_0 +\|\TP\p\erk\|_3\|h\|_{\hc}),
\end{equation}
in which the term $\|\TP\p\erk\|_3\lesssim\kk^{-1}\|\p^2\er\|_2$ by the property of tangential smoothing. The term $\p^3 \Delta_{\ark} h$ can be expressed as follows by using \eqref{hwave0r}
\begin{equation}
\p^3 \Delta_{\ark} h=\p^3(\sigma\p_t^2 h)+\p^3(\p_t\ark^{\nu\alpha}\p_{\nu}v_{\alpha}+\p_t\sigma\p_t h),
\end{equation}which produces one more time derivative and thus reduce the control of $\p^4\park h$ to $\p^3\p_t^2 h$:
\begin{equation}
\|\p^3 \Delta_{\ark} h\|_0\leq\|\sigma\p_3\p_t^2 h\|_0+\|[\p^3,\sigma]\p_t^2 h\|_0+\|\p_t\ark^{\nu\alpha}\p_{\nu}v_{\alpha}\|_3+\|\p_t\sigma\p_t h\|_3.
\end{equation}

As for $\p^4\p_t h$, we note that for any $1\leq r\leq 4$, $$(\p^rf)_{\alpha}=\p^{r-1}\p_{\alpha}f=\p^{r-1}(\ark^{\mu\alpha}\p_\mu f)+\p^{r-1}((\delta^{\mu\alpha}-\ark^{\mu\alpha})\p_{\mu}f)$$ together with $\|\ark-1\|\leq\epsilon$ gives
\begin{equation}\label{trick}
\|\p^rf\|_0\lesssim\|\p^{r-1}\park f\|_0+\epsilon\|\p^r f\|_0, 
\end{equation} where the last term can be absorbed by LHS after choosing $\epsilon>0$ sufficently small.

Therefore we have
\begin{equation}
\p^3\p_{\alpha}\p_t h=\p^3(\ark^{\mu\alpha}\p_\mu\p_t h)+\p^3(\underbrace{(\delta^{\mu\alpha}-\ark^{\mu\alpha})}_{\|\cdot\|_3\leq\epsilon}\p_\mu\p_t h),
\end{equation}which gives
\begin{equation}
\|\p^4\p_t h\|_0\lesssim\|\p^3\park \p_t h\|_0+\epsilon\|\p^4\p_t h\|_0,
\end{equation} where the last term can be absorbed by LHS after choosing $\epsilon>0$ sufficiently small. So we are able to reduce the estimates for $\|\p^4\park h\|_0$ and $\|\p^4\p_t h\|_0$ to $\|\p^3\p_t^2 h\|_0$ and $\|\p^3\park \p_t h\|_0$, respectively, plus lower order terms. In other words,  we replace one spatial derivative by one time derivative via the elliptic estimate and wave equation \eqref{hwave0r}. 

Next, since $\p_t h|_{\Gamma}=0$, we apply the elliptic estimate in Lemma \eqref{GLL} to $\park\p_t h$ to get
\begin{equation}
\|\p^3\park\p_th\|_0\lesssim C(\|\erk\|_{\hc})(\sum_{r\leq 2}\|\p^r \Delta_{\ark}\p_t h\|_0 +\|\TP\p\erk\|_3\|\p_th\|_3).
\end{equation} The term $\p^r \Delta_{\ark} \p_t h$ can be re-expressed as follows by commuting $\p_t$ through \eqref{hwave0r}:
\begin{equation}
\p^2 \Delta_{\ark} \p_t h=\sigma\p^2\p_t^3 h+\p^2\p_t(\p_t\ark^{\nu\alpha}\p_{\nu}v_{\alpha}+\p_t\sigma\p_t h)+[\p^2\p_t,\sigma]\p_t^2 h-\p^2([\p_t,\Delta_{\ark}]h),
\end{equation} and thus the control of $\p^2 \Delta_{\ark} \p_t $ is reduced to $\sigma\p^2\p_t^3 h$ plus the other terms on the RHS of the last inequality
\begin{equation}
\|\p^2 \Delta_{\ark} \p_t h \|_0\lesssim\|\sigma\p^2\p_t^3 h\|_0+\|\p_t(\p_t\ark^{\nu\alpha}\p_{\nu}v_{\alpha}+\p_t\sigma\p_t h)\|_2+\|[\p^2\p_t,\sigma]\p_t^2 h\|_0+\|[\p_t,\Delta_{\ark}]h\|_2.
\end{equation} As for $\p^3\p_t^2 h$, we again rewrite one Lagrangian spatial derivative in terms of one Eulerian spatial derivative plus an error term:
\begin{equation}
\p^2\p_{\alpha}\p_t^2 h=\p^2(\ark^{\mu\alpha}\p_{\mu}\p_t^2 h)+\p^2((\delta^{\mu\alpha}-\ark^{\mu\alpha})\p_{\mu}\p_t^2 h),
\end{equation}which gives
\begin{equation}
\|\p^3\p_t^2 h\|_0\lesssim\|\p^2\park \p_t^2 h\|_0+\epsilon\|\p^3\p_t^2 h\|_0,
\end{equation} where the last term can be again absorbed by LHS after choosing $\epsilon>0$ sufficiently small.

The reduction mechanism above can be summarized as the following diagram
\begin{equation}
\begin{aligned}
\p^4\p_t h&\xrightarrow{\eqref{trick}}\p^3\park h\xrightarrow[\eqref{hwave0r}]{\text{Lem }\ref{GLL}}\p^2\p_t^3 h\xrightarrow{\eqref{trick}}\p \park\p_t^3 h\xrightarrow[\eqref{hwave0r}]{\text{Lem }\ref{GLL}}\p_t^5 h;\\
\p^4\park h&\xrightarrow[\eqref{hwave0r}]{\text{Lem }\ref{GLL}}\p^3\p_t^2 h\xrightarrow{\eqref{trick}}\p^2\park\p_t^2 h\xrightarrow[\eqref{hwave0r}]{\text{Lem }\ref{GLL}}\p\p_t^4 h\xrightarrow{\eqref{trick}}\park\p_t^4 h.
\end{aligned}
\end{equation}

As is shown above, we are able to replace one spatial derivative by one time derivative after using the elliptic estimate and wave equation \eqref{hwave0r}. Repeat the steps above, we can reduce the estimates of $W^{(n+1)}$ to $\|\p_t^5 h\|_0$ and $\|\p_t^4\park h\|_0$ which can be controlled via the 5-th order wave equation of $h$ (i.e., taking $\p_t^4$ in \eqref{hwave0r}) plus commutator terms. Specifically, 
\begin{align}
\label{WW} \sum_{k=1}^{4}\|\p_t^{5-k}\p^k h\|_0+\|\p_t^{4-k}\p^k\park h\|_0&\lesssim C(\|\p\er\|_{L^{\infty}},\|\TP^2\er\|_2)(\sigma+\sigma^2)(\|\p_t^5 h\|_0+\|\park\p_t^4 h\|_0)\\
\label{comm1} &+ \frac{1}{\kk}C(\|\p\er\|_{L^{\infty}},\|\p^2\er\|_2)(\sigma+\sigma^2)(\|h\|_0+\|\p_th\|_0+\cdots+\|\p_t^3 h\|_0) \\
\label{comm2} &+ \|[\p^3,\sigma]\p_t^2 h\|_0+\|\p_t\ark^{\nu\alpha}\p_{\nu}v_{\alpha}\|_3+\|\p_t\sigma\p_t h\|_3 \\
\label{comm3} &+ \|\p_t(\p_t\ark^{\nu\alpha}\p_{\nu}v_{\alpha}+\p_t\sigma\p_t h)\|_2+\|[\p^2\p_t,\sigma]\p_t^2 h\|_0+\|[\p_t,\Delta_{\ark}]h\|_2\\
\label{comm4} &+ \|\p_t^2(\p_t\ark^{\nu\alpha}\p_{\nu}v_{\alpha}+\p_t\sigma\p_t h)\|_1+\|[\p\p_t^2,\sigma]\p_t^2 h\|_0+\|[\p_t^2,\Delta_{\ark}]h\|_1 \\
\label{comm5} &+ \|\p_t^3(\p_t\ark^{\nu\alpha}\p_{\nu}v_{\alpha}+\p_t\sigma\p_t h)\|_0+\|[\p_t^3,\sigma]\p_t^2 h\|_0+\|[\p_t^3,\Delta_{\ark}]h\|_0\\
\label{comm6} &+ \sum_{k=1}^4\|\p^k([\p_t^{4-k},\ark^{\mu\alpha}]\p_{\mu}h)\|_0^2
\end{align}
Here,  all the commutator and error terms \eqref{comm1}-\eqref{comm6} consists of $\leq 4$ derivatives of $v,\eta,h$, and $\ark$ which have no problem to bound. Thus,
\begin{equation}\label{commw}
\eqref{comm1}+\cdots+\eqref{comm6}\lesssim \PP_0\left(\EE^{(n+1)}(t)-W^{(n+1)}(t)\right),
\end{equation}
where the RHS is controlled in \eqref{En4final}.

It remains to control $\|\sqrt{\sigma}\p_t^5 h\|_0+\|\p_t^4\park h\|_0$. We apply $\p_t^4$ to \eqref{hwave0r} to get:
\begin{equation}\label{hwave5r}
\sigma \p_t^6 h-\ark^{\nu\alpha}\p_\nu(\ark^{\mu}_{\alpha}\p_\mu\p_t^4 h)=\underbrace{-\p_t^4(\p_t\ark^{\nu\alpha}\p_{\nu}v_{\alpha})-\p_t^4(\p_t\sigma\p_t h)-[\p_t^4,\sigma]\p_t^2 h+[\p_t^4,\Delta_{\ark}] h}_{=:F_5}.
\end{equation}
Multiplying \eqref{hwave5r} by $\p_t^5 h$ and integrate over $\Omega$, we get
\begin{equation}\label{h51}
\io\sigma\p_t^5h\p_t^6h\dy-\io\p_t^5h\ark^{\nu\alpha}\p_\nu(\ark^{\mu}_{\alpha}\p_\mu\p_t^4 h)\dy=\underbrace{\io F_5\p_t^5 h\dy}_{LW_1}.
\end{equation}
The first term in \eqref{h51} is
\begin{equation}\label{h511}
\frac{1}{2}\frac{d}{dt}\io \sigma|\p_t^5h|^2\dy-\frac{1}{2}\io\p_t\sigma|\p_t^5 h|^2\dy.
\end{equation} 
For the second term in \eqref{h51}, we integrate $\p_\nu$ by parts and note that $\p_t^5 h|_{\Gamma}=0$ makes the boundary integral vanish.
\begin{equation}\label{h512}
\begin{aligned}
&~~~~-\io\p_t^5h\ark^{\nu\alpha}\p_\nu(\ark^{\mu}_{\alpha}\p_\mu\p_t^4 h)\dy \\
&=\io(\ark^{\nu\alpha}\p_\nu\p_t^5h)(\ark^{\mu}_{\alpha}\p_\mu\p_t^4 h)\dy+\underbrace{\io\p_t^5h \p_\nu\ark^{\nu\alpha} (\ark^{\mu}_{\alpha}\p_\mu\p_t^4 h)\dy}_{LW_2} \\
&=\frac{1}{2}\frac{d}{dt}\io |\park \p_t^4h|^2\dy+\underbrace{\io([\ark^{\nu\alpha},\p_t]\p_\nu\p_t^4 h)(\ark^{\mu}_{\alpha}\p_\mu\p_t^4 h)\dy}_{LW3}+LW_2.
\end{aligned}
\end{equation}
 Plugging \eqref{h511} and \eqref{h512} into \eqref{h51}, we have
\begin{equation}\label{h52}
\frac{d}{dt}\left(\io\sigma|\p_t^5 h|^2+|\park\p_t^4 h|^2\dy\right)=\frac{1}{2}\io\p_t\sigma|\p_t^5 h|^2\dy+LW_1+LW_2+LW_3.
\end{equation}

Now we come to estimate the RHS of \eqref{h52}: First, invoking \eqref{weightr}, we have

\begin{equation}\label{LW0}
\frac{1}{2}\io\p_t\sigma|\p_t^5 h|^2\dy\lesssim \|\p_t\sigma\|_{L^{\infty}}\|\p_t^5 h\|_0^2 \lesssim \PP_0\|\sqrt{\sigma}\p_t^5 h\|_0^2.
\end{equation}
Second, invoking \eqref{circile vari estimate}, we have
\begin{equation}\label{LW2}
LW_2\lesssim \PP_0\|\p_t^5h\|_0\|\park\p_t^4 h\|_0,
\end{equation}

\begin{equation}\label{LW31}
LW_3=-\io(\p_t\ark^{\nu\alpha})(\p_\nu\p_t^4 h)(\ark^{\mu}_{\alpha}\p_\mu\p_t^4 h)\dy\lesssim \PP_0\|\p\p_t^4 h\|_0\|\park \p_t^4 h\|_0.
\end{equation}
We can write $\p_{\alpha}\p_t^4 h=\ark^{\mu\alpha}\p_{\mu}\p_t^4 h+(\delta^{\mu\alpha}-\ark^{\mu\alpha})\p_{\mu}\p_t^4 h$ and invoke $|\ark-\text{Id}|\leq \epsilon$ to get $\|\p\p_t^4 h\|_0\lesssim \|\park \p_t^4 h\|_0$. In consequence, 
\begin{equation}\label{LW3}
LW_3\lesssim \PP_0\|\park \p_t^4 h\|_0^2.
\end{equation}

It remains to estimate $LW_1$, i.e., $\|F_5\|_0$. 

\begin{itemize}
\item $\|\p_t^4(\p_t\ark^{\nu\alpha}\p_{\nu}v_{\alpha})\|_0$: There are two terms containing 5 derivatives: $\p_t^5\ark^{\nu\alpha}\p_{\nu}v_{\alpha}$ and $\p_t\ark^{\nu\alpha}\p_t^4\p_{\nu}v_{\alpha}$. The rest terms are of $\leq 4$ derivatives and hence controlled. By \eqref{circile vari estimate} we know that $\|\p_t^5\ark\|_0\lesssim \PP_0$, which gives $\|\p_t^5\ark^{\nu\alpha}\p_{\nu}v_{\alpha}\|_0\lesssim\PP_0\|\p v\|_2.$ As for the second term, we invoke the second equation of \eqref{linearr} to get $\p_t^4\p v=-\p_t^3 \p(\park h+ge_3)=-\p_t^3 \p(\park h)$, so we have
\begin{equation}\label{LW11}
\|\p_t^4(\p_t\ark^{\nu\alpha}\p_{\nu}v_{\alpha})\|_0\lesssim \PP_0\left(\|\p v\|_2+\|\p_t^3 \p(\park h)\|_0\right).
\end{equation}

\item $\|\p_t^4(\p_t\sigma\p_t h)\|_0$: Expanding all the terms, and then use the previous estimates for $\leq 4$ derivative and invoking \eqref{circile vari estimate}, we have
\begin{equation}\label{LW12}
\begin{aligned}
&~~~~\|\p_t^4(\p_t\sigma\p_t h)\|_0\\
&\lesssim\|\p_t^5\sigma\|_0\|\p_t h\|_2+\|\p_t^4\sigma\|_0\|\p_t^2 h\|_2+\|\p_t^3\sigma \|_1\|\p_t^3 h\|_1+\|\p_t^2 \sigma\|_{L^{\infty}}\|\p_t^4h\|_0+\|\p_t\sigma\|_{L^{\infty}}\|\p_t^5 h\|_0  \\
&\lesssim \PP_0 (\|\p_t h \|_2+\|\p_t^2 h\|_2+\|\p_t^3 h\|_1 +\|\p_t^4 h\|_0+\|\p_t^5 h\|_0 ).
\end{aligned}
\end{equation}
Also, one can control $[\p_t^4,\sigma]\p_t^2 h$ in exactly the same way, so we omit the details.

\item $[\p_t^4,\Delta_{\ark}] h$: A direct computation gives
\begin{align*}
[\p_t^4,\Delta_{\ark}] h&=\p_t^4 (\ark^{\nu\alpha}\p_\nu(\ark^{\mu}_{\alpha}\p_\mu h))-\ark^{\nu\alpha}\p_\nu(\ark^{\mu}_{\alpha}\p_\mu\p_t^4 h) \\
&=\p_t^4 (\ark^{\nu\alpha}\p_\nu(\ark^{\mu}_{\alpha}\p_\mu h))-\ark^{\nu\alpha}\p_\nu\p_t^4 (\ark^{\mu}_{\alpha}\p_\mu h)+\ark^{\nu\alpha}\p_\nu\p_t^4 (\ark^{\mu}_{\alpha}\p_\mu h)-\ark^{\nu\alpha}\p_\nu(\ark^{\mu}_{\alpha}\p_\mu\p_t^4 h)\\
&=[\p_t^4,\ark^{\nu\alpha}]\p_{\nu}(\ark^{\mu}_{\alpha}\p_\mu h)+\ark^{\nu}_{\alpha}\p_\nu([\p_t^4,\ark^{\mu}_{\alpha}]\p_\mu h)\\
&=\sum_{l=1}^{4}(\p_t^{l}\ark^{\nu\alpha})(\p_t^{4-l}\p_{\nu}(\ark^{\mu}_{\alpha}\p_\mu h)) +\ark^{\nu}_{\alpha}\p_\nu\left((\p_t^{l}\ark^{\mu}_{\alpha})(\p_t^{4-l}\p_\mu h)\right).
\end{align*}
Therefore, 
\begin{equation}\label{LW13}
\begin{aligned}
\|[\p_t^4,\Delta_{\ark}] h\|_0&\lesssim \|\p_t^4 \ar\|_0\|\eta\|_{\hc}\|h\|_{\hc}+\|\p_t^3\ar\|_1\|\p_t\park h\|_2+\|\p_t^2 \ar\|_{2}\|\p_t^2\park h\|_1+\|\p_t \ar\|_1\|\p_t^3\park h\|_1 \\
&+\|\ar\|_{L^{\infty}}(\|\p\p_t^4 \ar\|_0\|\p h\|_{L^{\infty}}+\|\p\p_t^3\ar\|_0\|\p\p_t h\|_2+\|\p\p_t^2\ar\|_1\|\p\p_t^2h\|_1+\|\p\p_t\ar\|_2\|\p\p_t^3 h\|_0)\\
&\lesssim \PP_0 \cdot (P(\EE^{(n+1)}-W^{(n+1)})+\|\p_t^3\park h\|_1).
\end{aligned}
\end{equation}
\end{itemize}
Combining \eqref{LW11}-\eqref{LW13}, one has
\begin{equation}\label{LW1}
LW_1\lesssim \frac{1}{\sqrt{\sigma}}\PP_0\cdot P(\EE^{(n+1)})\|\p_t^5 h\|_0.
\end{equation}
Summing up \eqref{h52}, \eqref{LW0}, \eqref{LW2}, and \eqref{LW3}, we get the estimates for the wave equation \eqref{hwave5r}
\begin{equation}\label{h50}
\frac{d}{dt}\left(\io\sigma|\p_t^5 h|^2+|\park\p_t^4 h|^2\dy\right)\lesssim \PP_0\cdot P(\EE^{(n+1)}(t)).
\end{equation}
Therefore we finish the control of $W^{(n+1)}$ by \eqref{WW},\eqref{commw} and \eqref{h50}
\begin{equation}\label{Wfinal}
\frac{d}{dt}W^{(n+1)}(t)\lesssim \PP_0\cdot P(\EE^{(n+1)}(t)).
\end{equation}

\subsubsection{Uniform-in-$n$ a priori estimates for the linearized approximation system}

From \eqref{El}, \eqref{Wl}, \eqref{En4final} and \eqref{Wfinal}, we get
\begin{equation}\label{Enfinal}
\EE^{(n+1)}(T)\lesssim \EE^{(n+1)}(0)+\PP_0\int_0^T P(\EE^{(n+1)}(t))\dt,
\end{equation} which gives the uniform-in-$n$ a priori estimates 
\[
\sup_{0\leq t\leq T_\kk}\EE^{(n+1)}(t)\lesssim_\kk \PP_0
\]for the linearized approximation system \eqref{linearr} (also for \eqref{linearn}) with the help of Gronwall-type inequality in Tao \cite{tao2006nonlinear}.
\begin{flushright}
$\square$
\end{flushright}

\subsection{Construction of the solutions to the linearized approximation system} \label{fixed-point}

In this subsection we are going to construct the solutions to the linearized approximation system \eqref{linearr}:
\[
\begin{cases}
\p_t\eta=v+\psir~~~&\text{ in }\Omega; \\
\p_t v=-\park h-ge_3~~~&\text{ in }\Omega; \\
\divr v=-\sigma \p_t h~~~&\text{ in }\Omega; \\
h=0~~~&\text{ on }\Gamma; \\
(\eta,v, h)|_{t=0}=(\text{Id},v_0. h_0),
\end{cases}
\]
given that $\er,\ar,\psir,\sigma$ satisfying Lemma \ref{apriorir}. 

\subsubsection{Function space and Solution map}

\paragraph*{Definition} (Norm, Function space and Contraction) 

We define the norm $$\|\cdot\|_{Z^r}:=\sum_{s=0}^r\sum_{k+l=s}\|\p_t^k\p^{l}\cdot\|_0$$ and define the function space
\begin{equation}\label{XX}
\begin{aligned}
\X(M,T):=&
\bigg\{(\xi, w,\pi):(w, \xi)|_{t=0}=(v_0,\text{Id}),\\
\sup_{t\in [0,T]}&\left(\|w(t),\p_t\pi(t)\|_{Z^4}+\|\nab_{\ak^{(n)}} \pi(t)\|_{L^\infty}+\|\p\nab_{\ak^{(n)}} \pi(t), \p_t\nab_{\ak^{(n)}} \pi(t)\|_{Z^3}+\|\p_t\xi(t)\|_{Z^3}+\|\p^2\xi(t)\|_{Z^2}+\|\p\xi(t)\|_{L^{\infty}}\right)\leq M\bigg\}.
\end{aligned}
\end{equation} 
We notice here that for given $M>0,T>0$, $\X(M,T)$ is a Banach space. 
\begin{rmk}
As mentioned in the remark after \eqref{eq h}, the quantity $\|\nab_{\ak^{(n)}} \pi(t)\|_{L^\infty}+\|\p\nab_{\ak^{(n)}} \pi(t), \p_t\nab_{\ak^{(n)}} \pi(t)\|_{Z^3}$ can be replaced by $\|\nab_{\ak^{(n)}} \pi(t)\|_{Z^4}$ if $\Omega$ is bounded. 
\end{rmk}

We then define the solution map $\Xi:\X(M,T)\to\X(M,T)$ by
\begin{equation}
\begin{aligned}
\Xi:\X(M,T)&\to\X(M,T) \\
(w,\pi,\xi)&\mapsto (v,h,\eta).
\end{aligned}
\end{equation} The image $(v,h,\eta)$ is defined as follows:
\begin{enumerate}
\item Define $\eta$ by 
\begin{equation}\label{defeta}
\p_t \eta=w+\psir,~~~\eta(0)=\text{Id}.
\end{equation}

\item Define $v$ by 
\begin{equation}\label{defv}
\p_t v=-\park\pi-ge_3,~~~v(0)=v_0.
\end{equation}

\item Define $h$ by the solution of the following wave equation
\begin{equation}\label{hwavec0}
\begin{cases}
\sigma\p_t^2 h-\Delta_{\ark}h=-\p_t \ark^{\nu\alpha}\p_\nu v_{\alpha}-\p_t\sigma\p_t h~~~&\text{ in }\Omega,\\
h=0~~~&\text{ on }\Gamma,\\
(h,\p_t h)|_{t=0}=(h_0,h_1).
\end{cases}
\end{equation} 
The existence of this linear wave equation can be shown by adapting the method provided in Lax-Phillips \cite{Lax} after turning it into a system of  hyperbolic equations. Also, this solution lies in the space $\X(M,T)$ owing to \eqref{waveZ4} in the upcoming subsection. 
\end{enumerate}
 Here, we have to show that a solution for \eqref{defeta}-\eqref{hwavec0} with $(w,\pi,\xi)=(v,h,\eta)$ implies a solution for \eqref{linearr}. It suffices to show that we can recover the third equation of \eqref{linearr}. First, since $\pi=h$, \eqref{defv} reads
$
\p_t v = -\nab_{\ark}h - ge_3.
$
We take $\di_{\ark}$ on both sides and get 
\begin{align}
\p_t (\di_{\ark} v) + \Delta_{\ark} h = \p_t \ark^{\nu\alpha}\p_\nu v_\alpha. \label{defv div}
\end{align}
Moreover, \eqref{hwavec0} implies $\lap_{\ark} h = \p_t(\sigma \p_th)+\p_t \ark^{\nu\alpha} \p_\nu v_\alpha$, and by plugging this to \eqref{defv div} we get
\begin{align}
\p_t (\divr v+\sigma \p_th)=0
\end{align}
and hence $\divr v+\sigma \p_th = $ constant. This constant must be $0$ since $\divr v+\sigma \p_th|_{t=0}=0$. 
\subsubsection{Construct the solution: Contraction Mapping Theorem}
Now we need to verify
\begin{enumerate}
\item $\Xi$ is a self-mapping of $\X$, 

\item $\Xi$ is a contraction on $\X$.
\end{enumerate} Once these two properties are proved, we can apply the Contraction Mapping Theorem to $\Xi$ to get there exists a unique fixed point $(v,h,\eta)$ of $\Xi$ which solves the linearized system \eqref{linearr}.

\bigskip

First we verify $\Xi$ is a self-mapping of $\X$.

\paragraph*{Estimates for $\eta$:} A direct computation gives
\begin{align}
\label{etaZ1} \|\p\eta\|_{L^{\infty}}&\leq\|\p\eta(0)\|_{L^{\infty}}+\int_0^T\|\p(w+\psir)\|_{L^{\infty}}\dt\leq 1+\int_0^T\|w\|_{Z^4}+\|\p\psir\|_2\dt,\\
\label{etaZ2} \|\p^2\eta\|_{Z^2}&\leq\|\p^2\eta(0)\|_{Z^2}+\int_0^T\|\p^2(w+\psir)\|_{Z^2}\leq\int_0^T \|w\|_{Z^4}+\|\psir\|_{Z^4}\dt,\\
\label{etaZ3} \|\p^{3-k}\p_t^{k}\p_t\eta\|_0&\lesssim\|\p^{3-k}\p_t^{k}\p_t\eta(0)\|_0+\int_0^T\|\p^{3-k}\p_t^{k+1}(w+\psir)\|_0\lesssim\int_0^T\|w\|_{Z^4}+\|\psir\|_{Z^4}\dt.
\end{align}

\paragraph*{Estimates for $v$}: First we have for $l=1,2,3,4$:
\begin{align}
\|v\|_0&\leq\|v_0\|_0+\int_0^T\|\p_t v(0)\|_0+\left(\int_0^t\|\p_t^2 v(\tau)\|_0d\tau\right)\dt\leq\|v_0\|_0+T\|\p_t v(0)\|_0+T\int_0^T\|\p_t\park\pi\|_{0}\dt \\
\|\p_t^l v\|_0&\leq\|\p_t^l v(0)\|_0+\int_0^T\|\p_t^{l}\park\pi\|_0\dt\leq\|\p_t^l v(0)\|_0+\int_0^T\|\p_t\park\pi\|_{Z^3}\dt~~~1\leq l\leq 4,\\
\|\park h\|_{L^{\infty}}&\leq g+\|\p_t v\|_{L^{\infty}}.
\end{align}
For the space-time derivatives, we also have
\begin{align}
\|\p^4 v\|_0&\leq \|v_0\|_4+\int_0^T\|\p_t\p^4 v\|_0\dt\leq\|v_0\|_4+\int_0^T\|\p^4\park\pi\|_0\dt\leq \|v_0\|_4+\int_0^T\|\p\park\pi\|_{Z^3}\dt, \\
\|\p_t^l\p^{4-l}v\|_{0}&\leq\|\p_t^l\p^{4-l}v\|_0+\int_0^T\|\p_t^l\p^{4-l}\park\pi\|_0\dt\leq\|v(0)\|_{Z^4}+\int_0^T\|\p_t\park\pi\|_{Z^3}\dt.
\end{align}
Therefore,
\begin{equation}\label{vZ4}
\|v\|_{Z^4}\leq\|v_0\|_{Z^4}+T\|\p_t v(0)\|_0+\int_0^T\|\park \pi(t)\|_{L^\infty}+\|\p\park \pi(t), \p_t\park \pi(t)\|_{Z^3}\lesssim (1+T)\PP_0+TM.
\end{equation}

\paragraph*{Estimates for $h$:} It suffices to estimate $\|\p_t\park h\|_{Z^3}$ and $\|\p\park h\|_{Z^3}$ via the wave equation of $h$, i.e., \eqref{hwavec0}. Again we can apply the same method as in Section \ref{wave5} to derive
\begin{equation}\label{waveZ4}
\|\p_t\park h\|_{Z^3}+\|\p\park h\|_{Z^3}\lesssim_M \PP_0+\PP_0\int_0^T P(\|\p\eta\|_{L^\infty}, \|\p_t^2 \eta\|_{2}, \|v\|_{Z^4}, \|\park \pi(t)\|_{L^\infty},\|\p\park \pi(t), \p_t\park \pi(t)\|_{Z^3}, \|\p_t h\|_{Z^4})\dt.
\end{equation}
Combining \eqref{etaZ1}-\eqref{etaZ3}, \eqref{vZ4} and \eqref{waveZ4}, we obtain that the solution map $\Xi$ is a self-map of $\X$ after applying the Gronwall's inequality.

Next we prove $\Xi:\X(M,T)\to \X(M,T)$ is a contraction. Given $(w_1,\pi_1,\xi_1),~(w_2,\pi_2,\xi_2)\in \X(M,T)$ and their images under $\Xi$ $(v_1,h_1,\eta_1),~(v_2,h_2,\eta_2)$, we define 
\[
[w]:=w_1-w_2,~[\pi]:=\pi_1-\pi_2,~[\xi]:=\xi_1-\xi_2;~~[v]:=v_1-v_2,~[h]:=h_1-h_2,~[\eta]:=\eta_1-\eta_2.
\]

From \eqref{defeta}, \eqref{defv} and \eqref{hwavec0}, we can derive the equations for $([v],[h],[\eta])$ with initial data $(\mathbf{0},0,\mathbf{0})$:
\begin{align*}
\p_t[\eta]&=[w], \\
\p_t[v]&=\park[\pi], \\
\sigma\p_t^2[h]-\Delta_{\ark}[h]&=-\p_t\ark^{\nu\alpha}\p_\nu[v]_{\alpha}-\p_t\sigma \p_t[h],~~~[h]|_{\Gamma}=0.
\end{align*} Similarly as above we can derive the estimates

\begin{equation}\label{gapZ4}
\begin{aligned}
&~~~~\|\p[\eta]\|_{L^{\infty}}+\|\p^2[\eta]\|_{Z^2}+\|\p_t[\eta]\|_{Z^3}+\|\park [\pi]\|_{L^\infty},\|\p\park [\pi], \p_t\park [\pi]\|_{Z^3}+\|[v]\|_{Z^4} \\
& \lesssim_M \PP_0\int_0^T P(\|\p[\xi]\|_{L^{\infty}},\|\p^2[\xi]\|_{Z^2},\|\p_t[\xi]\|_{Z^3},\|\park [\pi]\|_{L^\infty},\|\p\park [\pi], \p_t\park [\pi]\|_{Z^3},\|[w]\|_{Z^4})\dt.
\end{aligned} 
\end{equation}
Therefore, choosing $T_{\kk}>0$ sufficiently small such that RHS of \eqref{gapZ4} is bounded by $$\frac{1}{2}\left(\|\p[\xi]\|_{L^{\infty}}+\|\p^2[\xi]\|_{Z^2}+\|\p_t[\xi]\|_{Z^3}+\|\park[\pi]\|_{Z^4}+\|\p_t[\pi]\|_{Z^4}+\|[w]\|_{Z^4}\right),$$ we prove that $\Xi:\X(M,T_\kk)\to \X(M,T_{\kk})$ is a contraction self-map. By the Contraction Mapping Theorem, we know $\Xi$ has a unique fixed point $(v,h,\eta)\in \X(M, T_\kk)$ which is the solution to the linearized approximation system \eqref{linearr}.

\subsection{Iteration and convergence of the solutions to the linearized system}

Up to now we have constructed a sequence of solutions $\{(v^{(n)},h^{(n)},\eta^{(n)})\}_{n=1}^{\infty}$ which solves the $n$-th linearized $\kk$-approximation system \eqref{linearn}. The last step in this section is to prove that $\{(v^{(n)},h^{(n)},\eta^{(n)})\}_{n=1}^{\infty}$ converges in some strong Sobolev norm, and thus produce a solution $(v,h,\eta)$ to the nonlinear $\kk$-approximation system \eqref{app1}.

Let $n\geq 3$, and define 
\begin{equation}\label{gapl1}
[v]^{(n)}:=v^{(n+1)}-v^{(n)},~~[h]^{(n)}:=h^{(n+1)}-h^{(n)},~~[\eta]^{(n)}:=\eta^{(n+1)}-\eta^{(n)}, 
\end{equation}and
\begin{equation}\label{gapl2}
[a]^{(n)}:=a^{(n)}-a^{(n-1)},~~[\psi]^{(n)}:=\psi^{(n)}-\psi^{(n-1)}. 
\end{equation} Then these quantities satisfy the following system with \textit{vanishing initial data and no gravity term}:
\begin{equation}\label{linearg}
\begin{cases}
\p_t[\eta]^{(n)}=[v]^{(n)}+[\psi]^{(n)}~~~&\text{ in }\Omega \\
\p_t[v]^{(n)}=-\nabla_{\ak^{(n)}}[h]^{(n)}-\nabla_{[\ak]^{(n)}} h^{(n)}~~~&\text{ in }\Omega \\
\text{div}_{\ak^{(n)}}[v]^{(n)}=-\text{div}_{[\ak]^{(n)}}v^{(n)}-e'(h^{(n)})\p_t [h]^{(n)}-(e'(h^{(n)})-e'(h^{(n-1)}))\p_t h^{(n)}~~~&\text{ in }\Omega \\
[h]^{(n)}=0 ~~~&\text{ on }\Gamma \\
\end{cases}
\end{equation}

We will prove the following energy converges to 0 as $n\to\infty$ for all $t\in[0,T]$
\begin{equation}\label{gapEn}
[\EE]^{(n)}(t):=\sum_{k=0}^3\|\p_t^{3-k}[v]^{(n)}(t)\|_k^2+\|\p_t^{3-k}[h]^{(n)}(t)\|_k^2+\|[\eta]^{(n)}(t)\|_3^2+\|[a]^{(n)}(t)\|_2^2.
\end{equation}
\begin{rmk}
Since the gravity term has been cancelled in \eqref{linearg}, we then could directly include the standard $H^3$ Sobolev norm of $[h]$ in $[\EE]$ instead of $\|\p [h]\|_{L^{\infty}}^2+\|\p^2[h]\|_1^2$.
\end{rmk}

\subsubsection{Estimates of $[a]$, $[\psi]$ and $[\eta]$}

By definition, we have
\begin{align*}
[a]^{(n)\mu\nu}(T)&=\int_0^T\p_t(a^{(n)\mu\nu}-a^{(n-1)\mu\nu})\dt\\
&=-\int_0^T [a]^{(n)\mu\gamma}\p_\beta\p_t\eta^{(n)}_{\gamma}a^{(n)\beta\nu}+a^{(n-1)\mu\gamma}\p_\beta\p_t[\eta]^{(n-1)}_{\gamma}a^{(n)\beta\nu}+a^{(n-1)\mu\gamma}\p_\beta\p_t\eta^{(n-1)}_{\gamma}[a]^{(n)\beta\nu},
\end{align*} which gives
\begin{equation}
\|[a]^{(n)}(T)\|_{2}\lesssim\PP_0\int_0^T\|[a]^{(n)}(t)\|_2^2\|\p_t[\eta]^{(n-1)}\|_3\dt\lesssim\PP_0\int_0^T\|[a]^{(n)}(t)\|_2^2(\|[v]^{(n-1)}\|_3+\|[\psi]^{(n-1)}\|_3))\dt.
\end{equation}
As for $[\psi]^{(n)}$, it satisfies $-\Delta[\psi]^{(n)}=0$ subject to the following boundary condition
\begin{align*}
[\psi]^{(n)}=\TP^{-1}\mathbb{P}\bigg(&\TL[\eta]^{(n-1)}_{\beta}\ak^{(n)i\beta}\TP_i\lkk^2v^{(n)}+\TP\eta^{(n-1)}_{\beta}[\ak]^{(n)i\beta}\TP_i\lkk^2v^{(n)}+\TP\eta^{(n-1)}_{\beta}\ak^{(n-1)i\beta}\TP_i\lkk^2[v]^{(n-1)} \\
&-\TL\lkk^2[\eta]^{(n-1)}_{\beta}\ak^{(n) i\beta}\TP_i v^{(n)}-\TL\lkk^2\eta^{(n-1)}_{\beta}[\ak]^{(n) i\beta}\TP_i v^{(n)}-\TL\lkk^2\eta^{(n-1)}_{\beta}\ak^{(n-1) i\beta}\TP_i [v]^{(n-1)}\bigg).
\end{align*}
By the standard elliptic estimates, we have the control for $[\psi]^{(n)}$
\begin{equation}\label{psig3}
\|[\psi]^{(n)}\|_3^2\lesssim |[\psi]^{(n)}|_{2.5}^2\lesssim \PP_0\left(\|[\eta]^{(n-1)}\|_3^2+\|[v]^{(n-1)}\|_2^2+\|[\ak]^{(n)}\|_1^2\right).
\end{equation}
Therefore, we obtain
\begin{equation}\label{ag2}
\sup_{[0,T]}\|[a]^{(n)}\|_2^2\lesssim \PP_0T^2\left(\|[a]^{(n)},[a]^{(n-1)}\|_{L_t^{\infty}H^2}+\|[v]^{(n-1)},[v]^{(n-2)},[\eta]^{(n-2)}\|_{L_t^{\infty}H^3}^2\right),
\end{equation} and the bound for $[\eta]$ combining with $\p_t [\eta]^{(n)}=[v]^{(n)}+[\psi]^{(n)}$:
\begin{equation}\label{etag3}
\sup_{[0,T]}\|[\eta]^{(n)}\|_3^2\lesssim\PP_0T^2\left(\|[a]^{(n)}\|_{L_t^{\infty}H^2}+\|[v]^{(n)},[v]^{(n-1)},[\eta]^{(n-1)}\|_{L_t^{\infty}H^3}^2\right)
\end{equation}

Similar as in Lemma \ref{etapsi} and Lemma \ref{apriorir}, one can get estimates for the time derivatives of $[\eta]$ and $[\psi]$
\begin{align}
\|[\p_t\psi]^{(n)}\|_3^2&\lesssim\PP_0\left(\|[a]^{(n)}\|_{2}^2+\|[\p_t v]^{(n-1)}\|_{2}^2+\|[v]^{(n-1)},[\eta]^{(n-1)}\|_{3}^2\right)\\
\|[\p_t^2\psi]^{(n)}\|_2^2&\lesssim\PP_0\left(\|[a]^{(n)}\|_{2}^2+\|[\p_t^2 v]^{(n-1)}\|_{1}^2+\|[\p_t v]^{(n-1)}\|_{2}^2+\|[v]^{(n-1)},[\eta]^{(n-1)}\|_{3}^2\right)\\
\|[\p_t\eta]^{(n)}\|_3^2&\lesssim\PP_0T^2\left(\|[a]^{(n)},[\p_tv]^{(n)},[\p_tv]^{(n-1)}\|_{L_t^{\infty}H^2}+\|[v]^{(n)},[v]^{(n-1)},[\eta]^{(n-1)}\|_{L_t^{\infty}H^3}^2\right)\\
\|[\p_t^2\eta]^{(n)}\|_2^2&\lesssim\PP_0T^2\left(\|[\p_t^2 v]^{(n),(n-1)}\|_{L_t^{\infty}H_1}^2+\|[a]^{(n)},[\p_tv]^{(n),(n-1)}\|_{L_t^{\infty}H^2}+\|[v]^{(n),(n-1)},[\eta]^{(n-1)}\|_{L_t^{\infty}H^3}^2\right)\\
\|[\p_t^3\eta]^{(n)}\|_1^2&\lesssim\PP_0\left(\|[\p_t^2 v]^{(n),(n-1)}\|_{L_t^{\infty}H_1}^2+\|[a]^{(n)},[\p_tv]^{(n),(n-1)}\|_{L_t^{\infty}H^2}+\|[v]^{(n),(n-1)},[\eta]^{(n-1)}\|_{L_t^{\infty}H^3}^2\right).
\end{align}

\subsubsection{Estimates of $[h]$}

Taking $\tilde{J}^{(n)}\text{div}_{\ak^{(n)}}$ in the second equation of \eqref{linearg}, we get an analogous wave equation for $[h]$:
\begin{equation}
\begin{aligned}
e'(h^{(n)})\tilde{J}^{(n)}\p_t^2 [h]^{(n)}-\p_{\nu}(\tilde{E}^{\nu\mu}\p_{\mu}[h]^{(n)})&=\tilde{J}^{(n)}\ak^{(n)\nu\alpha}\p_{\nu}({[\ak]^{(n)\mu}}_{\alpha}\p_\mu h^{(n)}) \\
&~~~~-\tilde{J}^{(n)}\p_t\left((e'(h^{(n)})-e'(h^{(n-1)}))\p_t h^{(n)}\right) \\
&~~~~-\tilde{J}^{(n)}(\p_t\ak^{(n)\nu\alpha})\p_\nu [v]^{(n)}_{\alpha},
\end{aligned}
\end{equation}where $\tilde{E}^{\nu\mu}:=\tilde{J}^{(n)}\ak^{(n)\nu\alpha}{\ak^{(n)\mu}}_{\alpha}$.

One can apply the similar method in Section \ref{hapriori} and use the estimates of $[\eta],[\psi]$ to obtain the following energy estimates
\begin{equation}\label{hg3}
\sum_{k=0}^2\|\p_t^k [h]^{(n)}\|_{3-k}^2+\|\sqrt{e'(h^{(n)})}\p_t^3[h]^{(n)}\|_0^2\lesssim\PP_0\int_0^T \left([\EE]^{(n)}(t)+[\EE]^{(n-1)}(t)\right)\dt.
\end{equation}

\subsubsection{The div-curl estimates}

From Hodge's decomposition inequality Lemma \ref{hodge}, we have
\begin{align*}
\|[v]^{(n)}\|_3^2&\lesssim \|[v]^{(n)}\|_0^2+\|\dive[v]^{(n)}\|_2^2+\|\curl[v]^{(n)}\|_2^2+|[v]^{(n)}\cdot N|_{2.5} \\
\|[\p_tv]^{(n)}\|_2^2&\lesssim \|[\p_tv]^{(n)}\|_0^2+\|\dive[\p_tv]^{(n)}\|_1^2+\|\curl[\p_tv]^{(n)}\|_1^2+|[\p_tv]^{(n)}\cdot N|_{1.5} \\
\|[\p_t^2v]^{(n)}\|_1^2&\lesssim \|[\p_t^2v]^{(n)}\|_0^2+\|\dive[\p_t^2v]^{(n)}\|_0^2+\|\curl[\p_t^2v]^{(n)}\|_0^2+|[\p_tv]^{(n)}\cdot N|_{0.5} \\
\end{align*}

The $L^2$-norm can be bounded in the same way as in Section \ref{divcurl} and the boundary term can be reduced to the tangential estimates for $[v]$ and its time derivative. As for the $\curl$part, we apply $\curl_{\ak^{(n)}}$ to the second equation of \eqref{linearg} to get the evolution equation of $\curl_{\ak^{(n)}[v]^{(n)}}$
\begin{align}
\label{curlevog} \p_t(\curl_{\ak^{(n)}[v]^{(n)}})_{\lambda}=\epsilon_{\lambda\mu\alpha}\p_t\ak^{(n)\nu\mu}\p_\nu[v]^{(n)}_{\alpha}-\epsilon_{\lambda\mu\alpha}\p_t[\ak]^{(n)\nu\mu}\p_\nu v^{(n)}_{\alpha}.
\end{align} Applying $D^2=\p^2,\p\p_t$ or $\p_t^2$ to \eqref{curlevog}, and mimicking the proof in Section \ref{divcurl},  one can get
\begin{equation}\label{curlg}
\begin{aligned}
\|\curl [v]^{(n)}\|_2^2&\lesssim\epsilon\|[v]^{(n)}\|_3^2+\PP_0 T^2\left(\|[v]^{(n),(n-1)},[\eta]^{(n-1)}\|_{L_t^{\infty}H^3}+\|[\ak]^{(n),(n-1)}\|_{L_t^{\infty}H^2}\right)\\
&\lesssim\epsilon\|[v]^{(n)}\|_3^2+\PP_0T^2\sup_{[0,T]}[\EE]^{(n),(n-1)}(t) \\
\|\curl [\p_t v]^{(n)}\|_1^2&\lesssim\epsilon\|[\p_tv]^{(n)}\|_2^2+\PP_0T^2\sup_{[0,T]}[\EE]^{(n),(n-1)}(t) \\
\|\curl [\p_t^2 v]^{(n)}\|_0^2&\lesssim\epsilon\|[\p_t^2v]^{(n)}\|_1^2+\PP_0T^2\sup_{[0,T]}[\EE]^{(n),(n-1)}(t) .
\end{aligned}
\end{equation} Similar results hold for $\dive$control by using the same method as in Section \ref{divcurl}, so we only list the result here
\begin{equation}\label{divg}
\begin{aligned}
&~~~~\|\dive[v]^{(n)}\|_2^2+\|\dive[\p_tv]^{(n)}\|_1^2\|+\dive[\p_t^2v]^{(n)}\|_0^2\\
&\lesssim \PP_0T^2\sup_{t\in[0,T]}\left([\EE]^{(n)}(t)+[\EE]^{(n-1)}(t)\right) .
\end{aligned}
\end{equation}

\subsubsection{Tangential estimates of $[\p_t^kv]$ for $k\geq 1$}

Let $\dd^3=\TP^2\p_t,\TP\p_t^2,\p_t^3$. Using the same method as in Section \ref{tgtime} and Section \ref{linear4} (Step 4), we can derive the estimates
\begin{equation}\label{tgtimeg}
\sum_{k=1}^3\|\TP^{3-k}\p_t^k [v]^{(n)}\|_{0}^2+\|\TP^{3-k}\p_t^k [h]^{(n)}\|_0^2\lesssim \PP_0 \int_0^T [\EE]^{(n)}(t)+[\EE]^{(n-1)}(t)+[\EE]^{(n-2)}(t)\dt.
\end{equation}

\subsubsection{Tangential estimates of $[v]$: Alinhac's good unknown}

We adopt the same method as in Section \ref{tgspace}. For each $n$, we define the Alinhac's good unknowns by 
\begin{equation}
\VV^{(n+1)}=\TP^3v^{(n+1)}-\TP^3\ek^{(n)}\cdot\nabla_{\ak^{(n)}}v^{(n+1)},~~\HH^{(n+1)}=\TP^3h^{(n+1)}-\TP^3\ek^{(n)}\cdot\nabla_{\ak^{(n)}}h^{(n+1)}.
\end{equation} Their difference is denoted by 
\begin{align*}
[\VV]^{(n)}:=\VV^{(n+1)}-\VV^{(n)}, [\HH]^{(n)}:=\HH^{(n+1)}-\HH^{(n)}.
\end{align*}

Similarly as in Section \ref{tgspace}, we can derive the analogous version of \eqref{goodapp1} as
\begin{equation}\label{goodconv}
\p_t[\VV]^{(n)}+\nabla_{\ak^{(n)}}[\HH]^{(n)}=-\nabla_{[\ak]^{(n)}}\HH^{(n)}+\FF^{(n)},
\end{equation}subject to the boundary data
\begin{equation}\label{bdryconv}
[\HH]^{(n)}|_{\Gamma}=-\left(\TP^3\ek^{(n)}_{\beta}\ak^{(n)3\beta}+\TP^3[\ek]^{(n-1)}_{\beta}\ak^{(n)3\beta}+\TP^3\ek^{(n-1)}_{\beta}[\ak]^{(n)3\beta}\right),
\end{equation}and the compressibility equation
\begin{equation}\label{divconv}
\nabla_{\ak^{(n)}}\cdot[\VV]^{(n)}=-\nabla_{[\ak]^{(n)}}\cdot\VV^{(n)}+\GG^{(n)},
\end{equation}where
\begin{align*}
\FF^{(n)\alpha}&=\p_t\left(\TP^3[\ek]^{(n-1)}_{\beta}\ak^{(n)\mu\beta}\p_{\mu}v^{(n+1)}_{\alpha}+\TP^3\ek^{(n-1)}_{\beta}[\ak]^{(n)\mu\beta}\p_{\mu}v^{(n+1)}_{\alpha}+\TP^3\ek^{(n-1)}_{\beta}\ak^{(n)\mu\beta}\p_{\mu}[v]^{(n)}_{\alpha}\right)\\
&~~~~+[\ak]^{(n)\mu\beta}\p_{\mu}(\ak^{(n)\gamma\alpha}\p_{\gamma}h^{(n+1)})\TP^3\ek^{(n)}_{\beta}+\ak^{(n-1)\mu\beta}\p_{\mu}([\ak]^{(n)\gamma\alpha}\p_{\gamma}h^{(n+1)})\TP^3\ek^{(n)}_{\beta} \\
&~~~~+\ak^{(n-1)\mu\beta}\p_{\mu}(\ak^{(n-1)\gamma\alpha}\p_{\gamma}[h]^{(n)})\TP^3\ek^{(n)}_{\beta}+\ak^{(n-1)\mu\beta}\p_{\mu}([\ak]^{(n)\gamma\alpha}\p_{\gamma}h^{(n)})\TP^3[\ek]^{(n-1)}_{\beta}\\
&~~~~-\left[\TP^2,[\ak]^{(n)\mu\beta}\ak^{(n)\gamma\alpha}\TP\right]\p_{\gamma}\ek^{(n)}_{\beta}\p_{\mu}h^{(n+1)}-\left[\TP^2,\ak^{(n-1)\mu\beta}[\ak]^{(n)\gamma\alpha}\TP\right]\p_{\gamma}\ek^{(n)}_{\beta}\p_{\mu}h^{(n+1)}\\
&~~~~-\left[\TP^2,\ak^{(n-1)\mu\beta}\ak^{(n-1)\gamma\alpha}\TP\right]\p_{\gamma}[\ek]^{(n-1)}_{\beta}\p_{\mu}h^{(n+1)}-\left[\TP^2,\ak^{(n-1)\mu\beta}\ak^{(n-1)\gamma\alpha}\TP\right]\p_{\gamma}\ek^{(n-1)}_{\beta}\p_{\mu}[h]^{(n)}\\
&~~~~-\left[\TP^3,[\ak]^{(n)\mu\alpha},\p_{\mu}h^{(n+1)}\right]-\left[\TP^3,\ak^{(n-1)\mu\alpha},\p_{\mu}[h]^{(n)}\right],
\end{align*}
and 
\begin{align*}
\GG^{(n)}&=\TP^3(\dive_{\ak^{(n)}}[v]^{(n)}-\dive_{[\ak]^{(n)}}v^{(n)}) \\
&~~~~-\left[\TP^2,[\ak]^{(n)\mu\beta}\ak^{(n)\gamma\alpha}\TP\right]\p_{\gamma}\ek^{(n)}_{\beta}\p_{\mu}v^{(n+1)}_{\alpha}-\left[\TP^2,\ak^{(n-1)\mu\beta}[\ak]^{(n)\gamma\alpha}\TP\right]\p_{\gamma}\ek^{(n)}_{\beta}\p_{\mu}v^{(n+1)}_{\alpha} \\
&~~~~-\left[\TP^2,\ak^{(n-1)\mu\beta}\ak^{(n-1)\gamma\alpha}\TP\right]\p_{\gamma}[\ek]^{(n-1)}_{\beta}\p_{\mu}v^{(n+1)}_{\alpha} -\left[\TP^2,\ak^{(n-1)\mu\beta}\ak^{(n)\gamma\alpha}\TP\right]\p_{\gamma}\ek^{(n-1)}_{\beta}\p_{\mu}[v]^{(n)}_{\alpha} \\
&~~~~-\left[\TP^3,[\ak]^{(n)\mu\alpha},\p_{\mu}v^{(n+1)}_{\alpha}\right]-\left[\TP^3,\ak^{(n-1)\mu\alpha},\p_{\mu}[v]^{(n)}_{\alpha}\right] \\
&~~~~+[\ak]^{(n)\mu\beta}\p_{\mu}(\ak^{(n)\gamma\alpha}\p_{\gamma}v^{(n+1)}_{\alpha})\TP^3\ek^{(n)}_{\beta}+\ak^{(n-1)\mu\beta}\p_{\mu}([\ak]^{(n)\gamma\alpha}\p_{\gamma}v^{(n+1)}_{\alpha})\TP^3\ek^{(n)}_{\beta} \\
&~~~~+\ak^{(n-1)\mu\beta}\p_{\mu}(\ak^{(n-1)\gamma\alpha}\p_{\gamma}[v]^{(n)}_{\alpha})\TP^3\ek^{(n)}_{\beta}+\ak^{(n-1)\mu\beta}\p_{\mu}([\ak]^{(n)\gamma\alpha}\p_{\gamma}v^{(n)}_{\alpha})\TP^3[\ek]^{(n-1)}_{\beta}.
\end{align*}

Multiplying $[\VV]^{(n)}$ in \eqref{goodconv} and integrate by parts in the $[\HH]$ term, we get
\begin{align*}
\frac{1}{2}\frac{d}{dt}\|\VV^{(n)}\|_0^2&=\io [\HH]^{(n)}\left(\nabla_{\ak^{(n)}}\cdot[\VV]^{(n)}-\p_{\mu}\ak^{\mu\alpha}[\VV]^{(n)}_{\alpha}\right)\dy+\io (\FF^{(n)}-\nabla_{[\ak]^{(n)}}\HH^{(n)})\cdot[\VV]^{(n)}\dy\\
&~~~~-\ig [\HH]^{(n)}\ak^{(n)3\alpha}[\VV]^{(n)}_{\alpha}\dS.
\end{align*}
Similarly as in \eqref{tgs2}-\eqref{tgK}, the first two integrals contribute to the energy term
\[
-\frac{1}{2}\frac{d}{dt}\|e'(h^{(n)})\TP^4[h]^{(n)}\|_0^2.
\]
modulo error terms that can be controlled by $\PP_0 ([\EE]^{(n)}+[\EE]^{(n-1)})$.
As for the boundary term, we can mimic the proof in \eqref{Ir}, i.e., integrate $\TP^{0.5}$ by parts, to get
\begin{align*}
&~~~~-\ig [\HH]^{(n)}\ak^{(n)3\alpha}[\VV]^{(n)}_{\alpha}\dS \\
&=\ig \p_3 h\ak^{(n)3\alpha}[\VV]^{(n)}_{\alpha}\left(\TP^3\ek^{(n)}_{\beta}\ak^{(n)3\beta}+\TP^3[\ek]^{(n-1)}_{\beta}\ak^{(n)3\beta}+\TP^3\ek^{(n-1)}_{\beta}[\ak]^{(n)3\beta}\right) \\
&\lesssim  |[\VV]^{(n)}|_{\dot{H}^{-0.5}}\left(\frac{1}{\kk}\PP_0|\TP^2[\eta]^{(n-1)}|_{\dot{H}^{0.5}}+\|[\ak]\|_2\right).
\end{align*}

Summing up all the estimates above and using the analogue of \eqref{alinhac4}, we get
\begin{equation}\label{tgsconv}
\|\TP^3[v]^{(n)}\|_0^2+\|e'(h^{(n)})\TP^3 [h]^{(n)}\|_0^2\lesssim \PP_0T^2 \sup_{t\in[0,T]}\left([\EE]^{(n)}(t)+[\EE]^{(n-1)}(t)+[\EE]^{(n-2)}(t)\right).
\end{equation}

Finally, we combine the estimates for $[\eta],[a],[h],[v]$ and div-curl estimates above and obtain
\[
[\EE]^{(n)}(t)\lesssim \PP_0T^2 \sup_{t\in[0,T]}\left([\EE]^{(n)}(t)+[\EE]^{(n-1)}(t)+[\EE]^{(n-2)}(t)\right).
\] Therefore by choosing $T=T_\kk>0$ sufficiently small (The time $T_\kk$ depends on $\kk>0$ because the uniform-in-$n$ estimates depend on $\kk^{-1}$), we can get
\begin{equation}\label{normconv}
\sup_{[0,T_\kk]}[\EE]^{(n)}(t)\leq\frac{1}{8}\left(\sup_{t\in[0,T]}[\EE]^{(n-1)}(t)+\sup_{t\in[0,T]}[\EE]^{(n-2)}(t)\right),
\end{equation}which implies
\[
\sup_{[0,T_\kk]}[\EE]^{(n)}(t)\leq\frac{1}{2^n}\PP_0\to 0~~~\text{as }n\to\infty.
\]

\subsection{Construction of the solution to the approximation system \eqref{app1}}

\begin{prop}\label{uniformnn}
Suppose the initial data $(v_0,h_0)$ satisfying $\|v_0\|_4 + \|h_0\|_{\hc}\leq M_0$ and the compatibility conditions up to order $4$. Given $\kk>0$, there exists a $T_\kk>0$ such that the nonlinear $\kk$-approximation system \eqref{app1} has a unique solution $(v(\kk),h(\kk),\eta(\kk))$ in $[0,T_\kk]$ satisfying the estimates
\begin{equation}\label{Ell}
\sup_{0\leq t\leq T_\kk}\EE(t)\lesssim C(M_0),
\end{equation}where
\begin{equation}\label{Wll}
\begin{aligned}
\EE(t)&:=\|\p^2\eta(t)\|_2^2+\|\p\eta(t)\|_{L^{\infty}}^2+\sum_{k=0}^4\|\p_t^{4-k}v\|_k^2+\left(\|h\|_{\hc}^2+\sum_{k=0}^3+\|\p_t^{4-k}h\|_k^2\right)+W(t),\\
W(t)&:=\sum_{k=0}^4\|\p_t^{5-k}h(t)\|_k^2+\sum_{k=0}^3\|\p_t^{4-k}\pak h(t)\|_k^2+\|\pak h\|_{L^{\infty}}^2+\|\p\pak h\|_{3}^2.
\end{aligned}
\end{equation}
\end{prop}

\begin{proof}
In Section \ref{fixed-point}, we proved that the linearized system \eqref{linearn} admits a solution $(\eta^{(n+1)}, v^{(n+1)},h^{(n+1)})$ assuming that $$(\eta^{(k)}, v^{(k)},h^{(k)}), \,\,\, k\leq n$$ are known and satisfying \eqref{Eninduction}. 
Moreover, in light of \eqref{Enfinal} and \eqref{normconv}, we obtain the strong convergence of the sequence of approximation solutions $\{(\eta^{(n+1)}, v^{(n+1)},h^{(n+1)})\}$ as $n\to +\infty$. The limit $(\eta(\kk), v(\kk),h(\kk))$ solves the nonlinear $\kk$-approximation system \eqref{app1} and the energy estimate \eqref{Ell} is a direct consequence of the uniform-in-$n$ estimate \eqref{Enfinal}. 
\end{proof}

\section{Local well-posedness of the compressible gravity water wave system}\label{lwp}

From Proposition \ref{uniformnn}, given $\kk>0$, we have constructed a solution $(v(\kk),h(\kk),\eta(\kk))$ to the nonlinear $\kk$-approximation system \eqref{app1}. Proposition \ref{uniformkk} gives a $\kk$-independent estimate \eqref{Ekk0} on some time interval $[0,T_0]$, which yields a strong convergence to a limit $(v,h,\eta)$ for every $t\in[0,T_0]$. This limit $(v,h,\eta)$ is a solution to the compressible gravity water wave system \eqref{wwl} with energy estimate \eqref{EEEE} in Theorem \ref{MAIN} if we set $\kk\to 0+$ in \eqref{app1}. Therefore, the existence has been proved.

Let $(v^1,h^1,\eta^1),(v^2,h^2,\eta^2)$ be two solutions to the compressible gravity water wave system \eqref{wwl} with the initial data $(v_0, h_0)$ and $(\hat{v}_0, \hat{h}_0)$, respectively.  Denoting their difference by $([v],[h],[\eta]):=(v^1-v^2,h^1-h^2,\eta^1-\eta^2)$ and $a^i:=(\p\eta^i)^{-1}$ with $[a]:=a^2-a^1$, then $([v],[h],[\eta])$ solves the following system:

\begin{equation}\label{gap}
\begin{cases}
\p_t[\eta]=[v]~~~&\text{ in }\Omega, \\
\p_t[v]=-\nabla_{a^1}[h]+\nabla_{[a]}h^2~~~&\text{ in }\Omega, \\
\dive_{a^1}[v]=\dive_{[a]}v^2-e'(h^2)\p_t[h]-(e'(h^1)-e'(h^2))\p_t h^2~~~&\text{ in }\Omega \\
[h]=0~~~&\text{ on }\Gamma,\\
([\eta], [v], [h])|_{t=0} = (\mathbf{0}, v_0-\hat{v}_0, h_0-\hat{h}_0). 
\end{cases}
\end{equation} 
We define the energy functional of \eqref{gap} by 
\begin{equation}\label{Egap}
[\EE]=\|[\eta]\|_2^2+\sum_{k=0}^2\|\p_t^{2-k}[v]\|_k^2+\|\p_t^{2-k}[h]\|_k^2+|(a^{1})^{3\alpha}\TP^2[\eta]_{\alpha}|_0^2.
\end{equation}
This looks very similar to \eqref{linearg}. The only essential difference is the boundary term $$\ig [\HH](a^{1})^{3\alpha} [\VV]_{\alpha}\dS,$$ where we define the Alinhac's good unknowns 
\[
\VV^i=\TP^2v^i-\TP^2\eta^i\cdot\nabla_{a^i} v^i,~~\HH^i=\TP^2h^i-\TP^2\eta^i\cdot\nabla_{a^i} h^i, 
\]and
\[
[\VV]:=\VV^1-\VV^2,~~[\HH]:=\HH^1-\HH^2.
\]
 The boundary terms then becomes
\begin{align*}
\ig [\HH](a^{1})^{3\alpha} [\VV]_{\alpha}&=-\ig\p_3[h]\TP^2\eta^2_{\beta}(a^2)^{3\beta}(a^2)^{3\alpha}[\VV]_{\alpha}\dS-\ig \p_3 h^1(\TP^2[\eta]_{\beta}(a^1)^{3\beta}+\TP^2\eta^2_{\beta}[a]^{3\beta})(a^1)^{3\alpha}[\VV]_{\alpha}\dS \\
&\lesssim -\frac{1}{2}\frac{d}{dt}\ig \p_3 h^1|(a^{1})^{3\alpha}\TP^2[\eta]_{\alpha}|_0^2\dS\\
&~~~~-\ig\p_3 h^1(a^1)^{3\gamma}\TP^2[\eta]_{\gamma}(\TP^2\eta^2_{\beta}[a]^{\mu\beta}\p_\mu v^1_{\alpha}-\TP^2\eta^2_{\beta}(a^2)^{\mu\beta}\p_\mu[v]_{\alpha})(a^1)^{3\alpha}\dS \\
&~~~~-\ig \p_3 h^1(\TP^2[\eta]_{\beta}(a^1)^{3\beta}+\TP^2\eta^2_{\beta}[a]^{3\beta})(a^1)^{3\alpha}[\VV]_{\alpha}\dS\\
&\lesssim-\frac{c_0}{2}\frac{d}{dt}\ig |(a^{1})^{3\alpha}\TP^2[\eta]_{\alpha}|_0^2\dS+C(M_0)P([\EE](t)),
\end{align*}
where $M_0$ is the constant defined in Theorem \ref{MAIN}.
Here in the second step we use the precise formula of $[\VV]$, and in the third step we apply the physical sign condition for $h^1$. Therefore we have
\[
\sup_{t\in[0,T_0]}[\EE](t)\leq P(\|v_0-\hat{v}_0\|_2, \|h_0-\hat{h}_0\|_2)+\int_0^{T_0}C(M_0)P([\EE](t))\dt,
\] 
which implies \eqref{EEEE2}. Also, when $v_0=\hat{v}_0$ and $h_0=\hat{h}_0$, we know $[\EE](t)=0$ for all $t\in[0,T_0]$ which gives the uniqueness of the solution to the compressible gravity water wave system \eqref{wwl}.

\end{document}